\documentclass[11pt, twoside]{article}
\usepackage{mathrsfs}
\usepackage{amssymb}
\usepackage{amsmath}
\usepackage{mathrsfs}
\usepackage{amsthm}
\usepackage{amsfonts}
\usepackage{color}
\usepackage{latexsym}
\usepackage{txfonts}
\usepackage{indentfirst}
\usepackage{soul}

\usepackage{anysize}

\allowdisplaybreaks
\def\loc{\mathop\mathrm{\,loc\,}}

\newtheorem{theorem}{Theorem}[section]
\newtheorem{lemma}[theorem]{Lemma}

\newtheorem{proposition}[theorem]{Proposition}
\newtheorem{example}[theorem]{Example}
\theoremstyle{definition}
\newtheorem{remark}[theorem]{Remark}
\newtheorem{definition}[theorem]{Definition}

\numberwithin{equation}{section}

\newcounter{Theorem}
\renewcommand{\theTheorem}{\Alph{Theorem}}

\newenvironment{thm}[1][]{
  \refstepcounter{Theorem}
  \par\addvspace{12pt}
  \noindent\textbf{Theorem~\theTheorem.}
  \if\relax\detokenize{#1}\relax\else\ \textit{(#1)}\fi
  \ignorespaces
  \itshape
}{\par\addvspace{12pt}}

\begin{document}

\title{\bf\Large
Sharp Weighted
Cohen--Dahmen--Daubechies--DeVore Inequality
with Applications to (Weighted) Critical Sobolev Spaces,
Gagliardo--Nirenberg Inequalities,
and Muckenhoupt Weights
\footnotetext{\hspace{-0.35cm} 2020
{\it Mathematics Subject Classification}.
Primary 26D10; Secondary 46E35, 42B25, 42B35, 35A23.\endgraf
{\it Key words and phrases.}
Cohen--Dahmen--Daubechies--DeVore inequality,
Gagliardo--Nirenberg inequality,
Muckenhoupt weight,
Brezis--Seeger--Van Schaftingen--Yung formula,
real interpolation, ball Banach function space.
\endgraf
This project is partially supported by the National
Key Research and Development Program of China
(Grant No.\ 2020YFA0712900),
the National Natural Science Foundation of
China (Grant Nos.\ 12431006 and 12371093),
the Fundamental Research Funds
for the Central Universities (Grant No.\ 2233300008).
}}
\date{}
\author{Yinqin Li,
Dachun Yang\footnote{Corresponding author,
E-mail: \texttt{dcyang@bnu.edu.cn}/{\color{red}\today}/Final version.},
\ Wen Yuan, Yangyang Zhang and
Yirui Zhao}

\maketitle

\vspace{-0.7cm}

\begin{center}
\begin{minipage}{13cm}
{\small {\bf Abstract}\quad
In this article, we establish a quantitative weighted variant
of a far-reaching inequality obtained
by A. Cohen, W. Dahmen, I. Daubechies, and R. DeVore in 2003,
whose dependence on the $A_p$-weight constant for any $p\in[1,\infty)$ is sharp.
As applications, we obtain the almost characterization of the critical
weighted Sobolev space in terms of wavelets,
a sharp real interpolation between
this weighted Sobolev space and weighted Besov spaces,
and three new Gagliardo--Nirenberg type inequalities in
the framework of ball Banach function
spaces. Moreover, we apply this sharp weighted
inequality to extend the famous Brezis--Seeger--Van Schaftingen--Yung
formula in ball Banach function spaces,
which gives an affirmative answer to the question
in page 29 of [Calc. Var. Partial Differential Equations 62 (2023),
Paper No. 234]. Notably, we further establish
two new characterizations
of Muckenhoupt weights related to
the inequality of Cohen et al.\ and the formula of Brezis et al.
The most novelty of this article exists in
applying and further developing
the good cube method introduced by
Cohen et al.\ to trace the sharp dependences on weight constants.
}
\end{minipage}
\end{center}
	
\vspace{0.2cm}
	
\tableofcontents

\vspace{0.2cm}

\section{Introduction}

Throughout this article, we work in $\mathbb R^n$ and,
unless necessary, we will not explicitly specify this underlying space.

As is widely known, many classical function spaces
can be characterized in terms of wavelets
as unconditional bases, which has attracted
a lot of interest for applications in
harmonic analysis, partial differential equations,
data compression, and statistical estimation
(see, for example, \cite{d93,dvdd98,kl95,m92,mc97,w97}).
However, certain function spaces with critical exponents,
such as the Lebesgue space $L^1$,
the Sobolev space $W^{1,1}$, and the bounded variation space ${\rm BV}$,
are known to possess no unconditional basis of any type.

As an alternative approach, Cohen et al. \cite{cdpx99}
cleverly established an \emph{almost characterization} of ${\rm BV}$
(see also Theorem \ref{p2143} and Remark \ref{rem-cdpx} for details),
which yielded optimal compression
or estimation algorithms based on wavelet thresholding.
Later, by incorporating weights
of the form $|Q|^\beta$ for any cube $Q$ and some
$\beta\in\mathbb{R}$ into these weak type norms,
Cohen, Dahmen, Daubechies, and DeVore
\cite[Theorems 3.1 and 4.1]{cddd03}
established a far-reaching weak type inequality, which we simply
refer to as the CDDD \emph{inequality} and
which connects level sets of renormalized averaged moduli of
continuity to gradients.
Precisely, let $\mathcal{D}^{\bf0}:=\{2^j[m+[0,1)^n]:j\in\mathbb{Z},\
m\in\mathbb{Z}^n\}$ and,
for any cube $Q$
and any locally integrable function $f$,
$\omega_Q(f):=|Q|^{-1-\frac1n}
\int_{Q}\int_{Q}|f(x)-f(y)|\,dx\,dy$
denotes the \emph{renormalized averaged modulus of continuity}
of $f$.
Cohen et al. \cite[Theorems 3.1 and 4.1]{cddd03} showed
the following CDDD inequality.

\begin{thm}\label{cddd-03}
If $\beta\in (-\infty,1-\frac1n)
\cup(1,\infty)$, then there exists a positive
constant $C$ such that, for any $f\in W^{1,1}$,
\begin{align}\label{eq-cddd-e1}
\sup_{\lambda\in(0,\infty)}\lambda
\sum_{\{Q\in\mathcal{D}^{\bf0}:\omega_Q(f)>\lambda|Q|^\beta\}}
|Q|^{\beta}\le C \left\|\,|\nabla f|\,\right\|_{L^1},
\end{align}
here and thereafter,
$\nabla f$ denotes the gradient (in the sense of distributions) of $f$.
\end{thm}

This CDDD inequality \eqref{eq-cddd-e1} enabled one to obtain the
almost characterization of the bounded variation space BV
via wavelet coefficients on more general weak type spaces,
to investigate the sharp real interpolation
between bounded variation and Besov spaces,
and to establish Gagliardo--Nirenberg type inequalities (see,
for example, \cite{bsvy,cddd03,cdpx99}).
 Nowadays, the CDDD inequality \eqref{eq-cddd-e1}
has already found many significant applications;
we refer to \cite{bm18,bsvy,c24,dht20,ds23} for
the applications in harmonic analysis and to
\cite{cddd03,cdh00,cdkp01,w03} for
the applications in nonlinear approximation.

It is natural to ask whether
\emph{the {\rm CDDD} inequality can be extended
beyond weights of the form $|Q|^\beta$ in Theorem \ref{cddd-03}}.
In this article,
we answer this question affirmatively
by establishing a quantitative weighted CDDD inequality
involving Muckenhoupt $A_p$-weight class
for any $p\in[1,\infty)$ (see Theorem \ref{thm-cddd}).
Moreover, the dependences
on the weight constants are \emph{sharp}
(see Remark \ref{rem12}).
As applications, we establish the almost characterization of the critical
weighted Sobolev space in terms of wavelets (see Theorem \ref{p2143}),
a real interpolation between
this weighted Sobolev space and weighted Besov spaces with the sharp
assumption on the smoothness exponent
(see Theorem \ref{t953} and Remark \ref{rem38}),
and a new Gagliardo--Nirenberg type inequality in the framework
of ball Banach function
spaces (which include various important function spaces); see Theorem \ref{c1016}.
To our surprise, we find that the sharp weighted CDDD inequality
can also be applied to extend the famous Brezis--Seeger--Van Schaftingen--Yung
(for short, BSVY) formula obtained by Brezis et al. \cite{bsvy.arxiv}
to the framework of ball Banach function spaces,
which gives an affirmative answer to the question
in Zhu et al. \cite[Remark 3.17(iv)]{zyy23-1};
see Theorem \ref{xgamma<0} and Remark \ref{rem411}.
Notably, we further utilize
both the weighted CDDD inequality
and the upper estimate of
the weighted BSVY formula to
establish two new characterizations
of Muckenhoupt weights (see Theorem \ref{cddd}).

We first present the sharp weighted CDDD inequality.
In what follows, let $L^1_{\loc}$
denote the set of all locally integrable functions
on $\mathbb{R}^n$ and
$\mathcal{D}$ be a dyadic grid on $\mathbb{R}^n$
(see Definition \ref{df-dyadic} for its
definition).
For any $\lambda\in(0,\infty)$, $b\in\mathbb{R}$, and
$f\in L^1_{\loc}$, let
$\mathcal{D}_{\lambda,b}[f]
:=\{Q\in\mathcal{D}:
\omega_Q(f)>\lambda|Q|^{b}\}.$
Let $p\in [1,\infty)$
and $\omega\in A_p$ [that is, the Muckenhoupt $A_p$-weight;
see Definition \ref{1557}(i) for its definition]. We
use $\dot{W}^{1,p}_\omega$ to denote the
homogeneous weighted Sobolev space
[see Definition \ref{1557}(iii) for its definition] and
let
\begin{align*}
\Omega_{p,n}:=
\begin{cases}
{\displaystyle\left(-\infty,1-\frac1n\right)\cup(1,\infty)}
&\text{if}\ p=1,\\
{\displaystyle\mathbb{R}\setminus\left\{\frac1p\right\}}
&\text{if}\ p\in(1,\infty).
\end{cases}
\end{align*}
Then we have the following sharp
weighted CDDD inequality for general
dyadic grids.

\begin{theorem}\label{thm-cddd}
Let $p\in[1,\infty)$ and $\beta\in\Omega_{p,n}$.
\begin{enumerate}
  \item[{\rm(i)}] If $\upsilon\in A_1$,
  then there exists a positive constant $C$,
  depending only on $p$, $n$, and $\beta$, such that,
  for any dyadic grid $\mathcal{D}$ and
  any $f\in \dot{W}^{1,p}_\upsilon$,
  \begin{align}\label{thm-cddd-e1}
  \sup_{\lambda\in(0,\infty)}\lambda^p
  \sum_{Q\in\mathcal{D}_{\lambda,\beta+1-\frac1p}[f]}
  |Q|^{\beta p-1}\upsilon(Q)\le C
  [\upsilon]_{A_1}\left\|\,|\nabla f|\,\right\|_{L^p_{\upsilon}}^p.
  \end{align}

  \item[{\rm(ii)}] If $\upsilon\in A_p$,
  then there exists a positive constant $C$,
  depending only on $p$, $n$, and $\beta$, such that,
  for any dyadic grid $\mathcal{D}$ and
  any $f\in \dot{W}^{1,p}_\upsilon$,
  \begin{align}\label{thm-cddd-e2}
  \sup_{\lambda\in(0,\infty)}\lambda^p
  \sum_{Q\in\mathcal{D}_{\lambda,\beta+1-\frac1p}[f]}
  |Q|^{\beta p-1}\upsilon(Q)\le C
  [\upsilon]_{A_p}^{\alpha_{p,\beta}}
  \left\|\,|\nabla f|\,\right\|_{L^p_{\upsilon}}^p,
  \end{align}
  where
  \begin{align*}
  \alpha_{p,\beta}:=
  \begin{cases}
  {\displaystyle\frac{p}{p-1}}&{\displaystyle\text{if}\ p\in(1,\infty)\ \text{and}\
  \beta\in\left[\frac1p-1,\frac1p\right)},\\
  1&\text{otherwise}.
  \end{cases}
  \end{align*}
\end{enumerate}
\end{theorem}

\begin{remark}\label{rem12}
\begin{enumerate}
  \item[{\rm(i)}] If $p=1$,
$\upsilon\equiv1$, and $\mathcal{D}:=\mathcal{D}^{\bf0}$,
then Theorem \ref{thm-cddd} in this case exactly reduces
to the CDDD inequality
\eqref{eq-cddd-e1}
obtained by Cohen et al. in \cite{cddd03}.
  \item[{\rm(ii)}] The dependences on
  weight constants in Theorem \ref{thm-cddd} are sharp (see Proposition
  \ref{sharp}).
  To be precise, (i) and (ii) of Proposition \ref{sharp}
  imply that the exponent $1$ of $[\upsilon]_{A_1}$
in \eqref{thm-cddd-e1} can not be replaced by a smaller real number
and, if $p=1$ or $p\in(1,\infty)$ and $\beta\in(-\infty,\frac1p-1]\cup(\frac1p,\infty)$,
then the exponent $\alpha_{p,\beta}$
of $[\upsilon]_{A_p}$ in \eqref{thm-cddd-e2}
can not be replaced by a smaller real number;
Proposition \ref{sharp}(iii) shows that, if $p\in(1,\infty)$
and $\beta_0\in(\frac1p-1,\frac1p)$ and if assume
that \eqref{thm-cddd-e2} holds uniformly
with respect to $\beta\in(\frac1p-1,\beta_0)$
(this assumption is reasonable; see Proposition \ref{es-dyadic2}),
then the exponent $\alpha_{p,\beta}$
of $[\upsilon]_{A_p}$ in \eqref{thm-cddd-e2}
can not be replaced by a smaller real number.
Therefore, in these senses,
the dependences on weighted constants in Theorem \ref{thm-cddd}
are sharp.
\end{enumerate}
\end{remark}

To derive Theorem \ref{thm-cddd}, we apply and further develop
the method introduced by Cohen et al.\ in \cite{cddd03},
which we simply refer to as the \emph{good cube method}.
Precisely, Cohen et al. established the CDDD inequality \eqref{eq-cddd-e1}
by considering three cases for the dimension $n$ and
the parameter $\beta$.
The most challenging case occurs when
$n\in\mathbb{N}\cap[2,\infty)$ and
$\beta \in [0,1 - \frac{1}{n})$, where
they ingeniously partitioned the dyadic cubes in
$\mathbb{R}^n$ into good and bad cubes (see \cite[Definition 4.2]{cddd03} for their definitions)
and combined subtle geometric properties of good cubes with
both the coarea formula and the isoperimetric inequality.
In this article, we generalize their good cube method
to the weighted setting (see Definition \ref{df-good} for our
corresponding definitions of good and bad cubes)
to establish Theorem \ref{thm-cddd}
for this difficult case for $n$ and $\beta$
(see Proposition \ref{cddd-suff} and its proof).
On the other hand, for the case $n=1$ and $\beta\in(-\infty,0)$,
Cohen et al.\ employed a minimal cube technique (see \cite[pp.\,249--250]{cddd03}).
However, this approach seems to
be insufficient for obtaining the
sharp dependence on weight constants
because the transformation of larger cubes into minimal cubes
requires the doubling property of weights and hence
yields a superfluous weight constant.
To overcome this limitation, we modify the good cube method
via classifying cubes based on the magnitude of integral means
(see Proposition \ref{es-dyadic1} and its proof).
Furthermore, through a flexible combination of power weights and carefully
constructed functions,
we prove the sharpness of the dependences
on weight constants in Theorem \ref{thm-cddd} (see Proposition \ref{sharp} and
its proof).

Using the above sharp weighted CDDD inequalities
\eqref{thm-cddd-e1} and \eqref{thm-cddd-e2},
we give the following three applications.

\textbf{(I) Sharp real interpolation and Gagliardo--Nirenberg inequality.}
Notice that, when a function space has an equivalent
characterization via
some sequence space of wavelet coefficients,
the study of the real interpolation between
such space and Besov spaces reduces to
investigating the real interpolation of corresponding sequence spaces.
Although the critical Sobolev space $W^{1,1}$
and the bounded variation space ${\rm BV}$ lack
such characterizations,
Cohen et al. \cite{cddd03} still used
their almost wavelet characterization
to establish the corresponding real interpolation result (see also
\cite{c24}).
Applying Theorem \ref{thm-cddd}, we are able to
prove the almost wavelet characterization
(see Theorem \ref{p2143})
of the critical weighted Sobolev space $W^{1,1}_{\upsilon}$
[see Definition \ref{1557}(vi)],
which further yields the
following real interpolation result
between this weighted Sobolev space and the weighted Besov spaces
$B^{s,\upsilon}_{p,p}$ (see Section \ref{s-ip}
for the precise definitions of all the symbols).

\begin{theorem}\label{t953}
Let $\upsilon\in A_1$, $p\in(1,\infty)$, and $s\in\mathbb{R}$.
If $s\in(-\infty,\frac1p)\cup(1,\infty)$,
$\theta\in(0,1)$,
$\sigma\in\mathbb{R}$, and $q\in(1,\infty)$ satisfying
$\sigma=1-\theta+\theta s$ and $\frac1q=1-\theta+\frac{\theta}{p}$,
then $(W^{1,1}_{\upsilon},B^{s,\upsilon}_{p,p})_{\theta,q}
=B^{\sigma,\upsilon}_{q,q}.$
\end{theorem}

\begin{remark}\label{rem38}
\begin{enumerate}
  \item[{\rm(i)}] In Theorem \ref{t953},
the assumption $(-\infty,\frac{1}{p})\cup(1,\infty)$
on the smoothness exponent $s$ is sharp.
Indeed, let $\upsilon\equiv1$.
Then, in this case, from \cite[Theorem 1]{bm18}, we deduce that
Theorem \ref{t953} fails if $s\in[\frac{1}{p},1)$.
  
  \item[{\rm(ii)}] Notice that the proof of Theorem \ref{t953}
  strongly depends on the case $p=1$ of the weighted CDDD inequality 
  (Theorem \ref{thm-cddd}) and hence Theorem \ref{t953} needs
  the assumption $\upsilon\in A_1$, while the assumption $\upsilon\in A_p$
  with $p\in (1,\infty)$ is not enough. 
\end{enumerate}
\end{remark}

As a corollary of both Theorem \ref{t953} and the
method of extrapolation, we further obtain
a Gagliardo--Nirenberg inequality in the framework
of ball Banach function spaces (see Theorem \ref{c1016}),
whose assumption on the smoothness exponent
is also sharp [see Remark \ref{rem52}(i)].
Recall that the concept of ball
(quasi-)Banach function spaces was
introduced by Sawano et al. \cite{shyy2017}
to unify the study of several different important function spaces,
including weighted and variable Lebesgue spaces,
(Bourgain--)Morrey type spaces, local and global generalized
Herz spaces, and Orlicz(-slice) spaces;
see Section \ref{S5} for their exact definitions
and histories. Since its introduction and its wide generality, ball
(quasi-)Banach function spaces have recently attracted a lot of attention
and yielded many applications; we refer the reader to
\cite{hcy2021,is2017,tyyz2021,wyy2020,zwyy2021}
for the boundedness of operators and to
\cite{lyh2320,wyyz2021,zhyy2022} for the real-variable
theory of function spaces.

\textbf{(II) Brezis--Seeger--Van Schaftingen--Yung
formula.}
Notably, Theorem \ref{thm-cddd} has a deep connection with the
Brezis--Seeger--Van Schaftingen--Yung formula.
Recall that Brezis et al. \cite{bsvy.arxiv} recently
established a one-parameter family of equivalent formula of gradients via the size of
level sets of suitable difference quotients (see \cite{bsvy.arxiv,bsvy}),
which are simply called the BSVY \emph{formula}. Precisely speaking,
in what follows, for any $\gamma\in\mathbb{R}$,
$q,\lambda\in(0,\infty)$, and $f\in L^1_{\loc}$, let
\begin{align}\label{Elambda}
E_{\lambda,\frac{\gamma}{q}}[f]:=\left\{(x,y)\in\mathbb{R}^n\times\mathbb{R}^n:
x\neq y,\ \frac{|f(x)-f(y)|}{|x-y|^{1+\frac{\gamma}{q}}}>\lambda\right\}.
\end{align}
In \cite{bsvy.arxiv}, Brezis et al. proved that,
if $p\in(1,\infty)$ and $\gamma\in \mathbb{R}\setminus\{0\}$ or
$p=1$ and $\gamma\in(-\infty,-1)\cup(0,\infty)$,
then, for any $f\in \dot{W}^{1,p}$ (the homogeneous Sobolev space),
\begin{align}\label{bsvy-e1}
\sup_{\lambda\in(0,\infty)}\lambda
\left[\int_{\mathbb{R}^n}\int_{\mathbb{R}^n}
{\bf 1}_{E_{\lambda,\frac{\gamma}{p}}[f]}(x,y)
|x-y|^{\gamma-n}\,dy\,dx\right]^{\frac1p}
\sim\left\|\,\left|\nabla f\right|\,\right\|_{L^p}.
\end{align}
This formula gives a brand new answer to
the question of fundamental importance: how to
represent gradients via difference quotients;
see \cite{bbm2001,b2002,bsvy,bvy2021,dssvy}
for more related studies. Moreover,
using the BSVY formula, Brezis et al.
gave new characterizations of both homogeneous Sobolev spaces
and bounded variation spaces in \cite{bsvy.arxiv} and
established critical fractional Sobolev and fractional
Gagliardo--Nirenberg type inequalities in \cite{bsvy,bvy2021}
(see also Section \ref{sec-gn}).
For more developments of BSVY formulae and their applications,
we refer to
\cite{dgpyyz2022,dlyyz2022,
dm.arXiv1,dm.arXiv,gh2023,p22}.

Very recently, Zhu et al.\ \cite{zyy23-1}
extended the BSVY formula to
a framework of function spaces, called the ball
Banach function space.
However, that result is not complete. Indeed,
Zhu et al.\ in \cite[Remark 3.17(iv)]{zyy23-1} explicitly pointed out
that, in some function spaces with critical exponents,
it was \emph{unknown} whether the upper
estimate of the BSVY formula in ball Banach function spaces holds
when $\gamma\in(-\infty,0)$ and $n\in\mathbb{N}\cap[2,\infty)$.

In this article, applying Theorem \ref{thm-cddd},
we give an affirmative answer
to the question
in \cite[Remark 3.17(iv)]{zyy23-1}
via establishing a BSVY formula in ball Banach function spaces
with sharp parameters (see Theorem \ref{xgamma<0} and Remark \ref{rem411}).
To be precise, by
leveraging the approximation properties of shifted
dyadic grids in $\mathbb{R}^n$ [see Lemma \ref{2115}(ii)]
and the famous telescope method (see Lemma \ref{Poin}),
we establish two key dominations for weak norms of
the level set \eqref{Elambda} (see Propositions \ref{point} and \ref{point2})
and hence connect the upper estimate of the BSVY formula
\eqref{bsvy-e1} to the weighted CDDD inequality \eqref{thm-cddd-e1}.
This, together with Theorem \ref{thm-cddd} and
the standard extrapolation in ball Banach function spaces,
further yields the desired BSVY formula.
Moreover, as applications of Theorem \ref{xgamma<0}, we
obtain two Gagliardo--Nirenberg type inequalities
in ball Banach function spaces, which also
improve the corresponding results in \cite{zyy23-1}
by removing the restriction $n=1$;
see Theorems \ref{2256} and \ref{2255} and Remark
\ref{rm48}. All these results are
of quite wide generality and are applied
to various specific function spaces,
all of which improve the corresponding results
obtained in \cite{zyy23-1} by removing the restriction $n=1$
and making the range of the exponent $q$ sharp;
see Section \ref{S5} for the details.

\textbf{(III) New Characterizations of Muckenhoupt Weights.}
As a further application of Theorems \ref{thm-cddd} and \ref{xgamma<0},
we obtain two new characterizations of the Muckenhoupt weight class
via the weighted CDDD inequality and
the upper estimate of the weighted BSVY formula as follows,
which reveals deep connections among these three objects
and gives a new perspective
of the Muckenhoupt weight class via the smoothness
structure of functions.
For any $p\in [1,\infty)$ and any nonnegative locally
integrable function $\upsilon$ on
$\mathbb{R}^n$,
$\dot{Y}^{1,p}_{\upsilon}$
denotes the homogeneous
weighted Sobolev space [see \eqref{df-Ynp}
for its definition] and
$\{\mathcal{D}^\alpha\}_{\alpha\in\{0,\frac13,\frac23\}^n}$
denotes the set of shifted dyadic grids (see Example \ref{e-dya}).

\begin{theorem}\label{cddd}
If $p\in[1,\infty)$
and $\upsilon\in L^1_{\loc}$ is nonnegative,
then the following statements
are mutually equivalent.
\begin{enumerate}
   \item[{\rm(i)}] $\upsilon\in A_p$.
  \item[{\rm(ii)}] There exist $\beta\in\mathbb{R}\setminus\{\frac1p\}$
  and $C\in(0,\infty)$ such that,
  for any $\alpha\in\{0,\frac13,\frac23\}^n$ and
  $f\in \dot{Y}^{1,p}_{\upsilon}\cap L^1_{\loc}$,
  \begin{align*}
  \sup_{\lambda\in(0,\infty)}\lambda^p
  \sum_{Q\in\mathcal{D}^\alpha_{\lambda,\beta+1-\frac1p}[f]}
  |Q|^{\beta p-1}\upsilon(Q)\le
  C\|f\|_{\dot{Y}^{1,p}_{\upsilon}}^p.
  \end{align*}
  \item[{\rm(iii)}] There exist $\gamma\in\mathbb{R}$,
  $q\in(0,\infty)$, and $C\in(0,\infty)$ such that,
  for any $f\in \dot{Y}^{1,p}_{\upsilon}\cap
  L^1_{\loc}$,
\begin{align}\label{cddd-necee11}
&\sup_{\lambda\in(0,\infty)}
\lambda^p\int_{\mathbb{R}^n}
\left[\int_{\mathbb{R}^n}{\bf 1}_{E_{\lambda,
\frac{\gamma}{q}}[f]}(x,y)|x-y|^{\gamma-n}\,dy
\right]^{\frac pq}\upsilon(x)\,dx
\le C\|f\|_{\dot{Y}^{1,p}_{\upsilon}}^p.
\end{align}
\end{enumerate}
\end{theorem}

Moreover, Theorem \ref{thm-cddd} surely have
more applications and,
in particular,
has been used
to establish another new characterization
of Muckenhoupt weights and a sharp representation formula of gradients
in terms of oscillations in
\cite{zlyyz24} and
to study the BSVY formula for higher-order gradients
in ball Banach function spaces in \cite{hlyyz25}.

The organization of the remainder of this article is as follows.
In Section \ref{section1}, we aim to prove Theorem \ref{thm-cddd}
and its sharpness.
The main target of Section \ref{s-ip} is to
establish an almost characterization of
the critical weighted Sobolev space
in terms of Daubechies wavelets (see Theorem \ref{p2143});
as an application, we further show
the sharp real interpolation between weighted
Sobolev and Besov spaces (that is, Theorem \ref{t953}).
Later, we apply Theorem \ref{thm-cddd}
to establish the BSVY formula in ball Banach function spaces
in Section \ref{section2} (see Theorem \ref{xgamma<0}).
Next, we show Theorem \ref{cddd}
(that is, two new characterizations of Muckenhoupt weights)
in Section \ref{sec-6}.
In Section \ref{sec-gn}, we further use Theorems \ref{t953}
and \ref{xgamma<0} to establish three Gagliardo--Nirenberg type inequalities
in the framework of ball Banach function spaces,
which are sharp or improve the corresponding known results
(see Theorems \ref{c1016}, \ref{2256}, and \ref{2255} and
Remarks \ref{rem52} and \ref{rm48}).
Finally, in Section \ref{S5},
we apply Theorems \ref{xgamma<0}, \ref{c1016},
\ref{2256}, and \ref{2255}
to obtain various Gagliardo--Nirenberg inequalities and
BSVY formulae in specific function spaces.

We end this introduction by making some conventions on symbols.
We always let $\mathbb{N}:=\{1,2,\ldots\}$ and
$\mathbb{Z}_+:=\mathbb{N}\cup\{0\}$.
For any $q\in[1,\infty]$,
its \emph{conjugate index} is denoted by $q'$,
that is, $\frac{1}{q}+\frac{1}{q'}=1$.
If $E$ is a subset of $\mathbb{R}^n$, we denote by
$E^\complement$ the set $\mathbb{R}^n\setminus E$,
by $|E|$ the $n$-dimensional Lebesgue measure of $E$,
and by $\partial E$ the boundary of $E$.
Moreover, we use $\mathbf{0}$ to denote the
\emph{origin} of $\mathbb{R}^n$.
For any $x\in\mathbb{R}^n$ and $r\in(0,\infty)$,
let $B(x,r):=\{y\in\mathbb{R}^n:|x-y|<r\}$
be a ball in $\mathbb{R}^n$.
The symbol $\mathcal{Q}$ denotes the set of all cubes
with edges parallel to the coordinate axes.
For any $\lambda\in(0,\infty)$
and any ball $B:=B(x_B,r_B)$ in
$\mathbb{R}^n$ with center $x_B\in\mathbb{R}^n$
and radius $r_B\in(0,\infty)$, let $\lambda B:=B(x_B,\lambda r_B)$;
for any cube $Q\in\mathcal{Q}$,
$l(Q)$ means its edge length and
$\lambda Q$ means the cube with the same
center as $Q$ and $\lambda$ times its edge length.
In addition, we always denote by $C$ a
\emph{positive constant} which is independent
of the main parameters involved, but may vary from line to line.
We use $C_{(\alpha,\dots)}$ to denote a positive constant depending
on the indicated parameters $\alpha,\, \dots$.
The symbol $f\lesssim g$ means $f\le Cg$
and, if $f\lesssim g\lesssim f$, then we write $f\sim g$;
moreover, if $f\le Cg$ and $g=h$ or $g\le h$,
we then write $f\lesssim g=h$ or $f\lesssim g\le h$.
Throughout this article, we denote by
$C_{\mathrm{c}}$
the set of all continuous functions with compact support,
by $C_{\rm c}^N$, with $N\in\mathbb{N}\cup\{\infty\}$, the set of all $N$-times
differentiable functions with compact support,
and by $C^\infty$
the set of all infinitely differentiable
functions.
For any $f\in L^1_{\loc}$ and any bounded
measurable set $E\subset\mathbb{R}^n$, let
$
f(E):=\int_{E}f(x)\,dx$ and
$
f_E:=\fint_Ef(x)\,dx
:=\frac{1}{|E|}\int_Ef(x)\,dx.$
In addition, $\varepsilon\to 0^+$ means that
there exists $c_0\in(0,\infty)$ such that
$\varepsilon\in(0,c_0)$ and $\varepsilon\to0$.
Finally, in all proofs, we consistently retain the symbols
introduced in the original theorem (or related statement).

\section{Proof of Theorem \ref{thm-cddd}
(Sharp Weighted CDDD Inequality)}\label{section1}

The target of this section is to prove Theorem \ref{thm-cddd}
and its sharpness. First, we recall
the following well-known concept of dyadic grids
(see, for example, \cite{cm13,ln19}).

\begin{definition}\label{df-dyadic}
A set  $ \mathcal{D} $ of countable cubes in $\mathbb{R}^n$
is called a \emph{dyadic grid} if it satisfies that
\begin{enumerate}
\item[{\rm (i)}] if $Q\in \mathcal{D}$, then $l(Q)=2^k$ for some $k\in{\mathbb Z}$;
\item[{\rm (ii)}] if $P,Q\in\mathcal{D}$, then $P\cap Q\in \{\emptyset,P,Q\}$;
\item[{\rm (iii)}] for any $k\in{\mathbb Z}$, one has
$\mathbb{R}^n=\bigcup_{\{Q\in\mathcal{D}:l(Q)=2^k\}}Q.$
\end{enumerate}
\end{definition}

Throughout this article, for any dyadic grid $\mathcal{D}$
and any $k\in\mathbb{Z}$,
define $\mathcal{D}_k:=\{Q\in\mathcal{D}:l(Q)=2^k\}$
and, for any $k\in\mathbb{Z}_+$ and $Q\in\mathcal{D}$,
define $\mathcal{D}_k(Q):=\{P\in\mathcal{D}:P\subset Q,\ l(P)=2^{-k}l(Q)\}$.

\begin{example}\label{e-dya}
For any $\alpha\in\{0,\frac{1}{3},\frac{2}{3}\}^n$,
the \emph{shifted dyadic grid} $\mathcal{D}^\alpha$ is defined by
setting
\begin{align}\label{dyadic}
\mathcal{D}^\alpha:=\left\{2^j\left[m+[0,1)^n+(-1)^j\alpha\right]:
j\in\mathbb{Z},\ m\in\mathbb{Z}^n\right\}.
\end{align}
By definitions, one can easily find that, for any
$\alpha\in\{0,\frac13,\frac23\}^n$,
$\mathcal{D}^\alpha$
is a dyadic grid in the sense of Definition
\ref{df-dyadic}.
\end{example}

Moreover, recall the following
definitions of $A_p$-weights
as well as weighted functions spaces;
see, for instance, \cite[Section 7.1]{g2014}.
In what follows, we
denote by $\mathscr{M}$
the set of all measurable functions
on $\mathbb{R}^n$.

\begin{definition}\label{1557}
Let $p\in[1,\infty)$ and
$\upsilon$ be a nonnegative locally integrable function on $\mathbb{R}^n$.
\begin{enumerate}
\item[{\rm (i)}]
The function $\upsilon$ is called an \emph{$A_p$-weight} if,
when $p\in(1,\infty)$,
\begin{align*}
[\upsilon]_{A_p}:=
\sup_{Q\subset {\mathbb{R}^n}}
\left[\frac{1}{|Q|}\int_{Q}\upsilon(x)\,dx\right]
\left\{\frac{1}{|Q|}\int_{Q}[\upsilon(x)]^{-\frac{1}{p-1}}\,dx\right\}^{p-1}
<\infty
\end{align*}
and, when $p=1$,
\begin{align*}
[\upsilon]_{A_1}:=\sup_{Q\subset {\mathbb{R}^n}}
\left[\frac{1}{|Q|}
\int_{Q}\upsilon(x)\,dx\right]
\left\|\upsilon^{-1}\right\|_{L^\infty (Q)}
<\infty,
\end{align*}
where
the suprema are
taken over all cubes $Q\subset {\mathbb{R}^n}$.
\item[{\rm(ii)}]The \emph{weighted Lebesgue space} $L^p_\upsilon$
is defined to be the set of all $f\in\mathscr{M}$ such that
$\|f\|_{L^p_\upsilon}:=[\int_{\mathbb{R}^n}|f(x)|^p
\upsilon(x)\,dx]^{\frac{1}{p}}<\infty.$
\item[{\rm(iii)}]
The \emph{homogeneous weighted Sobolev space}
$\dot{W}_{\upsilon}^{1,p}$ is defined to be the set
of all $f\in L^1_{\rm loc}$ such that
the weak gradient $\nabla f$ of $f$ exists
and
$|\nabla f| \in L^p_\upsilon$.
Moreover, the \emph{seminorm}
$\|\cdot\|_{\dot{W}^{1,p}_{\upsilon}}$
of $\dot{W}^{1,p}_{\upsilon}$
is defined by setting,
for any $f\in \dot{W}^{1,p}_{\upsilon}$,
$\|f\|_{\dot{W}^{1,p}_{\upsilon}}
:=\|\,|\nabla f|\,\|_{L^p_\upsilon}.$
\item[{\rm(iv)}]
The \emph{inhomogeneous weighted Sobolev space}
$W_{\upsilon}^{1,p}$ is defined to be the set
of all tempered distributions
$f$ such that $f\in L^1_{\upsilon}$,
the weak gradient $\nabla f$ of $f$ exists,
and $|\nabla f| \in L^p_\upsilon$.
Moreover, the \emph{norm}
$\|\cdot\|_{W^{1,p}_{\upsilon}}$
of $W^{1,p}_{\upsilon}$
is defined by setting,
for any $f\in W^{1,p}_{\upsilon}$,
$\|f\|_{W^{1,p}_{\upsilon}}
:=\|f\|_{L^1_\upsilon}+\|\,|\nabla f|\,\|_{L^p_\upsilon}.$
\end{enumerate}
\end{definition}

In this article, we frequently need the following properties
of Muckenhoupt $A_p$-weights;
see, for instance, \cite[Sections 7.1 and 7.2]{g2014}
and \cite[Lemma 3.1]{c25}.
In what follows, $M$ denotes the
\emph{Hardy--Littlewood maximal operator} which is defined by setting,
for any $f\in L^1_{\loc}$ and
$x\in\mathbb{R}^n$,
$M(f)(x):=\sup_{Q\ni x}\frac{1}{|Q|}\int_Q|f(y)|\,dy,$
where the supremum is taken over all
cubes $Q\in\mathcal{Q}$ containing $x$.

\begin{lemma}\label{ApProperty}
Let $p\in[1,\infty)$ and $\upsilon\in A_p$.
\begin{enumerate}
\item[\textup{(i)}]
If $p=1$, then, for almost every $x\in\mathbb{R}^n$,
$M(\upsilon)(x)\leq[\upsilon]_{A_1}\upsilon(x)$.
\item[\textup{(ii)}] For any cube $Q\in\mathcal{Q}$
and any measurable set $S\subset Q$,
$\upsilon(Q)\le[\upsilon]_{A_p}
\left(|Q|/|S|\right)^p\upsilon(S)$.
\item[\textup{(iii)}]
If $\omega$ is a nonnegative locally integrable
function on $\mathbb{R}^n$, then $\omega\in A_p$
if and only if
\begin{align*}
[\omega]^*_{A_p}:
=\sup_{Q\subset {\mathbb{R}^n} }
\sup_{\|f{\bf 1}_{Q}\|_{L^p_{\omega}}
\in (0,\infty)}
\frac{[\frac{1}{|Q|}\int_{Q}|f(x)|\,dx]^p}{\frac{1}{\omega(Q)}
\int_{Q}|f(x)|^p\omega(x)\,dx}<\infty,
\end{align*}
where the first supremum is taken over all cubes $Q\subset
{\mathbb{R}^n}$
and the second supremum is taken
over all $f\in L^{1}_{\rm loc}$ such that
$\|f{\bf 1}_{Q}\|_{L^p_{\omega}}
\in (0,\infty).$ Moreover, for any $\omega\in A_p$,
$[\omega]^*_{A_p}=[\omega]_{A_p}$.
\item[\textup{(iv)}]
If $p\in(1,\infty)$ and $\mu:=\upsilon^{1-p'}$,
then $\mu\in A_{p'}$ and, for any $q\in[1,p)$,
$[\mu]_{A_{p'}}^{p-1}=[\upsilon]_{A_p}
\leq[\upsilon]_{A_q}.$
\item[\textup{(v)}]
If $p\in(1,\infty)$, then $M$ is bounded on $L^p_\upsilon$
and there exists a positive constant $C_{(n,p)}$,
depending only on both $n$ and $p$, such that
\begin{align}\label{lem-apwight-e1}
\|M\|_{L^p_\upsilon
\to L^p_\upsilon}
\leq C_{(n,p)}[\upsilon]_{A_p}^{\frac{1}{p-1}},
\end{align}
where
$\|M\|_{L^p_\upsilon\to L^p_\upsilon}$
denotes the operator norm of $M$ on $L^p_\upsilon$;
furthermore, if $\upsilon\in A_1$, then
$[\upsilon]_{A_p}^{\frac{1}{p-1}}$ in \eqref{lem-apwight-e1}
can be replaced by
$[\upsilon]_{A_1}^{\frac1p}$.
\item[{\rm(vi)}] If $p\in(1,\infty)$, then
$A_p=\bigcup_{q\in(1,p)}A_q.$
\end{enumerate}
\end{lemma}

To show Theorem \ref{thm-cddd}, we
consider three different cases
for $p$, $n$, and $\beta$. We first
deal with the case $p\in(1,\infty)$ and $\beta\in[\frac{1}{p}-1,\frac1p)$
in the following proposition.

\begin{proposition}\label{es-dyadic2}
If $p\in(1,\infty)$ and $\upsilon\in A_p$, then there exists a positive constant
$C$, depending only on $n$ and $p$,
such that, for any $\beta\in[\frac1p-1,\frac1p)$,
any dyadic grid $\mathcal{D}$,
and any $f\in\dot{W}^{1,p}_\upsilon$,
\begin{align}\label{es-dyadic2-e1}
\sup_{\lambda\in(0,\infty)}\lambda^p
\sum_{Q\in\mathcal{D}_{\lambda,\beta+1-\frac1p}[f]}
|Q|^{\beta p-1}\upsilon(Q)\le C
\frac{1}{\frac1p-\beta}[\upsilon]_{A_p}^{\frac{p}{p-1}}
\|\,|\nabla f|\,\|_{L^p_\upsilon}^p.
\end{align}
Furthermore, if $\upsilon\in A_1$, then
$[\upsilon]_{A_p}^{\frac{p}{p-1}}$ in \eqref{es-dyadic2-e1}
can be replaced by $[\upsilon]_{A_1}$.
\end{proposition}

To show this proposition,
we need the following sparseness of
cubes in the set
$\mathcal{D}_{\lambda,\beta+1-\frac1p}[f]$,
which is a consequence of the properties
of dyadic grids.

\begin{lemma}\label{lem-sd}
If $p\in[1,\infty)$,
$\beta\in\mathbb{R}\setminus\{\frac1p\}$,
$\lambda,r\in(0,\infty)$, and $f\in C^\infty$
satisfies $|\nabla f|\in C_{\rm c}$,
then, for any dyadic grid $\mathcal{D}$ and any
$x\in\bigcup_{Q\in\mathcal{D}_{\lambda,\beta+1-\frac1p}[f]}Q$,
there exists a cube
$Q_x\in \mathcal{D}_{\lambda,\beta+1-\frac1p}[f]$
containing $x$ such that
\begin{align}\label{eq2030}
\sum_{Q\in\mathcal{D}
_{\lambda,\beta+1-\frac1p}[f]}
\left|Q\right|^{r(\beta-\frac1p)}
{\bf 1}_Q(x) \lesssim\frac{1}{|\frac1p-\beta|}
 |Q_x|^{r(\beta-\frac1p)},
\end{align}
where the implicit positive
constant depends only on
$n$, $p$, and $r$.
\end{lemma}

\begin{proof}
Fix a dyadic grid $\mathcal{D}$ and
$x\in\bigcup_{Q\in\mathcal{D}_{\lambda,\beta+1-\frac1p}[f]}Q$.
Applying \cite[(2.21)]{cddd03}, we obtain,
for any $Q\in\mathcal{Q}$,
\begin{align}\label{esti-omegae5}
\omega_Q(f)\le\sqrt{n}\int_{Q}\left|\nabla f(x)\right|\,dx.
\end{align}
Using this and the assumption $|\nabla f|\in C_{\rm c}$,
we conclude that, for any $Q\in \mathcal{D}
_{\lambda,\beta+1-\frac1p}[f]$,
\begin{align}\label{esti-omegae70}
\lambda|Q|^{\beta+1-\frac1p}
<\omega_Q(f)
\le\sqrt{n}\left\|\,\left|\nabla f\right|\,\right\|_{L^\infty}|Q|.
\end{align}

Now, we consider the following two cases for
$\beta\in \mathbb{R}\setminus \{\frac1p\}$.

\emph{Case 1: $\beta\in (-\infty, \frac{1}{p})$.} In this case,
from \eqref{esti-omegae70}, we deduce that, for any $Q\in \mathcal{D}
_{\lambda,\beta+1-\frac1p}[f]$,
$|Q|> [\sqrt{n}\lambda^{-1}
\|\,|\nabla f|\,\|_{L^\infty}]^{\frac{1}{\beta-\frac1p}}>0.$
Therefore, by Definition \ref{df-dyadic}(ii), we find that there exists a unique cube
$Q_x\in \mathcal{D}_{\lambda,\beta+1-\frac1p}[f]$ that is minimal with
respect to the set inclusion such that $x\in Q_x$; moreover,
for any $k\in\mathbb{Z}_+$,
\begin{align*}
&\sharp\left\{Q\in\mathcal{D}_{\lambda,
\beta+1-\frac1p}[f]:
Q\supset Q_x,\ |Q|=2^{kn}|Q_x|\right\}
\le 1,
\end{align*}
here and thereafter, $\sharp E$ denotes the cardinality
of the set $E$.
Applying these and the assumption
$\beta\in(-\infty,\frac1p)$ again,
we conclude that
\begin{align}\label{pointe8+}
\sum_{Q\in\mathcal{D}_{\lambda,\beta+1-\frac1p}[f]}
\left|Q\right|^{r(\beta-\frac1p)}
{\bf 1}_Q(x)
&=\sum_{k\in\mathbb{Z}_+}
\sum_{\{Q\in\mathcal{D}_{\lambda,
\beta+1-\frac1p}[f]:Q\supset Q_x,\,|Q|=2^{kn}|Q_x|\}}
\left|Q\right|^{r(\beta-\frac1p)}\notag\\
&\le\left|Q_x\right|^{r(\beta-\frac1p)}
\sum_{k\in\mathbb{Z}_+}2^{knr(\beta-\frac1p)}
\sim\frac{1}{\frac1p-\beta}\left|Q_x\right|^{r(\beta-\frac1p)},
\end{align}
where the implicit positive constant depends only on $n$,
$p$, and $r$.
This further implies that \eqref{eq2030} holds in this case.

\emph{Case 2: $\beta\in (\frac{1}{p},\infty)$.}
In this case, by \eqref{esti-omegae70}, we find that, for any
$Q\in \mathcal{D}_{\lambda,\beta+1-\frac1p}[f]$,
$
|Q|<[\sqrt{n}\lambda^{-1}
\|\,|\nabla f|\,\|_{L^\infty}
]^{\frac{1}{\beta-\frac1p}}<\infty.
$
Using this and Definition \ref{df-dyadic}(ii),
we conclude that there exists a unique cube
$Q_x\in \mathcal{D}_{\lambda,\beta+1-\frac1p}[f]$ that is
maximal with respect to the set inclusion such that
$x\in Q_x$; moreover,
for any $k\in\mathbb{Z}_+$,
\begin{align*}
&\sharp\left\{Q\in\mathcal{D}_{\lambda,
\beta+1-\frac1p}[f]:
x\in Q\subset Q_x, |Q|=2^{-kn}|Q_x|\right\}
\le 1.
\end{align*}
Applying these, the assumption
$\beta\in(\frac1p,\infty)$,
and an argument similar to that used
in \eqref{pointe8+}, we obtain
\eqref{eq2030} holds also in this case.
This finishes the proof of Lemma \ref{lem-sd}.
\end{proof}

\begin{proof}[Proof of Proposition \ref{es-dyadic2}]
Fix $\beta\in[\frac1p-1,\frac1p)$,
a dyadic grid $\mathcal{D}$, $\upsilon\in A_p$,
and $f\in C^\infty$ with $|\nabla f|\in C_{\rm c}$.
Then, by Tonelli's
theorem and Lemma \ref{lem-sd} with $r: = p$,
we obtain, for any $\lambda\in(0,\infty)$,
\begin{align}\label{eq2317}
\sum_{Q\in\mathcal{D}_{\lambda,\beta+1-\frac1p}[f]}
|Q|^{\beta p-1}\upsilon(Q)
& = \int_{\mathbb{R}^n}\sum_{Q\in\mathcal{D}_{\lambda,\beta+1-\frac1p}[f]}
|Q|^{\beta p-1}\upsilon(x){\bf 1}_{Q}(x)\,dx\notag\\
&\lesssim\frac{1}{\frac1p-\beta}
\int_{\bigcup_{Q\in\mathcal{D}_{\lambda,\beta+1-\frac1p}[f]}Q}
|Q_x|^{\beta p -1}\upsilon(x)\,dx,
\end{align}
where, for any
$x\in \bigcup_{Q\in\mathcal{D}_{\lambda,\beta+1-\frac1p}[f]}Q$,
$Q_x$ is as in Lemma \ref{lem-sd}.
From \eqref{esti-omegae70},
we infer that, for any $\lambda\in(0,\infty)$ and
$x\in \bigcup_{Q\in\mathcal{D}_{\lambda,\beta+1-\frac1p}[f]}Q$,
$\lambda|Q_x|^{\beta - \frac1p}\le \sqrt{n} M(|\nabla f|)(x),$
which, combined with \eqref{eq2317}, further implies that
\begin{align*}
\lambda^p\sum_{Q\in\mathcal{D}^\alpha_{\lambda,\beta+1-\frac1p}[f]}
|Q|^{\beta p-1}\upsilon(Q)
\lesssim\frac{1}{\frac1p-\beta}
\int_{\mathbb{R}^n}\left[M(|\nabla f|)(x)
\right]^{p}\upsilon(x)\,dx.
\end{align*}
Using this and Lemma \ref{ApProperty}(v),
we find that Proposition \ref{es-dyadic2} holds when $f\in C^\infty$
satisfying $|\nabla f|\in C_{\rm c}$.
Moreover, let $f\in\dot{W}^{1,p}_\upsilon$.
Applying \cite[Theorem 2.6]{dlyyz.arxiv},
we conclude that there exists
a sequence $\{f_k\}_{k\in\mathbb{N}}$ in
$\{g\in C^\infty:|\nabla g|\in C_{\rm c}\}$ such that,
for any $R\in(0,\infty)$,
$\|(f_k-f){\bf1}_{B({\bf0},R)}\|_{L^p_\upsilon}\to0$
and $\|f_k-f\|_{\dot{W}^{1,p}_\upsilon}\to0$ as $k\to\infty$.
Therefore, from these and Fatou's lemma, we
infer that the present proposition also hold
when $f\in\dot{W}^{1,p}_\upsilon$,
which completes the proof of Proposition \ref{es-dyadic2}.
\end{proof}

Next, we consider the case
$\beta\in(-\infty,\frac1p-1)\cup(\frac1p,\infty)$ in Theorem \ref{thm-cddd}.
To this end, for any
dyadic grid $\mathcal{D}$, any $\lambda\in(0,\infty)$,
and any $b\in\mathbb{R}$, define
$\mathcal{E}_{\lambda,b}:=\{Q\in\mathcal{D}:
\fint_Q|f(x)|\,dx>\lambda|Q|^b\}.$
The following estimate for $\mathcal{E}_{\lambda,b}$
plays an essential role in the proof of
Theorem \ref{thm-cddd} when $\beta\in(-\infty,\frac1p-1)\cup(\frac1p,\infty)$.

\begin{proposition}\label{es-dyadic1}
If $p\in[1,\infty)$ and $\beta\in(-\infty,\frac1p-1)\cup(\frac1p,\infty)$,
then there exists a positive constant $C$ such that,
for any dyadic grid $\mathcal{D}$,
any $\upsilon\in A_p$, and any $f\in L^\infty$ with compact support,
\begin{align*}
\sup_{\lambda\in(0,\infty)}\lambda^p
\sum_{Q\in\mathcal{E}_{\lambda,\beta-\frac1p}[f]}
|Q|^{\beta p-1}\upsilon(Q)
\le C[\upsilon]_{A_p}\|f\|_{L^p_\upsilon}^p.
\end{align*}
\end{proposition}

To prove this proposition,
we employ the good cube method
pioneered by Cohen et al.\ in \cite{cddd03}.
The following concepts of good and bad
cubes were originally introduced in the
unweighted setting in
\cite[Definition 4.2]{cddd03}.

\begin{definition}\label{df-good}
Let $\mathcal{D}$ be a dyadic grid and
$\mathcal{S}\subset\mathcal{D}$.
Assume that $\sigma\in\mathbb{R}$ and $\upsilon$
is a nonnegative locally integrable function on
$\mathbb{R}^n$.
A cube $Q\in\mathcal{S}$ is said to be \emph{$(\sigma,\upsilon)$-good}
in $\mathcal{S}$ if $Q$ is minimal in $\mathcal{S}$ or,
for any collection $\mathcal{P}\subset\mathcal{S}$
of pairwise disjoint cubes strictly contained in
$Q$,
$$\sum_{P\in\mathcal{P}}|P|^{\sigma-1}\upsilon(P)
\le |Q|^{\sigma-1}\upsilon(Q).$$
A cube $Q\in\mathcal{S}$ is said to be \emph{$(\sigma,\upsilon)$-bad}
in $\mathcal{S}$ if it is not $(\sigma,\upsilon)$-good.
The set of all $(\sigma,\upsilon)$-good [resp.\ $(\sigma,\upsilon)$-bad]
cubes in $\mathcal{S}$
is denoted by $\mathcal{G}_{\sigma,\upsilon}(\mathcal{S})$
[resp.\ $\mathcal{B}_{\sigma,\upsilon}(\mathcal{S})$].
\end{definition}

The following
domination properties of good cubes,
whose core idea stems from Cohen et al.\ \cite{cddd03},
are important to the proofs of Propositions \ref{es-dyadic1}
and \ref{cddd-suff}
(and hence Theorem \ref{thm-cddd}) and also of independent interest.

\begin{lemma}\label{lem-good}
Let $\mathcal{D}$ be a dyadic grid and $\mathcal{S}\subset\mathcal{D}$.
Assume that $\sigma\in\mathbb{R}$ and $\upsilon$
is a nonnegative locally integrable function.
\begin{enumerate}
  \item[{\rm(i)}] If there exists a collection
  of cubes in $\mathcal{S}$ that is minimal with respect to the set inclusion,
  then, for any $\gamma\in(-\infty,\sigma)$,
  \begin{align*}
  \sum_{Q\in\mathcal{S}}|Q|^{\gamma-1}\upsilon(Q)
  \le\frac{1}{1-2^{n(\gamma-\sigma)}}
  \sum_{Q\in\mathcal{G}_{\sigma,\upsilon}(\mathcal{S})}|Q|^{\gamma-1}
  \upsilon(Q).
  \end{align*}
  \item[{\rm(ii)}] If $F\subset\mathcal{G}_{\sigma,\upsilon}(\mathcal{S})$
  is a collection of disjoint cubes and $E\subset\mathcal{G}_{\sigma,\upsilon}(\mathcal{S})$
  satisfies that, for any $P\in E$, there exists
  $Q\in F$ such that $P\subset Q$,
  then, for any $\alpha\in(\sigma,\infty)$,
  \begin{align*}
  \sum_{Q\in E}|Q|^{\alpha-1}\upsilon(Q)\le\frac{1}{1-2^{n(\sigma-\alpha)}}
  \sum_{Q\in F}|Q|^{\alpha-1}\upsilon(Q).
  \end{align*}
\end{enumerate}
\end{lemma}

\begin{proof}
We first prove (i). Indeed,
similar to an
argument that used in \cite[p.\,252]{cddd03},
by the definitions of both $(\sigma,\upsilon)$-good and
$(\sigma,\upsilon)$-bad cubes and the
existence of minimal cubes in $\mathcal{S}$,
we obtain,
for any $Q\in\mathcal{B}_{\sigma,\upsilon}(\mathcal{S})$,
\begin{align*}
|Q|^{\sigma-1}\upsilon(Q)
\le\sum_{\{P\in\mathcal{G}_{\sigma,\upsilon}(\mathcal{S}):
P\subsetneqq Q\}}|P|^{\sigma-1}\upsilon(P).
\end{align*}
Therefore, for any $\gamma\in(-\infty,\sigma)$ and
$Q\in\mathcal{B}_{\sigma,\upsilon}(\mathcal{S})$,
\begin{align*}
|Q|^{\gamma-1}\upsilon(Q)
&\le|Q|^{\gamma-\sigma}
\sum_{\{P\in\mathcal{G}_{\sigma,\upsilon}(\mathcal{S}):
P\subsetneqq Q\}}|P|^{\sigma-1}\upsilon(P)
\le |Q|^{\gamma-\sigma}\sum_{j\in\mathbb{N}}
\sum_{P\in\mathcal{G}_{\sigma,\upsilon}(\mathcal{S})
\cap\mathcal{D}_j(Q)}|P|^{\sigma-1}\upsilon(P)\\
&\le\sum_{j\in\mathbb{N}}2^{jn(\gamma-\sigma)}
\sum_{P\in\mathcal{G}_{\sigma,\upsilon}(\mathcal{S})
\cap\mathcal{D}_j(Q)}|P|^{\gamma-1}\upsilon(P).
\end{align*}
From this and Tonelli's theorem, we infer that,
for any $\gamma\in(-\infty,\sigma)$, $\gamma-\sigma\in(-\infty,0)$ and
hence
\begin{align*}
\sum_{Q\in\mathcal{B}_{\sigma,\upsilon}(\mathcal{S})}
|Q|^{\gamma-1}\upsilon(Q)
&\le \sum_{P\in\mathcal{G}_{\sigma,\upsilon}(\mathcal{S})}
|P|^{\gamma-1}\upsilon(P)
\sum_{j\in\mathbb{N}}
\sum_{\{Q\in\mathcal{B}_{\sigma,\upsilon}(\mathcal{S}):
P\in\mathcal{D}_j(Q)\}}2^{jn(\gamma-\sigma)}\\
&=\frac{2^{n(\gamma-\sigma)}}{1-2^{n(\gamma-\sigma)}}
\sum_{P\in\mathcal{G}_{\sigma,\gamma}(\mathcal{S})}
|P|^{\gamma-1}\upsilon(P).
\end{align*}
This, together with the fact that
$\mathcal{S}$ is the disjoint union of both $\mathcal{G}_{\sigma,\upsilon}(\mathcal{S})$
and $\mathcal{B}_{\sigma,\upsilon}(\mathcal{S})$, finishes the proof of (i).

Next, we show (ii). Notice that,
for any $Q\in F$ and $j\in\mathbb{N}$,
$E\cap\mathcal{D}_j(Q)$
is a collection of pairwise disjoint cubes strictly
contained in $Q$. Applying this and the definition of
$(\sigma,\upsilon)$-good cubes, we conclude that,
for any $\alpha\in(\sigma,\infty)$, $Q\in F$, and $j\in\mathbb{Z}_+$,
\begin{align*}
\sum_{P\in E\cap\mathcal{D}_j(Q)}|P|^{\alpha-1}\upsilon(P)
=2^{jn(\sigma-\alpha)}|Q|^{\alpha-\sigma}
\sum_{P\in E\cap\mathcal{D}_j(Q)}
|P|^{\sigma-1}\upsilon(P)\le 2^{jn(\sigma-\alpha)}
|Q|^{\alpha-1}\upsilon(Q).
\end{align*}
Consequently, from this and the assumptions on both
$E$ and $F$, it follows that, for any $\alpha\in(\sigma,\infty)$,
\begin{align*}
\sum_{Q\in E}|Q|^{\alpha-1}\upsilon(Q)
&\le\sum_{Q\in F}\sum_{j\in\mathbb{Z}_+}
\sum_{P\in E\cap\mathcal{D}_{j}(Q)}
|P|^{\alpha-1}\upsilon(P)\\
&\le \sum_{Q\in F}|Q|^{\alpha-1}\upsilon(Q)\sum_{j\in\mathbb{Z}_+}
2^{jn(\sigma-\alpha)}
=\frac{1}{1-2^{n(\sigma-\alpha)}}\sum_{Q\in F}|Q|^{\alpha-1}\upsilon(Q),
\end{align*}
which completes the proof of (ii) and hence Lemma \ref{lem-good}.
\end{proof}

\begin{proof}[Proof of Proposition \ref{es-dyadic1}]
Fix a dyadic grid $\mathcal{D}$,
$\upsilon\in A_p$, $f\in L^\infty$ with compact support, and $\lambda\in(0,\infty)$.
We prove the present proposition by considering the following
two cases for $\beta$.

\emph{Case 1:} $\beta\in (\frac{1}{p},\infty)$.
In this case, for any $Q\in\mathcal{E}_{\lambda,\beta-\frac1p}[f]$,
it holds that
$$\lambda|Q|^{\beta-\frac1p}
<\fint_Q|f(x)|\,dx\le\|f\|_{L^\infty},$$
which implies $|Q|<[\lambda^{-1}\|f\|_{L^\infty}]^{\frac{1}{\beta-\frac1p}}$.
Using this and Definition \ref{df-dyadic}(ii),
we find that each cube $Q\in\mathcal{E}_{\lambda,\beta-\frac1p}[f]$ is contained in
a cube belonging to $\mathcal{E}_{\lambda,\beta-\frac1p}[f]$ that is
maximal with respect to the set inclusion;
denote by $\mathcal{E}^{\rm max}_{\lambda,\beta-\frac1p}[f]$
this set of maximal cubes.
Therefore, applying Definition \ref{df-dyadic}(ii) again
and the assumption $\beta\in(\frac1p,\infty)$,
we obtain
\begin{align}\label{es-dyadic1-e2}
\sum_{Q\in\mathcal{E}_{\lambda,\beta-\frac1p}[f]}
|Q|^{\beta p-1}\upsilon(Q)
&\le\sum_{Q\in\mathcal{E}^{\rm max}_{\lambda,\beta-\frac1p}[f]}
\sum_{j\in\mathbb{Z}_+}\sum_{P\in\mathcal{E}_{\lambda,
\beta-\frac1p}[f]\cap\mathcal{D}_j(Q)}|P|^{\beta p-1}
\upsilon(P)\notag\\
&=\sum_{Q\in\mathcal{E}^{\rm max}_{\lambda,\beta-\frac1p}[f]}
|Q|^{\beta p-1}\sum_{j\in\mathbb{Z}_+}
2^{-jn(\beta p-1)}\sum_{P\in\mathcal{E}_{\lambda,
\beta-\frac1p}[f]\cap\mathcal{D}_j(Q)}\upsilon(P)\notag\\
&\lesssim\sum_{Q\in\mathcal{E}^{\rm max}_{\lambda,\beta-\frac1p}[f]}
|Q|^{\beta p-1}\upsilon(Q).
\end{align}
On the other hand, from Lemma \ref{ApProperty}(iii),
we deduce that,
for any $Q\in\mathcal{E}_{\lambda,\beta-\frac1p}[f]$,
\begin{align}\label{es-dyadic1-e3}
\lambda^p|Q|^{\beta p-1}
<\left[\fint_Q|f(x)|\,dx\right]^p
\le \frac{[\upsilon]_{A_p}}{\upsilon(Q)}
\int_{Q}|f(x)|^p\upsilon(x)\,dx.
\end{align}
Combining this, \eqref{es-dyadic1-e2}, and the disjointness
of cubes in $\mathcal{E}^{\rm max}_{\lambda,\beta-\frac1p}[f]$,
we conclude that
\begin{align*}
\lambda^p\sum_{Q\in\mathcal{E}_{\lambda,\beta-\frac1p}[f]}
|Q|^{\beta p-1}\upsilon(Q)
\lesssim[\upsilon]_{A_p}\sum_{Q\in\mathcal{E}^{\rm max}_{\lambda,\beta-\frac1p}[f]}
\int_{Q}|f(x)|^p\upsilon(x)\,dx
\le[\upsilon]_{A_p}\|f\|_{L^p_\upsilon}^p.
\end{align*}
This then finishes the proof of Proposition \ref{es-dyadic1} in this case.

\emph{Case 2:} $\beta\in(-\infty,\frac1p-1)$. In this case,
choose $\sigma\in(\beta p,1-p)$ and, to simplify the
representation, we denote by $\mathcal{G}$
the set of all $(\sigma,\upsilon)$-good cubes in $\mathcal{E}_{\lambda,
\beta-\frac1p}[f]$
in the present proof. Thus,
by Lemma \ref{lem-good}(i) with $\gamma:=\beta p$
 and the definition
of $\mathcal{E}_{\lambda,\beta-\frac1p}[f]$,
we conclude that
\begin{align*}
\lambda^p\sum_{Q\in\mathcal{E}_{\lambda,
\beta-\frac1p}[f]}|Q|^{\beta p-1}\upsilon(Q)
&\lesssim\lambda^p\sum_{Q\in\mathcal{G}}
|Q|^{\beta p-1}\upsilon(Q)
\le\sum_{Q\in\mathcal{G}}\left[\fint_Q|f(x)|\,dx\right]^p
\upsilon(Q).
\end{align*}
This implies that, to complete the present proof,
we only need to show
\begin{align}\label{es-dyadic1-e4}
\sum_{Q\in\mathcal{G}}
\left[\fint_Q|f(x)|\,dx\right]^p
\upsilon(Q)\lesssim[\upsilon]_{A_p}\|f\|_{L^p_\upsilon}^p,
\end{align}
where the implicit positive constant
depends only on $p$, $n$, and $\sigma$.
To achieve this, we further classify the cubes $Q$
in $\mathcal{G}$ by the magnitude of $\fint_{Q}|f(x)|\,dx$.
Indeed, for any $k\in\mathbb{Z}$, define
\begin{align*}
E_k:=\left\{Q\in\mathcal{G}:2^{kn}<\fint_Q
|f(x)|\,dx\le 2^{(k+1)n}\right\}
\end{align*}
and
$F_k:=\{Q\in\mathcal{G}:Q\text{ is a maximal cube such that }
\fint_Q|f(x)|\,dx>2^{kn}\}.$
If $E_k$ is nonempty for some $k\in\mathbb{Z}$,
then, for any $Q\in E_k$,
$\fint_{Q}|f(x)|\,dx>2^{kn}$, which,
combined with the assumption $f\in L^\infty$ with compact support,
further implies that $|Q|<2^{-kn}\|f\|_{L^1}<\infty$.
Consequently, if $E_k$ is nonempty for some $k\in\mathbb{Z}$,
then $F_k$ is also nonempty and each cube in $E_k$ is
contained in a cube in $F_k$.
Using this, the disjointness of cubes in $F_k$,
the fact that $\sigma<1$, and Lemma \ref{lem-good}(ii)
with $\alpha:=1$,
we obtain
\begin{align*}
\sum_{Q\in\mathcal{G}}
\left[\fint_Q|f(x)|\,dx\right]^p\upsilon(Q)
\sim\sum_{k\in\mathbb{Z}}2^{knp}\sum_{Q\in E_k}\upsilon(Q)
\lesssim\sum_{k\in\mathbb{Z}}2^{knp}\sum_{Q\in F_k}\upsilon(Q).
\end{align*}
This, together with Fatou's lemma,
further implies that, to prove \eqref{es-dyadic1-e4},
it suffices to show that, for any $N\in\mathbb{N}$,
\begin{align}\label{es-dyadic1-e5}
\sum_{k=-N}^\infty 2^{knp}\sum_{Q\in F_k}\upsilon(Q)
\lesssim[\upsilon]_{A_p}\|f\|_{L^p_\upsilon}^p
\end{align}
with the implicit positive constant
depending only on $p$, $n$, and $\sigma$.
To this end, fix $k\in\mathbb{Z}$ such that $F_{k+1}\ne \emptyset$
and choose $P\in F_{k+1}$.
Then, from the definitions of $\{F_{k}\}_{k\in\mathbb{Z}}$,
we infer that
$\fint_{P}|f(x)|\,dx>2^{(k+1)n}>2^{kn}$ and hence
there exists $Q\in F_k$ such that $P\subset Q$.
Moreover, by the maximality of $Q$, we find that
$\fint_{Q}|f(x)|\,dx\le 2^{(k+1)n}$, which implies $Q\notin F_{k+1}$
and hence $P\subsetneqq Q$. Therefore, each cube in $F_{k+1}$
is strictly contained in some cube in $F_k$.
Applying this, the fact that
$\sigma<1$, and the definition of $(\sigma,\upsilon)$-good cubes,
we conclude that
\begin{align*}
\sum_{P\in F_{k+1}}\upsilon(P)
\le\sum_{Q\in F_k}2^{n(\sigma-1)}|Q|^{1-\sigma}\sum_{\{P\in F_{k+1}:P\subsetneqq Q\}}
|P|^{\sigma-1}\upsilon(P)
\le 2^{n(\sigma-1)}\sum_{Q\in F_k}\upsilon(Q).
\end{align*}
From this, the assumption $\sigma\in(\beta p,1-p)$,
\eqref{es-dyadic1-e3}, and the disjointness of cubes in $F_k$
for any given $k\in\mathbb{Z}$,
we deduce that, for any $N\in\mathbb{N}$,
\begin{align*}
\sum_{k=-N}^\infty2^{knp}\sum_{Q\in F_k}\upsilon(Q)
&\le\sum_{k=-N}^{\infty}2^{knp}2^{(k+N)n(\sigma-1)}
\sum_{Q\in F_{-N}}\upsilon(Q)
\sim2^{-Nnp}\sum_{Q\in F_{-N}}\upsilon(Q)\\
&<\sum_{Q\in F_{-N}}\left[\fint_Q|f(x)|\,dx\right]^p
\upsilon(Q)\\
&\le[\upsilon]_{A_p}
\sum_{Q\in F_{-N}}\int_Q|f(x)|^p\upsilon(x)\,dx
\le [\upsilon]_{A_p}\|f\|_{L^p_\upsilon}^p.
\end{align*}
This implies \eqref{es-dyadic1-e5} and hence
\eqref{es-dyadic1-e4}, which then completes the proof of
Proposition \ref{es-dyadic1}.
\end{proof}

Finally, we use the
good cube method to
deal with the third case of Theorem \ref{thm-cddd} as follows,
namely $p=1$, $n\in\mathbb{N}\cap[2,\infty)$,
and $\beta\in(-\infty,1-\frac1n)$.

\begin{proposition}\label{cddd-suff}
If $n\in\mathbb{N}\cap[2,\infty)$ and $\beta\in(-\infty,
1-\frac1n)$,
then there exists a positive constant $C$ such that,
for any dyadic grid $\mathcal{D}$, any $\upsilon\in A_1$,
and any $f\in C^\infty$ satisfying
$|\nabla f|\in C_{\rm c}$,
\begin{align*}
\sup_{\lambda\in(0,\infty)}\lambda
\sum_{Q\in\mathcal{D}_{\lambda,\beta}[f]}
|Q|^{\beta-1}\upsilon(Q)\le
C[\upsilon]_{A_1}
\|\,|\nabla f|\,\|_{L^1_\upsilon}.
\end{align*}
\end{proposition}

\begin{proof}
Fix a dyadic grid $\mathcal{D}$, $\upsilon\in A_1$,
$\lambda\in(0,\infty)$,
and $f\in C^\infty$ with
$|\nabla f|\in C_{\rm c}$ and choose
$\sigma\in(\beta,1-\frac1n)$.
In the present proof,
we simply denote by $\mathcal{G}$
the set of all $(\sigma,\upsilon)$-good cubes in
$\mathcal{D}_{\lambda,\beta}[f]$.
Then, from Lemma \ref{lem-good}(i) and the definition of
$\mathcal{D}_{\lambda,\beta}[f]$, we deduce that
\begin{align}\label{esti-omegae1}
\lambda\sum_{Q\in\mathcal{D}_{\lambda,\beta}[f]}
|Q|^{\beta-1}\upsilon(Q)\lesssim\lambda
\sum_{Q\in\mathcal{G}}|Q|^{\beta-1}
\upsilon(Q)\le\sum_{Q\in\mathcal{G}}\omega_Q(f)
\frac{\upsilon(Q)}{|Q|}.
\end{align}
On the other hand, since $n\in\mathbb{N}\cap[2,\infty)$,
there exists a constant $C_f$ such that
$f-C_f\in C_{\rm c}$.
This implies that, if, for any $t\in\mathbb{R}$,
let $E_t:=\{x\in\mathbb{R}^n:f(x)-C_f> t\}$,
then, for any $Q\in\mathcal{G}$ and $x\in Q$,
\begin{align*}
f(x)=C_f+\left[f(x)-C_f\right]=C_f-\left\|f-C_f\right\|_{L^\infty(Q)}
+\int_{-\|f-C_f\|_{L^\infty(Q)}}^{\infty}
{\bf 1}_{E_t}(x)\,dt.
\end{align*}
Using this, we obtain, for any $Q\in\mathcal{G}$
and $x,y\in Q$,
\begin{align*}
|f(x)-f(y)|=\left|\int_{-\|f-C_f\|_{L^\infty(Q)}}^{\infty}
\left[{\bf 1}_{E_t}(x)-{\bf 1}_{E_t}(y)\right]\,dt\right|\le\int_{-\infty}^{\infty}
\left|{\bf 1}_{E_t}(x)-{\bf 1}_{E_t}(y)\right|\,dt.
\end{align*}
Thus, from this and Tonelli's theorem, we further infer that,
for any $Q\in\mathcal{G}$,
\begin{align}\label{esti-omegae2}
\omega_Q(f)\le\int_{-\infty}^{\infty}
\omega_Q({\bf 1}_{E_t})\,dt.
\end{align}
To complete the proof of the present proposition, it suffices to
show that, for any domain $E\subset\mathbb{R}^n$
with smooth boundary satisfying $\min\{|E|,|E^\complement|\}<\infty$,
\begin{align}\label{esti-omegae3}
\sum_{Q\in\mathcal{G}}\omega_Q({\bf 1}_E)
\frac{\upsilon(Q)}{|Q|}\lesssim[\upsilon]_{A_1}
\int_{\partial E}\upsilon(x)\,d\mathcal{H}^{n-1}(x),
\end{align}
where $\mathcal{H}^{n-1}$ denotes the $(n-1)$-dimensional Hausdorff
measure and the implicit positive constant depends only on $n$ and $\sigma$.
Indeed, if we assume that \eqref{esti-omegae3} holds for the moment,
then, by Sard's theorem (see, for instance, \cite[Theorem 6.10]{l13})
and \cite[Example 1.32]{l13}, we find that, for almost every $t\in(0,\infty)$,
$E_t$ is a domain with smooth boundary, $\partial E_t
=\{x\in\mathbb{R}^n:f(x)-C_f=t\}$, and
$$\min\left\{|E_t|,\left|E_t^\complement\right|\right\}
\le\frac{1}{t}\left\|f-C_f\right\|_{L^1}<\infty.$$
Combining these, \eqref{esti-omegae1}, \eqref{esti-omegae2},
Tonelli's theorem,
\eqref{esti-omegae3}, and the
famous coarea formula (see, for instance,
\cite[p.\,258]{f69}),
we further conclude that
\begin{align*}
\lambda\sum_{Q\in\mathcal{D}_{\lambda,\beta}[f]}
|Q|^{\beta-1}\upsilon(Q)
&\lesssim\int_{-\infty}^{\infty}\sum_{Q\in\mathcal{G}}
\omega_Q({\bf 1}_{E_t})\frac{\upsilon(Q)}{|Q|}\,dt
\lesssim[\upsilon]_{A_1}
\int_{-\infty}^{\infty}\int_{\partial E_t}\upsilon(x)
\,d\mathcal{H}^{n-1}(x)\,dt\\
&=[\upsilon]_{A_1}
\int_{-\infty}^{\infty}\int_{\{y\in\mathbb{R}^n:
f(y)-C_f=t\}}\upsilon(x)
\,d\mathcal{H}^{n-1}(x)\,dt\\
&=[\upsilon]_{A_1}
\int_{\mathbb{R}^n}\left|\nabla f(x)\right|
\upsilon(x)\,dx.
\end{align*}
This further
implies that the present proposition holds.

Therefore, we now only need to prove \eqref{esti-omegae3}.
To achieve this, for any $k\in \mathbb{N}$,
define
\begin{align*}
\mathcal{G}_{k}
:=\left\{Q\in\mathcal{G}:
\min\left\{|Q\cap E|,|Q\setminus E|\right\}
\in\left(2^{-k-1}|Q|,2^{-k}|Q|\right]\right\}.
\end{align*}
Then, for any given $k\in\mathbb{N}$ and any
$Q\in\mathcal{G}_{k}$, it holds that
\begin{align*}
|Q|<2^{k+1}\min\left\{|Q\cap E|,|Q\setminus E|\right\}
\le 2^{k+1}\min\left\{|E|,\left|E^\complement\right|\right\}
<\infty,
\end{align*}
which further implies that
each cube $Q\in\mathcal{G}_{k}$ is contained
in a cube in $\mathcal{G}_{k}$ that is
maximal with respect to the set inclusion.
For any $k\in\mathbb{N}$,
denote by  $\mathcal{G}_{k}^{\rm max}$
the set of all maximal cubes in $\mathcal{G}_{k}$.
Thus, from the definitions of $\{\mathcal{G}_k\}_{k\in\mathbb{N}}$
and Lemma \ref{lem-good}(ii) with $\alpha:=1-\frac1n$,
it follows that
\begin{align}\label{esti-omegae4}
\sum_{Q\in\mathcal{G}}\omega_Q({\bf 1}_E)
\frac{\upsilon(Q)}{|Q|}
&=2\sum_{Q\in\mathcal{G}}|Q|^{-1-\frac1n}
\int_{Q\cap E}\int_{Q\setminus E}\,dx\,dy\frac{\upsilon(Q)}{|Q|}\notag\\
&\le2\sum_{Q\in\mathcal{G}}|Q|^{-1-\frac1n}
\min\left\{|Q\cap E|,|Q\setminus E|\right\}\upsilon(Q)\notag\\
&\sim\sum_{k\in\mathbb{N}}2^{-k}\sum_{Q\in\mathcal{G}_k}
|Q|^{-\frac1n}\upsilon(Q)
\lesssim\sum_{k\in\mathbb{N}}2^{-k}\sum_{Q\in\mathcal{G}_k^{\rm max}}
|Q|^{-\frac1n}\upsilon(Q).
\end{align}
On the other hand, for any $k\in\mathbb{N}$,
using the definition of $\mathcal{G}_{k}$ again
and the isoperimetric inequality
(see, for instance, \cite[Theorem 2.3]{cddd03}),
we conclude that, for any $Q\in\mathcal{G}_{k}$,
\begin{align*}
|Q|\sim 2^k
\min\left\{|Q\cap E|,|Q\setminus E|\right\}
\lesssim 2^k\left[\mathcal{H}^{n-1}
\left(\partial E\cap Q\right)\right]^{\frac{n}{n-1}}.
\end{align*}
Applying this, \eqref{esti-omegae4},
Lemma \ref{ApProperty}(i), and the disjointness
of cubes in $\mathcal{G}_{k}^{\rm max}$ for any $k\in\mathbb{N}$, we further
find that
\begin{align*}
\sum_{Q\in\mathcal{G}}\omega_Q({\bf 1}_E)
\frac{\upsilon(Q)}{|Q|}&\lesssim\sum_{k\in\mathbb{N}}
2^{-\frac kn}
\sum_{Q\in\mathcal{G}_k^{\rm max}}\mathcal{H}^{n-1}
(\partial E\cap Q)\frac{\upsilon(Q)}{|Q|}\\
&\le[\upsilon]_{A_1}\sum_{k\in\mathbb{N}}2^{-\frac kn}
\sum_{Q\in\mathcal{G}_k^{\rm max}}
\int_{\partial E\cap Q}\upsilon(x)\,d\mathcal{H}^{n-1}(x)\\
&\lesssim[\upsilon]_{A_1}
\int_{\partial E}\upsilon(x)\,d\mathcal{H}^{n-1}(x).
\end{align*}
This finishes the proof of \eqref{esti-omegae3} and hence
Proposition \ref{cddd-suff}.
\end{proof}

We now give the proof of Theorem \ref{thm-cddd}.

\begin{proof}[Proof of Theorem \ref{thm-cddd}]
Let $f\in C^\infty$ satisfy
$|\nabla f|\in C_{\rm c}$.
Then, using \eqref{esti-omegae5}, Lemma \ref{ApProperty}(iv),
and Propositions
\ref{es-dyadic2}, \ref{es-dyadic1}, and \ref{cddd-suff},
we find that both \eqref{thm-cddd-e1}
and \eqref{thm-cddd-e2} hold for this $f$.
Combining this and an argument similar to that
used in the proof of Proposition \ref{es-dyadic2},
we conclude that \eqref{thm-cddd-e1}
and \eqref{thm-cddd-e2} also hold
when $f\in\dot{W}^{1,p}_\upsilon$.
This then finishes the proof of
Theorem \ref{thm-cddd}.
\end{proof}

Next, we prove that the dependences
on weight constants in Theorem \ref{thm-cddd} are sharp by using
the shifted dyadic grids $\{\mathcal{D}^\alpha\}_{\alpha\in\{0,\frac13,\frac23\}^n}$
presented in Example \ref{e-dya}.
For any $\alpha\in\{0,\frac13,\frac23\}^n$,
$b\in\mathbb{R}$, $\lambda\in(0,\infty)$, and $f\in L^1_{\loc}$,
define
\begin{align}\label{df-ab}
\mathcal{D}^{\alpha}_{\lambda,b}[f]
:=\left\{Q\in\mathcal{D}^{\alpha}:
\omega_Q(f)>\lambda|Q|^{b}\right\}.
\end{align}
Then the following proposition shows the sharpness
of weight constants in Theorem \ref{thm-cddd}.

\begin{proposition}\label{sharp}
Let $p\in[1,\infty)$.
\begin{enumerate}
  \item[{\rm(i)}] If there exist $\beta,\gamma\in\mathbb{R}$
  and $C\in(0,\infty)$ such that, for any $\upsilon\in A_1(\mathbb{R})$
  and $f\in \dot{W}^{1,p}_\upsilon(\mathbb{R})$,
  \begin{align}\label{sharp-e1}
  \sup_{\lambda\in(0,\infty)}\lambda^p
  \sum_{I\in\mathcal{D}^{0}_{\lambda,\beta+1-\frac1p}[f]}
  |I|^{\beta p-1}\upsilon(I)\le C[\upsilon]_{A_1(\mathbb{R})}^\gamma
  \|f'\|_{L^p_\upsilon(\mathbb{R})}^p,
  \end{align}
  then $\gamma\ge 1$.
  \item[{\rm(ii)}] If $p\in(1,\infty)$ and there exist
  $\gamma\in\mathbb{R}$ and $C\in(0,\infty)$ such that, for any
  $\upsilon\in A_p(\mathbb{R})$ and $f\in\dot{W}^{1,p}_\upsilon(\mathbb{R})$,
  \begin{align}\label{sharp-e2}
  \sup_{\lambda\in(0,\infty)}\lambda^p
  \sum_{I\in\mathcal{D}^{\frac13}_{\lambda,0}[f]}
  |I|^{-p}\upsilon(I)\le C[\upsilon]_{A_p(\mathbb{R})}^\gamma
  \|f'\|_{L^p_\upsilon(\mathbb{R})}^p,
  \end{align}
  then $\gamma\ge\frac{p}{p-1}$.
  \item[{\rm(iii)}]
  If $p\in(1,\infty)$, $\beta_0\in(\frac{1}{p}-1,\frac1p)$,
  and there exist $\gamma\in\mathbb{R}$ and
  $C\in(0,\infty)$ such that, for any $\beta\in(\frac1p-1,\beta_0)$,
  $\upsilon\in A_p(\mathbb{R})$, and
  $f\in \dot{W}^{1,p}_\upsilon(\mathbb{R})$,
  \begin{align}\label{sharp-e6}
  \sup_{\lambda\in(0,\infty)}\lambda^p
  \sum_{I\in\mathcal{D}^{\frac13}_{\lambda,\beta+1-\frac1p}[f]}
  |I|^{\beta p-1}\upsilon(I)\le C[\upsilon]_{A_p(\mathbb{R})}^\gamma
  \|f'\|_{L^p_\upsilon(\mathbb{R})}^p,
  \end{align}
  then $\gamma\ge \frac{p}{p-1}$.
\end{enumerate}
\end{proposition}

\begin{proof}
We first prove (i). To this end, fix
$\delta\in(0,1)$ and,
for any $x\in\mathbb{R}$, define
$\upsilon(x):=|x-\frac12|^{\delta-1}$.
Then, using \cite[Example 7.1.7]{g2014}, we find that
$\upsilon\in A_1(\mathbb{R})$ and $[\upsilon]_{A_1(\mathbb{R})}\sim\frac1\delta$
with the positive equivalence constants independent of $\delta$
(see also \cite[p.\,26]{c25}).
On the other hand, choose $f\in C_{\rm c}^\infty(\mathbb{R})$
satisfying ${\bf 1}_{(0,1)}\le f\le {\bf 1}_{(-1,2)}$.
Then
\begin{align}\label{sharp-e3}
\|f'\|_{L^p_\upsilon(\mathbb{R})}^p
=\int_{(-1,2)\setminus(0,1)}|f'(x)|^p\upsilon(x)\,dx
\lesssim\|f'\|_{L^\infty(\mathbb{R})}^p,
\end{align}
which implies $f\in\dot{W}^{1,p}_\upsilon(\mathbb{R})$.
In addition, for the interval $I_0:=(0,4)$, it holds that
\begin{align*}
\omega_{I_0}(f)=\frac{1}{|I_0|^{2}}
\int_{I_0}\int_{I_0}|f(x)-f(y)|\,dx\,dy
>\frac1{16}\int_{0}^{1}\int_{3}^{4}\,dx\,dy
=4^{-\beta-3+\frac1p}|I_0|^{\beta+1-\frac1p}.
\end{align*}
Thus, $I_0\in\mathcal{D}^{0}_{4^{-\beta-3+\frac1p},\beta-\frac1p}[f]$.
Applying this, \eqref{sharp-e3}, the fact that
$[\upsilon]_{A_1(\mathbb{R})}\sim\frac1\delta$
with the positive equivalence constants independent of $\delta$,
and \eqref{sharp-e1}, we obtain
$$\delta^{-\gamma}\sim[\upsilon]_{A_1(\mathbb{R})}^\gamma
\gtrsim\upsilon(I_0)\ge\int_{\frac12}^{4}\left(x-\frac12\right)
^{\delta-1}\,dx\sim\delta^{-1},$$
where the implicit positive constants are independent of $\delta$.
This, combined with the arbitrariness of $\delta\in(0,1)$,
further implies that $\gamma\ge 1$, which completes the proof of (i).

Next, we show (ii). To achieve this, fix $\delta\in(0,\frac12)$
and let $\upsilon(x):=|x|^{(p-1)(1-\delta)}$
for any $x\in\mathbb{R}$. From \cite[Lemma 1.4]{b93},
we infer that $\upsilon\in A_p(\mathbb{R})$ and
$[\upsilon]_{A_p(\mathbb{R})}\sim\delta^{1-p}$,
where the positive equivalence constants are independent of
$\delta$.
Moreover, for any $x\in\mathbb{R}$, define
$f(x):=\int_{-\infty}^{x}t^{\delta-1}{\bf 1}_{(0,1)}(t)\,dt$.
Then the weak derivative $f'(x)=x^{\delta-1}{\bf 1}_{(0,1)}(x)$
for almost every $x\in\mathbb{R}$
and hence
\begin{align}\label{sharp-e4}
\|f'\|_{L^p_\upsilon(\mathbb{R})}^p
=\int_{0}^{1}x^{\delta-1}\,dx=\frac1\delta.
\end{align}
Therefore, $f\in\dot{W}^{1,p}_\upsilon(\mathbb{R})$.
We now claim that $\{I_j\}_{j=2}^\infty:=\{[-\frac{2^{2j-1}}{3},
\frac{2^{2j}}{3})\}_{j=2}^\infty
\subset\mathcal{D}^{\frac13}_{\frac{1}{9\delta},0}[f]$.
Indeed, for any $j\in\mathbb{N}\cap[2,\infty)$, it holds that
\begin{align*}
\omega_{I_j}(f)>|I_j|^{-2}\int_{1}^{\frac{2^{2j}}{3}}\int_{-\frac{2^{2j-1}}{3}}^{0}
\int_{0}^{1}t^{\delta-1}\,dt\,dx\,dy
=\frac{1}{\delta}\left(\frac{2^{2j}}{3}-1\right)
\frac{2^{2j-1}}{3}|I_j|^{-2}
>\frac{1}{9\delta},
\end{align*}
which implies the above claim. Using this,
the fact that $[\upsilon]_{A_p(\mathbb{R})}\sim\delta^{1-p}$
with the positive equivalence constants independent of
$\delta$, \eqref{sharp-e4}, \eqref{sharp-e2}, and
the assumption $p\in(1,\infty)$, we conclude that
\begin{align*}
\delta^{(1-p)\gamma-1}&\sim[\upsilon]_{A_p(\mathbb{R})}^\gamma
\|f'\|_{L^p_\upsilon(\mathbb{R})}^p
\gtrsim\delta^{-p}\sum_{j=2}^{\infty}2^{-2jp}\upsilon(I_j)\\
&\ge\delta^{-p}\sum_{j=2}^{\infty}2^{-2jp}\int_{0}^{\frac{2^{2j}}{3}}
x^{(p-1)(1-\delta)}\,dx\sim\delta^{-p}\sum_{j=2}^{\infty}2^{2j\delta(1-p)}
\sim\delta^{-p-1},
\end{align*}
where the implicit positive constants are independent of $\delta$.
From this and the arbitrariness of $\delta\in(0,\frac12)$,
we further deduce that $\gamma\ge\frac{p}{p-1}$.
This finishes the proof of (ii).

Finally, we prove (iii).
Indeed, for any $\beta\in(\frac1p-1,\beta_0)$, define
$\upsilon_\beta(x):=|x|^{(p-1)(\frac1p-\beta)}$ for any $x\in\mathbb{R}$ and
$f_\beta(y):=\int_{-\infty}^{y}t^{\beta-\frac1p}{\bf 1}_{(0,1)}(t)\,dt$
for any $y\in(0,\infty)$.
Then, from \cite[Lemma 1.4]{b93} and an argument similar to that used in
\eqref{sharp-e4}, it follows that,
for any $\beta\in(\frac1p-1,\beta_0)$,
$\upsilon_\beta\in A_p(\mathbb{R})$,
$[\upsilon_\beta]_{A_p(\mathbb{R})}\sim(\beta+1-\frac1p)^{1-p}$
with the positive equivalence constants depending only on $p$,
$f_\beta\in\dot{W}^{1,p}_{\upsilon_\beta}(\mathbb{R})$,
and
$\|f_\beta'\|_{L^p_{\upsilon_\beta}(\mathbb{R})}^p=(\beta+1-\frac1p)^{-1}$.
Define the
sequence of intervals $\{Q_j\}_{j\in\mathbb{N}}:=
\{(-\frac{2^{-2j-1}}{3},\frac{2^{-2j}}{3})\}_
{j\in\mathbb{N}}\subset\mathcal{D}^{\frac13}$.
Then, for any $\beta\in(\frac1p-1,\beta_0)$ and
$j\in\mathbb{N}$, it holds that
\begin{align*}
\omega_{Q_j}(f_\beta)>|Q_j|^{-2}
\int_{\frac{2^{-2j-1}}{3}}^{\frac{2^{-2j}}{3}}
\int_{-\frac{2^{-2j-1}}{3}}^{0}
\int_{0}^{\frac{2^{-2j-1}}{3}}t^{\beta-\frac1p}\,dt\,dx\,dy
\sim\left(\beta+1-\frac1p\right)^{-1}|Q_j|^{\beta+1-\frac1p},
\end{align*}
where the positive equivalence constants depend only on $p$
and $\beta_0$.
For any $\beta\in(\frac1p-1,\beta_0)$, combining this, \eqref{sharp-e6},
$\|f_\beta'\|_{L^p_{\upsilon_\beta}(\mathbb{R})}^p=(\beta+1-\frac1p)^{-1}$,
and $[\upsilon_\beta]_{A_p(\mathbb{R})}\sim(\beta+1-\frac1p)^{1-p}$
with the positive equivalence constants depending only on $p$,
we obtain
\begin{align*}
\left(\beta+1-\frac1p\right)^{(1-p)\gamma-1}&\sim
\left[\upsilon_\beta\right]_{A_p(\mathbb{R})}^\gamma
\left\|f_\beta'\right\|_{L^p_{\upsilon_\beta}(\mathbb{R})}^p
\gtrsim\left(\beta+1-\frac1p\right)^{-p}
\sum_{j\in\mathbb{N}}|Q_j|^{\beta p-1}\upsilon_\beta(Q_j)\\
&\ge\left(\beta+1-\frac1p\right)^{-p}
\sum_{j\in\mathbb{N}}2^{-2j(\beta p-1)}
\int_{0}^{\frac{2^{-2j}}{3}}x^{(p-1)(\frac1p-\beta)}\,dx\\
&\sim\left(\beta+1-\frac1p\right)^{-p}
\sum_{j\in\mathbb{N}}2^{-2j(\beta+1-\frac1p)}
\sim\left(\beta+1-\frac1p\right)^{-p-1},
\end{align*}
where the implicit positive constants
depend only on $p$ and $\beta_0$.
Using this and the arbitrariness of $\beta\in(\frac1p-1,\beta_0)$,
we conclude $\gamma\ge\frac{p}{p-1}$.
This finishes the proof of (iii) and hence
Proposition \ref{sharp}.
\end{proof}

\section{Almost Wavelet Characterization of $W^{1,1}_{\upsilon}$
and its Sharp Real Interpolation}\label{s-ip}

The main target of this section is to
establish an almost characterization of
the critical weighted Sobolev space
$W^{1,1}_{\upsilon}$ in terms of Daubechies wavelets.
As an application, we prove
the real interpolation between this weighted
Sobolev space and weighed Besov spaces with the sharp
smoothness exponent (Theorem \ref{t953}),
which is important to obtain Gagliardo--Nirenberg
inequalities in Section \ref{sec-gn}.

In what follows, we simply write $\mathcal{D}:=\mathcal{D}^{\bf 0}$.
Let $p\in[1,\infty]$, $I:=2^{-j}[k+[0,1)^n]\in\mathcal{D}$
for some $j\in\mathbb{Z}$ and $k\in\mathbb{Z}^n$,
and $g$ be a complex-valued function on $\mathbb{R}^n$. Define
$g_{I,p}(\cdot):=2^{\frac{jn}{p}}g(2^{j}\cdot-k).$
Notice that $g_{I,p}$ is $L^p$-normalized, that is,
$\|g_{I,p}\|_{L^p}=\|g\|_{L^p}$.

Now, we recall some notation
on the wavelet system (see, for example,
\cite{cddd03}). Take a univariate scaling function
$\varphi\in C_{\rm c}^N(\mathbb{R})$ and an associated univariate
wavelet function $\psi\in C_{\rm c}^{N}(\mathbb{R})$ for some
$N\in\mathbb{N}$ and define
$\psi^0:=\varphi$ and $\psi^1:=\psi$. For any $e\in \{0,1\}^n$
and $x:=(x_1,\ldots,x_n)\in\mathbb{R}^n$,
let $\psi^e(x):=\prod_{i=1}^{n}\psi^{e_i}(x_i)$.
Also, assume that, for any $e\in E:=\{0,1\}^n\setminus\{{\bf0}\}$
and $\alpha:=(\alpha_1,\ldots,\alpha_n)\in \mathbb{Z}_+^n$ with $|\alpha|
:=|\alpha_1|+\cdots+|\alpha_n|\le N$,
\begin{align}\label{e1937}
\int_{\mathbb{R}^n}x^\alpha \psi^e(x)\,dx=0.
\end{align}
Let $\mathcal{D}_+:=\bigcup_{j\ge 0}\mathcal{D}_j$.
It is well known that the $L^2$-normalized
wavelet system $\{\psi^e_{I,2}\}_{e\in E,I\in\mathcal{D}_+}
\cup\{\psi^{\bf 0}_{I,2}\}_{I\in\mathcal{D}_0}$
forms a complete orthonormal basis in $L^2$ and
$N$ is called the \emph{order} of this wavelet system.
In what follows, to ensure the compatibility with
the scaling properties of the critical Sobolev space,
we use the \emph{renormalized
wavelet system} $\{\psi^e_{I,n^\ast}\}_{e\in E,I\in\mathcal{D}_+}
\cup\{\psi^{\bf 0}_{I,n^\ast}\}_{I\in\mathcal{D}_0}$
and its \emph{dual system}
$\{\psi^e_{I,n}\}_{e\in E,I\in\mathcal{D}_+}
\cup\{\psi^{\bf 0}_{I,n}\}_{I\in\mathcal{D}_0}$,
where $n^*:=\frac{n}{n-1}$.
For simplicity of presentation,
let $\Omega_0:=\{{\bf0}\}\times\mathcal{D}_0$,
$\Omega_1:=E\times\mathcal{D}_+$, and $\Omega:=\Omega_0\cup\Omega_1$.
Moreover, for any $(e,I)\in\Omega$ and $p\in[1,\infty]$,
define $I_\omega:=I$ and
$\psi_{\omega,p}:=\psi^e_{I,p}$.
Then, for any suitable distribution $f\in\mathcal{S}'$, its
wavelet coefficient is defined to be
$\{\langle f,\psi_{\omega,n}\rangle\}_{\omega\in\Omega}$
whenever $\langle f,\psi_{\omega,n}\rangle$ for any $\omega\in\Omega$
is well-defined (see Proposition \ref{wavelet}).

We next introduce some weighted sequence spaces.
Let $\upsilon\in A_1$, $\beta\in\mathbb{R}$, and $p\in[1,\infty)$.
The \emph{sequence space} $\ell^p_{\beta,\upsilon}$
(resp.\ \emph{weak sequence space} $w\ell_{\beta,\upsilon}^1$)
is defined to be the set of all $a:=\{a_{\omega}\}_{\omega\in\Omega}$
in $\mathbb{C}$ such that
\begin{align*}
\|a\|_{\ell_{\beta,\upsilon}^{p}}:=
\left[\sum_{\omega\in\Omega}
\left|I_\omega\right|^{\beta-1}\upsilon
\left(I_\omega\right)
\left|a_\omega\right|^p
\right]^{\frac1p}<\infty
\end{align*}
\begin{align*}
\left[\text{resp.}\ \|a\|_{w\ell_{\beta,\upsilon}^1}:=
\sup_{\lambda\in(0,\infty)}\lambda
\sum_{\{\omega\in\Omega:|a_\omega|>\lambda\}}
\left|I_\omega\right|^{\beta-1}
\upsilon\left(I_\omega\right)<\infty\right].
\end{align*}

We now establish the following almost characterization
of $W^{1,1}_{\upsilon}$ in terms of wavelet expansions.

\begin{theorem}\label{p2143}
If the order $N$ of a renormalized
wavelet system
$\{\psi_{\omega,n^*}\}_{\omega\in\Omega}$ satisfies
$N>n+1$, $\upsilon\in A_1$,
and $\beta\in (-\infty,1-\frac1n)\cup(1,\infty)$, then, for any $f\in W^{1,1}_{\upsilon}$,
\begin{align}\label{e1453}
f=\sum_{\omega\in\Omega}\left\langle f,\psi_{\omega,n}\right\rangle
\psi_{\omega,n^*}
\end{align}
holds in $\mathcal{S}'$ and
\begin{align}\label{e2026}
\left\|\left\{\frac{\langle f,\psi_{\omega,n}\rangle}
{|I_\omega|^\beta}\right\}_{\omega\in\Omega}
\right\|_{w\ell^1_{\beta,\upsilon}}
\lesssim[\upsilon]_{A_1}\|f\|_{W^{1,1}_{\upsilon}}\lesssim[\upsilon]_{A_1}^2
\left\|\left\{\frac{\langle f,\psi_{\omega,n}\rangle}
{|I_\omega|^\beta}\right\}_{\omega\in\Omega}
\right\|_{\ell^1_{\beta,\upsilon}},
\end{align}
where, for any $\omega\in\Omega$,
$\langle f,\psi_{\omega,n}\rangle$
is defined by using the Lebesgue integral
between $f$ and $\psi_{\omega,n}$
and the implicit positive constants are
independent of both $[\upsilon]_{A_1}$ and $f$.
\end{theorem}

\begin{remark}\label{rem-cdpx}
If $\upsilon\equiv1$, then Theorem \ref{p2143}
reduces to the almost wavelet characterization
obtained in \cite[Theorem 1.3]{cddd03}; in particular,
if $\beta=0$, then Theorem \ref{p2143}
was originally established in \cite[p.\,615,
Section 8]{cdpx99} (see also \cite[Theorem 1.1]{cddd03}).
Observe that the ``norm" $\|\cdot\|_{w\ell^1_{\beta,\upsilon}}$ appearing in the
left-hand side of \eqref{e2026} is the weak version
of the norm $\|\cdot\|_{\ell^1_{\beta,\upsilon}}$
appearing in its right-hand side, which is precisely
why the characterization in Theorem \ref{p2143}
is referred to as the almost characterization of
$W^{1,1}_\upsilon$.
\end{remark}

To show this theorem, we need the
Calder\'on reproducing formula.
For any $f\in\mathcal{S}$,
recall its \emph{Fourier transform} $\mathcal{F} f$ is defined by
setting, for any $x\in\mathbb{R}^n$,
$\mathcal{F}f(x):=\int_{\mathbb{R}^n}
f(\xi)e^{-2\pi ix\cdot \xi}\,d\xi$
and its \emph{inverse Fourier transform} $\mathcal{F}^{-1}f$
is defined by setting
$\mathcal{F}^{-1}f(x):=\mathcal{F}f(-x)$ for any
$x\in\mathbb{R}^n$.
Choose a radial function $\eta_0\in\mathcal{S}$ satisfying
\begin{align}\label{ass-eta0}
\mathop\mathrm{\,supp\,}
\left(\mathcal{F}\theta_0\right)\subset
B({\bf0},2)\quad\text{and}\quad
\inf_{\xi\in B({\bf0},\frac53)}\left|\mathcal{F}\theta_0(\xi)\right|>0
\end{align}
and the other radial function $\eta_1\in\mathcal{S}$ satisfying
\begin{align}\label{ass-eta1}
\mathop\mathrm{\,supp\,}
\left(\mathcal{F}\theta_1\right)\subset
B({\bf0},2)\setminus B\left({\bf0},\frac12\right)
\quad\text{and}\quad
\inf_{\xi\in B({\bf0},\frac53)\setminus
B({\bf0},\frac35)}\left|\mathcal{F}\theta_1(\xi)\right|>0.
\end{align}
Using \cite[pp.\,130--131]{fj90}
(see also \cite[p.\,24, (2.6)]{ysy10}), we find that
there exist two radial functions $\eta_0,\eta_1\in\mathcal{S}$ satisfying
the same assumptions, respectively, as in \eqref{ass-eta0} and
\eqref{ass-eta1} and, for any $\xi\in\mathbb{R}^n$,
\begin{align*}
\mathcal{F}\theta_0(\xi)\mathcal{F}\eta_0(\xi)
+\sum_{j\in\mathbb{N}}
\mathcal{F}\theta_1\left(2^{-j}\xi\right)
\mathcal{F}\eta_1\left(2^{-j}\xi\right)=1.
\end{align*}
Moreover, if let $\theta_{I,2}:=(\theta_{0})_{I,2}$
and $\eta_{I,2}:=(\eta_{0})_{I,2}$
for any $I\in\mathcal{D}_0$ and
$\theta_{I,2}:=(\theta_1)_{I,2}$ and $\eta_{I,2}:=(\eta_1)_{I,2}$  for any
$I\in\mathcal{D}_+\setminus\mathcal{D}_0$,
then the following Calder\'on
reproducing formula holds, which is precisely \cite[(12.4)]{fj90}
(see also \cite[p.\,24, Lemma 2.3]{ysy10}).

\begin{lemma}\label{crf}
For any $f\in\mathcal{S}$ (resp. $f\in\mathcal{S}'$),
$f=\sum_{I\in\mathcal{D}_+}\left\langle f,\theta_{I,2}\right\rangle
\eta_{I,2}$
holds in $\mathcal{S}$ (resp. $\mathcal{S}'$),
here and thereafter, $\langle\cdot,\cdot\rangle$ denotes the
the duality pairing between $\mathcal{S}'$ and $\mathcal{S}$.
\end{lemma}

Let $L\in\mathbb{Z}_+$.
For any  $L$-times
differentiable function
$\varphi$, define
$$
\|\varphi\|_{L}:=\sup_{\{\gamma\in\mathbb{Z}_+^n:|\gamma|\le L\}}
\sup_{x\in\mathbb{R}^n}(1+|x|)^{n+L+|\gamma|}
\left|\partial^\gamma\varphi(x)\right|,$$
here and thereafter, for any $\gamma:=(\gamma_1,\ldots,\gamma_n)\in\mathbb{Z}_+^n$,
$\partial^\gamma:=(\frac{\partial}{\partial x_1})^{\gamma_1}
\cdots(\frac{\partial}{\partial x_n})^{\gamma_n}$.
An $f\in\mathcal{S}'$ is called
a \emph{tempered distribution of order $L$}
if there exists a positive constant $C$ such that,
for any $\varphi\in\mathcal{S}$,
$|\langle f,\varphi\rangle|\le C\|\varphi\|_{L}$.
Based on the above reproducing formula,
we obtain the following
representation
criterion of wavelets coefficients,
which plays an important role in the proof of Theorem \ref{p2143}. This
criterion
might be folklore but, for completeness, we give a
detailed proof here.

\begin{proposition}\label{wavelet}
Let $L\in\mathbb{Z}_+$ and $f\in\mathcal{S}'\cap L^1_{\rm loc}$
be a tempered distribution of order $L$.
\begin{enumerate}
  \item[{\rm(i)}] For any $N\in\mathbb{N}\cap(n+1+2L,\infty)$
  and $\Phi\in C^N_{\rm c}$,
  \begin{align}\label{wavelet-e1}
\int_{\mathbb{R}^n}f(x)\Phi(x)\,dx=\sum_{I\in\mathcal{D}_+}
\left\langle f,\theta_{I,2}\right\rangle
\left\langle\Phi,\eta_{I,2}\right\rangle.
  \end{align}
  \item[{\rm(ii)}] If the order $N$ of a renormalized wavelet system
  $\{\psi_{\omega,n^*}\}_{\omega\in\Omega}$ satisfies
  $N>n+1+2L$, then \eqref{e1453}
holds in $\mathcal{S}'$
with $\{\langle f,\psi_{\omega,n}\rangle\}_{\omega\in\Omega}$
therein being defined by using the Lebesgue integrals
between $f$ and $\{\psi_{\omega,n}\}_{\omega\in\Omega}$.
\end{enumerate}
\end{proposition}

\begin{proof}
We first prove (i). To this end,
fix $N\in\mathbb{N}\cap(n+1+2L,\infty)$ and
$\Phi\in C_{\rm c}^N$. Choose
$\rho\in C_{\rm c}^\infty$ satisfying
$\mathop\mathrm{\,supp\,}(\rho)\subset B({\bf0},1)$
and $\int_{\mathbb{R}^n} \rho(x)\,dx=1$.
For any $k\in\mathbb{N}$, define $\rho_k(\cdot):=2^{kn}\rho(2^k\cdot)$ and
$\Phi_k:=\rho_k\ast \Phi\in C_{\rm c}^\infty$.
From this, Lemma \ref{crf}, and $f\in\mathcal{S}'\cap L^1_{\rm loc}$,
it follows that \eqref{wavelet-e1}
holds with $\Phi$ replaced by $\Phi_k$ for any $k\in\mathbb{N}$.
Moreover, using \cite[THEOREM 4.1(v)]{eg15}),
we find that, for any
$\alpha\in\mathbb{Z}_+^n$ with $|\alpha|\le N$,
\begin{align}\label{e1947}
\|\partial^\alpha \Phi-\partial^\alpha \Phi_k\|_{L^\infty}
\to 0
\end{align}
as $k\to \infty$.
Suppose $\mathop\mathrm{\,supp\,}(\Phi)\subset B({\bf 0},M)$
for some $M\in\mathbb{N}$. Then $\mathop\mathrm{\,supp\,}(\Phi_k)\subset B({\bf 0},M+1)$
for any $k\in\mathbb{N}$.
This, together with the assumption $f\in L^1_{\rm loc}$,
H\"older's inequality, and \eqref{e1947},
further implies that
\begin{align}\label{e2136}
\left|\int_{\mathbb{R}^n}f(x)
[\Phi(x)-\Phi_k(x)]\,dx\right|\le
\|f\|_{L^1(B({\bf 0},M+1))}\|\Phi-\Phi_k\|_{L^\infty}\to 0
\end{align}
as $k\to\infty$.

Now, we show that the series on
the right-hand side of \eqref{wavelet-e1} is absolutely convergent.
Indeed, fix $I:=2^{-j_I}[k_I+[0,1)^n]\in\mathcal{D}_+$ with
$j_I\in\mathbb{Z}_+$ and $k_I\in\mathbb{Z}^n$.
By the assumption that $f$ is a tempered
distribution of order $L$ and some simple calculations, we obtain
\begin{align}\label{e2011}
|\langle f,\theta_{I,2}\rangle|&\lesssim
\left\|\theta_{I,2}\right\|_{L}
=\sup_{\{\gamma\in\mathbb{Z}_+^n:|\gamma|\le L\}}
|I|^{-\frac12-\frac{|\gamma|}{n}}
\sup_{x\in\mathbb{R}^n}(1+|x|^{n+L+|\gamma|})
\left|\partial^\gamma\theta\left(2^{j_I}x-k_I\right)\right|\notag\\
&\lesssim|I|^{-\frac12-\frac Ln}
\sup_{x\in\mathbb{R}^n}\frac{(1+|x|)^{n+2L}}{(1+|2^{j_I}x-k_I|)^{n+2L+1}}
\lesssim|I|^{-\frac12-\frac Ln}\left(1+|2^{-j_I}k_I|\right)^{n+2L}.
\end{align}
On the other hand, observe that the following estimate holds:
for any $I\in\mathcal{D}_+\setminus\mathcal{D}_0$,
\begin{align}\label{es-diag}
\left|\left\langle \Phi,\eta_{I,2}\right\rangle\right|
=\left|\int_{\mathbb{R}^n}\Phi(x)\eta_{I,2}(x)\,dx\right|
\lesssim\|\Phi\|_{N}\|\eta\|_{N}\|I|^{\frac12}2^{-j_I(N-1)}
\frac{1}{(1+|2^{-j_I}k_I|)^{n+N-1}};
\end{align}
indeed, this is exactly \cite[Lemma 2.4]{ysy10}
with a weaker assumption on $\Phi$, however, via
carefully checking the proof of \cite[Lemma 2.4]{ysy10},
we find that \eqref{es-diag} also holds in our assumption on $\Phi$.
From this, \eqref{e2011}, the definition of $\|\cdot\|_{N}$,
the assumption  $\mathop\mathrm{\,supp\,}(\Phi)\subset B({\bf 0},M)$,
the fact that, for any
$j\in\mathbb{Z}$ and $\beta\in(n,\infty)$,
\begin{align}\label{wavelet-e3}
\sum_{k\in\mathbb{Z}^n}\frac{1}{(1+|2^{-j}k|)^{\beta}}
\lesssim 2^{jn}\sum_{k\in\mathbb{Z}^n}
\int_{2^{-j}[k+[0,1)^n]}\frac{1}{(1+|x|)^\beta}\,dx
=2^{jn}\int_{\mathbb{R}^n}\frac{1}{(1+|x|)^\beta}\,dx\sim2^{jn},
\end{align}
and the assumption $N>n+1+2L$, we deduce that
\begin{align}\label{small-dya}
&\sum_{I\in\mathcal{D}_+\setminus\mathcal{D}_0}
\left|\left\langle f,\theta_{I,2}\right\rangle
\left\langle \Phi,\eta_{I,2}\right\rangle\right|\notag\\
&\quad\lesssim\|\Phi\|_{N}\sum_{I\in\mathcal{D}_+\setminus
\mathcal{D}_0}|I|^{-\frac{L}{n}}2^{-j_I(N-1)}
\frac{1}{(1+|2^{-j_I}k_I|)^{N-1-2L}}\notag\\
&\quad\lesssim M^{n+2N}\sup_{\{\alpha\in\mathbb{Z}_+^n:|\alpha|\le N\}}
\left\|\partial^\alpha\Phi\right\|_{L^\infty}
\sum_{j\in\mathbb{N}}
2^{-j(N-1-L)}\sum_{k\in\mathbb{Z}^n}
\frac{1}{(1+|2^{-j}k|)^{N-1-2L}}\notag\\
&\quad\sim M^{n+2N}\sup_{\{\alpha\in\mathbb{Z}_+^n:|\alpha|\le N\}}
\left\|\partial^\alpha\Phi\right\|_{L^\infty}
\sum_{j\in\mathbb{N}}2^{-j(N-1-n-L)}
\sim\ M^{n+2N}\sup_{\{\alpha\in\mathbb{Z}_+^n:|\alpha|\le N\}}
\left\|\partial^\alpha\Phi\right\|_{L^\infty}.
\end{align}
In addition, if $I\in\mathcal{D}_0$, then,
by the assumption $\mathop\mathrm{\,supp\,}(\Phi)\subset B({\bf 0},M)$,
we easily find that, when $|2^{-j_I}k_I|> 2M$,
$|\langle \Phi,\eta_{I,2}\rangle|\lesssim \|\Phi\|_{L^\infty}
\frac{1}{(1+|2^{-j_I}k_I|)^{2n+2L+1}}$
and, when $|2^{-j_I}k_I|\le 2M$,
$|\langle \Phi,\eta_{I,2}\rangle|\lesssim \|\Phi\|_{L^\infty}$.
These, combined with \eqref{e2011} and \eqref{wavelet-e3} again,
further imply that
\begin{align*}
&\sum_{I\in\mathcal{D}_0}\left|
\left\langle f,\theta_{I,2}\right\rangle
\left\langle \Phi,\eta_{I,2}\right\rangle\right|\\
&\quad\lesssim\|\Phi\|_{L^\infty}
\left[\sum_{\{k\in\mathbb{Z}^n:|k|>2M\}}\frac{1}{(1+|k|)^{n+1}}
+\sum_{\{k\in\mathbb{Z}^n:|k|\le2M\}}(1+|k|)^{n+2L}\right]
\lesssim M^{2(n+L)}\left\|\Phi\right\|_{L^\infty}.
\end{align*}
Applying this and \eqref{small-dya},
we conclude that
\begin{align}\label{e2125}
\sum_{I\in\mathcal{D}_+}\left|
\left\langle f,\theta_{I,2}\right\rangle
\left\langle \Phi,\eta_{I,2}\right\rangle\right|
\lesssim M^{2(N+n+L)}\sup_{\{\alpha\in\mathbb{Z}_+^n:|\alpha|\le N\}}
\left\|\partial^\alpha\Phi\right\|_{L^\infty}.
\end{align}
Therefore, the series on
the right-hand side of \eqref{wavelet-e1} is absolutely convergent.

Finally, we prove that \eqref{wavelet-e1} holds.
Indeed, for any $k\in\mathbb{N}$,
using \eqref{e2125} with $\Phi$ replaced by
$\Phi-\Phi_k$ and the fact that
$\mathop\mathrm{\,supp\,}(\Phi-\Phi_k)\subset B({\bf 0},M+1)$,
we obtain
\begin{align*}
&\left| \sum_{I\in\mathcal{D}_+}
\left\langle f,\theta_{I,2}\right\rangle
\left\langle \Phi,\eta_{I,2}\right\rangle
- \sum_{I\in\mathcal{D}_+}
\left\langle f,\theta_{I,2}\right\rangle
\left\langle \Phi_k,\eta_{I,2}\right\rangle\right|
\lesssim M^{2(N+n+L)}\sup_{\{\alpha\in\mathbb{Z}_+^n:|\alpha|\le N\}}
\left\|\partial^\alpha\Phi-\partial^\alpha\Phi_k\right\|_{L^\infty}.
\end{align*}
From this, the fact that \eqref{wavelet-e1}
holds with $\Phi$ replaced by $\Phi_k$ for any $k\in\mathbb{N}$,
\eqref{e1947}, and \eqref{e2136}, it follows that \eqref{wavelet-e1}
also holds for any such $\Phi$.
This finishes the proof of (i).

Next, we show (ii).
For any $\phi\in\mathcal{S}$, consider the series
\begin{align*}
S:=&\,\sum_{e\in E}\sum_{Q,R\in\mathcal{D}_+}
\left\langle f,\theta_{R,2}\right\rangle \left\langle
\psi^e_{Q,n},\eta_{R,2}\right\rangle
\left\langle \psi^e_{Q,n^\ast},\phi\right\rangle
+\sum_{Q\in\mathcal{D}_0}\sum_{R\in\mathcal{D}_+}\left\langle f,\theta_{R,2}\right\rangle
\left\langle \psi^{\bf 0}_{Q,n},\eta_{R,2}
\right\rangle\left\langle \psi^{\bf 0}_{Q,n^\ast},\phi\right\rangle.
\end{align*}
Repeating the proof \eqref{e2125}, we conclude that the series $S$
is also absolutely convergent.
Then, applying both an argument similar to that used in the proof of
\cite[(4.6)]{bhyy24} and (i), we obtain \eqref{e1453} holds in $\mathcal{S}'$,
which completes the proof of (ii) and hence Proposition
\ref{wavelet}.
\end{proof}

We also require the following growth property of $A_1$-weights.

\begin{lemma}\label{l1922}
If $\upsilon\in A_1$, then, for almost every $x\in\mathbb{R}^n$,
$\upsilon(x)\gtrsim\frac{\upsilon(B({\bf0},1))}{[\upsilon]_{A_1}^2
(1+|x|)^n}$,
where the implicit positive constant depends only on $n$.
\end{lemma}

\begin{proof}
Since $\upsilon\in A_1$, from Lemma \ref{ApProperty}(i),
it follows that, for almost every $x\in\mathbb{R}^n$,
there exists $y\in\mathbb{R}^n$ such that $x\in B(y,1)$ and hence
\begin{align}\label{l1922e1}
\upsilon(B(y,1))\lesssim M(\upsilon)(x)\le[\upsilon]_{A_1} \upsilon(x).
\end{align}
Fix $x$ satisfying this estimate and
let $k\in\mathbb{Z}_+$ be such that $k\in[\log_2(1+|y|),1+\log_2(1+|y|))$.
Then $B({\bf 0},1)\subset B(y,2^k)$. By this and
Lemma \ref{ApProperty}(ii), we obtain
\begin{align*}
\upsilon(B({\bf 0},1))&\le \upsilon\left(B\left(y,2^k\right)\right)\le[\upsilon]_{A_1}
2^{kn}\upsilon(B(y,1))\lesssim [\upsilon]_{A_1}(1+|y|)^n\upsilon(B(y,1))\\
&\lesssim[\upsilon]_{A_1}(1+|x|)^n\upsilon(B(y,1)),
\end{align*}
which, together with \eqref{l1922e1},
further implies the desired conclusion
and hence completes the proof of Lemma \ref{l1922}.
\end{proof}

Moreover, the following exquisite geometrical properties about shifted
dyadic grids (see, for instance, \cite[p.\,479]{mtt02},
\cite[Section 2.2]{lsu2012}, or \cite[Lemma 3.2.26]{hvvw16})
are needed and play a key role in subsequent
sections.

\begin{lemma}\label{2115}
For any $\alpha\in\{0,\frac{1}{3},\frac{2}{3}\}^n$,
let $\mathcal{D}^\alpha$ be the same as in \eqref{dyadic}.
Then the following properties hold.
\begin{enumerate}
\item[\textup{(i)}]
For any $Q,P\in\mathcal{D}^\alpha$ with
$\alpha\in\{0,\frac{1}{3},\frac{2}{3}\}^n$,
$Q\cap P\in\{\emptyset,Q,P\}${\rm;}
\item[$\mathrm{(ii)}$]
For any cube $P\in\mathcal{Q}$,
there exist $\alpha\in\{0,\frac13,\frac23\}^n$
and $Q\in\mathcal{D}^\alpha$
such that $P\subset Q$ and $l(Q)\in(\frac{3}{2}l(P),3l(P)]$.
\end{enumerate}
\end{lemma}

Based on this, we obtain the following
geometrical lemma, which is also
important for the proof of Theorem \ref{p2143}
and is of independent interest.

\begin{lemma}\label{lem-dom}
Let $K\in(0,\infty)$, $\mathcal{G}:=\{Q_j\}_{j}$ be a collection
of disjoint cubes with the same edge length,
and $\mathcal{S}:=\{KQ\}_{Q\in\mathcal{G}}$.
If, for any $Q\in\mathcal{S}$, define
\begin{align*}
{\rm Dom}(Q):=\left\{P\in\bigcup_{\alpha\in\{0,\frac13,\frac23\}^n}
\mathcal{D}^\alpha:P\supset Q,\ l(P)\in\left(\frac32l(Q),3l(Q)\right]\right\},
\end{align*}
then ${\rm Dom}(\mathcal{S}):=\bigcup_{Q\in\mathcal{S}}
{\rm Dom}(Q)$
is a nonempty set of cubes with the same edge length and, for any
$P\in{\rm Dom}(\mathcal{S})$,
$\sharp\{Q\in\mathcal{S}:P\in{\rm Dom}(Q)\}\le 3^nK^n$.
\end{lemma}

\begin{proof}
By Lemma \ref{2115}(ii), we find that,
for any $Q\in\mathcal{S}$, the set
${\rm Dom}(Q)$ is nonempty. On the other hand,
denote by $l$
the edge length of cubes in $\mathcal{S}$.
Choose the unique $j\in\mathbb{Z}$ satisfying
$2^{-j}\in(\frac32l,3l]$. Then every cube in ${\rm Dom}(\mathcal{S})$
has the edge length $2^{-j}$.
Next, we fix $P\in{\rm Dom}(\mathcal{S})$ and prove that
$$\sharp\left\{Q\in\mathcal{S}:P\in{\rm Dom}(Q)\right\}\le 3^nK^n.$$
Indeed, from the definition of $\mathcal{S}$, we
deduce that $\{\frac1K Q:Q\in\mathcal{S},\ P\in{\rm Dom}(Q)\}$
is a collection of disjoint cubes contained in $P$.
This implies that
\begin{align*}
3^nl^n\ge|P|\ge \sum_{\{Q\in\mathcal{S}:P\in{\rm Dom}(Q)\}}
\left|\frac1K Q\right|=\frac{l^n}{K^n}\sharp
\left\{Q\in\mathcal{S}:P\in{\rm Dom}(Q)\right\}.
\end{align*}
Therefore, $\sharp\{Q\in\mathcal{S}:P\in{\rm Dom}(Q)\}\le 3^nK^n$,
which then completes the proof of Lemma \ref{lem-dom}.
\end{proof}

Now, we are ready to prove Theorem \ref{p2143}.

\begin{proof}[Proof of Theorem \ref{p2143}]
Fix $f\in W^{1,1}_\upsilon$.
Using H\"older's inequality, Lemma \ref{l1922}, and
the definition of $\|\cdot\|_0$, we
conclude that, for any $\phi\in\mathcal{S}$,
\begin{align}\label{e1959}
\left|\int_{\mathbb{R}^n}f(x)\phi(x)\,dx\right|\le \|f\|_{L^1_\upsilon}
\left\|\upsilon^{-1}\phi\right\|_{L^\infty}\lesssim
\|f\|_{L^1_\upsilon}
\left\|(1+|\cdot|)^n\phi\right\|_{L^\infty}
\le\|f\|_{L^1_\upsilon}\|\phi\|_{0}.
\end{align}
Thus, $f$ is a tempered distribution of order $0$.
By this and Proposition \ref{wavelet}, we obtain \eqref{e1453} holds in $\mathcal{S}'$.

Now, we show the lower estimate of \eqref{e2026}.
Assume that, for any $i\in\{0,1\}$,
$\mathop\mathrm{\,supp\,}(\psi^i)\subset [-M,M]$
for some $M\in\mathbb{N}$.
Then it is easy to conclude that, for any $e\in E$ and $Q\in\mathcal{D}_+\setminus\mathcal{D}_0$
or $e\in\{0,1\}^n$ and $Q\in\mathcal{D}_0$,
$\mathop\mathrm{\,supp\,}(\psi^e_{Q,n})\subset 3MQ$.
For any $j\in\mathbb{Z}_+$, define $\mathcal{S}_j:=\{3MQ\}_{Q\in\mathcal{D}_j}$.
Choose the unique $j_M\in\mathbb{N}$ such that $2^{j_M}\in(\frac92M,9M]$.
Then, by the definition, we conclude that, for any $j\in\mathbb{Z}_+$,
the edge length of every cube in ${\rm Dom}(\mathcal{S}_j)$
is $2^{-j+j_M}$, where ${\rm Dom}(\mathcal{S}_j)$ is defined as in
Lemma \ref{lem-dom}. Applying Lemma \ref{2115}(ii) and H\"older's
inequality,
we find that, for any $Q\in\mathcal{D}_0$ and $e\in\{0,1\}^n$, there exists
$P_Q\in{\rm Dom}(\mathcal{S}_0)$ such that
$$\left|\left\langle f,\psi_{Q,n}^e\right\rangle\right|
\le \|f\|_{L^1(3MQ)}\left\|\psi_{Q,n}^e\right\|_{L^\infty}
\lesssim\|f\|_{L^1(P_Q)}.$$
From this, Tonelli's theorem, and Lemmas \ref{lem-dom}, \ref{ApProperty}(iii),
and \ref{2115}(i),
we deduce that, for any $\lambda\in(0,\infty)$,
\begin{align}\label{wavelet-d0}
&\lambda\sum_{e\in\{0,1\}^n}
\sum_{\{Q\in\mathcal{D}_0:|\langle f,\psi^e_{Q,n}\rangle|>\lambda|Q|^\beta\}}
|Q|^{\beta-1}\upsilon(Q)\notag\\
&\quad\le \sum_{e\in\{0,1\}^n}\sum_{Q\in\mathcal{D}_0}
\left|\left\langle f,\psi^e_{Q,n}\right\rangle\right|
\frac{\upsilon(Q)}{|Q|}
\lesssim\sum_{Q\in\mathcal{S}_0}\sum_{P\in{\rm Dom}(Q)}
\|f\|_{L^1(P)}\frac{\upsilon(P)}{|P|}\notag\\
&\quad=\sum_{P\in{\rm Dom}(\mathcal{S}_0)}
\sum_{\{Q\in\mathcal{S}_0:P\in{\rm Dom}(Q)\}}\|f\|_{L^1(P)}
\frac{\upsilon(P)}{|P|}\notag\\
&\quad\lesssim[\upsilon]_{A_1}\sum_{\alpha\in\{0,\frac13,\frac23\}^n}
\sum_{P\in\mathcal{D}^\alpha_{-j_M}}\int_{P}|f(x)|\upsilon(x)\,dx
\lesssim[\upsilon]_{A_1}\|f\|_{L^1_\upsilon}.
\end{align}
On the other hand, fixing $\lambda\in(0,\infty)$ and
using Lemma \ref{2115}(ii) again,
\eqref{e1937}, and H\"older's inequality,
we conclude that,
for any $j\in\mathbb{N}$,
$Q\in\mathcal{D}_j$ such that $\sup_{e\in E}
|\langle f,\psi^e_{Q,n}\rangle|>\lambda|Q|^\beta$,
there exists $P_Q\in{\rm Dom}(\mathcal{S}_j)$
such that, for any $e\in E$,
\begin{align*}
\left|\left\langle f,\psi^e_{Q,n}\right\rangle\right|
&=\left|\left\langle f-f_{P_Q},\psi^e_{Q,n}\right\rangle\right|
\le \left\|f-f_{P_Q}\right\|_{L^1(P_Q)}
\left\|\psi^e_{Q,n}\right\|_{L^\infty}\\
&\sim\left|P_Q\right|^
{-\frac1n}\left\|f-f_{P_Q}\right\|_{L^1(P_Q)}\le \omega_{P_Q}(f)
\end{align*}
with the positive equivalence constant depending only on $M$ and $n$;
therefore, there exists a positive constant $C_0$,
depending only on $M$, $n$, and $\beta$, such that $\omega_{P_Q}(f)>C_0\lambda|P_Q|^\beta$.
This, combined with Tonelli's theorem, Lemma \ref{lem-dom}, and
Theorem \ref{thm-cddd}(i) via letting
$p=1$ and replacing $\mathcal{D}$ therein
by $\{\mathcal{D}^\alpha\}_{\alpha\in\{0,\frac13,\frac23\}^n}$,
further implies that, for any $\lambda\in(0,\infty)$,
\begin{align*}
&\lambda\sum_{e\in E}
\sum_{\{Q\in\mathcal{D}_+\setminus
\mathcal{D}_0:|\langle f,\psi^e_{Q,n}\rangle|>
\lambda|Q|^\beta\}}|Q|^{\beta-1}\upsilon(Q)\\
&\quad\lesssim\lambda
\sum_{\{Q\in\mathcal{D}_+\setminus
\mathcal{D}_0:\sup_{e\in E}|\langle f,\psi^e_{Q,n}\rangle|
>\lambda|Q|^\beta\}}|Q|^{\beta-1}\upsilon(Q)\\
&\quad\lesssim\lambda\sum_{j\in\mathbb{N}}\sum_{Q\in\mathcal{S}_j}
\sum_{\{P\in{\rm Dom}(Q):\omega_P(f)>C_0\lambda|P|^\beta\}}
|P|^{\beta-1}\upsilon(P)\\
&\quad=\lambda\sum_{\alpha\in\{0,\frac13,\frac23\}^n}\sum_{j\in\mathbb{N}}
\sum_{P\in\mathcal{D}^\alpha_{C_0\lambda,\beta}[f]\cap
\mathcal{D}^\alpha_{j-j_M}}\sum_{\{Q\in\mathcal{D}^\alpha:P\in{\rm Dom}(3MQ)\}}
|P|^{\beta-1}\upsilon(P)\\
&\quad\lesssim\lambda\sum_{\alpha\in\{0,\frac13,\frac23\}^n}
\sum_{P\in\mathcal{D}^\alpha_{C_0\lambda,\beta}[f]}|P|^{\beta-1}\upsilon(P)
\lesssim[\upsilon]_{A_1}\|\,|\nabla f|\,\|_{L^1_\upsilon}.
\end{align*}
Combining this and \eqref{wavelet-d0}, we obtain
the lower estimate of \eqref{e2026}.

Finally, it remains to prove the upper estimate of
\eqref{e2026}. Without loss of generality,
we may assume that the rightmost-hand side of \eqref{e2026}
is finite.
For any $e\in E$ and $Q\in\mathcal{D}_+\setminus\mathcal{D}_0$
or $e\in\{0,1\}^n$ and $Q\in\mathcal{D}_0$,
using the definition of $\psi^e_{Q,n^*}$, the facts that
$\mathop\mathrm{\,supp\,}(\psi^e_{Q,n^*})\subset 3MQ$
and that $l(Q)\le 1$, and
Lemma \ref{ApProperty}(ii), we find that
\begin{align}\label{wavelet-e2}
\left\|\psi^e_{Q,n^*}\right\|_{W^{1,1}_\upsilon}
&\le|Q|^{-\frac{1}{n^*}}\upsilon(3MQ)\left\|\psi^e\right\|_{L^\infty}
+|Q|^{-\frac{1}{n^*}-\frac1n}\upsilon(3MQ)
\left\|\,\left|\nabla \psi^e\right|\,\right\|_{L^\infty}\notag\\
&\lesssim\frac{\upsilon(3MQ)}{|Q|}
\lesssim[\upsilon]_{A_1}\frac{\upsilon(Q)}{|Q|}.
\end{align}
If $\{\langle f,\psi_{\omega,n}\rangle\}_{\omega\in\Omega}$ is a finite sequence, that is,
only finite elements in $\{\langle f,\psi_{\omega,n}\rangle\}_{\omega\in\Omega}$ is nonzero,
then, by the triangle inequality, the Cauchy--Schwarz inequality,
and \eqref{wavelet-e2}, we conclude that
\begin{align}\label{wavelet-e4}
\left\|\sum_{\omega\in\Omega}\left\langle f,\psi_{\omega,n}\right\rangle
\psi_{\omega,n^*}\right\|_{W^{1,1}_{\upsilon}}
&\le\sum_{\omega\in\Omega}\left|\left\langle f,\psi_{\omega,n}\right\rangle\right|
\left\|\psi_{\omega,n^*}\right\|_{W^{1,1}_{\upsilon}}
\lesssim [\upsilon]_{A_1}\sum_{\omega\in\Omega}
\left|\left\langle f,\psi_{\omega,n}\right\rangle\right|
\frac{\upsilon(I_\omega)}{|I_\omega|}\notag\\
&=[\upsilon]_{A_1}\left\|\left\{
\frac{\left\langle f,\psi_{\omega,n}\right\rangle}{|I_\omega|^\beta}\right\}
_{\omega\in\Omega}\right\|_{\ell^1_{\beta,\upsilon}}.
\end{align}
For general sequence $f^{\rm w}$,
since $W^{1,1}_{\upsilon}$ is complete
(see, for instance, \cite[THEOREM 3.3]{af03}) and $W^{1,1}_\upsilon
\hookrightarrow\mathcal{S}'$ [which is an immediate
 consequence of \eqref{e1959}], from
\eqref{wavelet-e4} together with both a
standard limiting argument and \eqref{e1453}, we easily deduce the
desired upper estimate. This then finishes the proof of Theorem \ref{p2143}.
\end{proof}

Our second aim of the present section is to apply
Theorem \ref{p2143} to establish
the sharp real interpolation between the critical weighted
Sobolev space and weighted Besov spaces (that is, Theorem
\ref{t953}).
To this end, we first recall some basic concepts and
useful tools on
Besov spaces and real interpolation spaces.
Let $\theta_0,\theta_1\in\mathcal{S}$ be,
respectively, as in \eqref{ass-eta0}
and, for any $j\in\mathbb{N}\cap[2,\infty)$, define
$\theta_j(\cdot)=2^{jn}\theta_1(2^j\cdot)$.
For any $s\in\mathbb{R}$, $p\in[1,\infty)$, and $\upsilon\in A_p$,
the \emph{weighted Besov space} $B^{s,\upsilon}_{p,p}$ is defined to be the set of all
$f\in\mathcal{S}'$ such that
$\|f\|_{B^{s,\upsilon}_{p,p}}:=
(\sum_{j\in\mathbb{Z}_+}2^{jsp}
\|\theta_j\ast f\|_{L^p_\upsilon})^{\frac1p}<\infty;
$ we refer to, for example,
\cite{b82,bbd20,bd17,s18,Tr83} for its real-variable theory.
In particular, one can characterize
weighted Besov spaces
in terms of the Daubechies wavelet system, that is,
for any $s\in\mathbb{R}$, $p\in[1,\infty),$ and $\upsilon\in A_p$
$f\in B^{s,\upsilon}_{p,p}$ if and only if
\eqref{e1453} holds in $\mathcal{S}'$ and
\begin{align}\label{e2028}
\|f\|_{B^{s,\upsilon}_{p,p}}\sim\left\|\left\{
\frac{\langle f,\psi_{\omega,n}\rangle}{|I_\omega|^\beta}\right\}
_{\omega\in\Omega}\right\|_{\ell_{\beta,\upsilon}^{p}}<\infty,
\end{align}
where $\beta:=1+\frac{(s-1)p}{(p-1)n}$ and,
for any $\omega\in\Omega$, $\langle f,\psi_{\omega,n}\rangle$
is defined as in the right-hand side of \eqref{wavelet-e1}
with $\Phi$ therein replaced by $\psi_{\omega,n}$;
see, for instance, \cite[Corollary 10.3]{r03}.

Now, we recall some concepts on real
interpolation spaces (see, for example, \cite{bs1988,bl76}).
Let $(X_0,X_1)$ be a couple of Banach spaces that are both
continuously embedded into the same Banach space $\mathcal{X}$.
The \emph{space} $X_0+X_1$ is defined to be the set of all $f\in\mathcal{X}$
which can be written as $g+h$ with $g\in X_0$ and $h\in X_1$.
For any $f\in X_0+X_1$, we define its \emph{norm} in $X_0+X_1$ by setting
$\|f\|_{X_0+X_1}:=\inf_{f=g+h}\|g\|_{X_0}+\|h\|_{X_1}$
and the \emph{K-functional} $K(f,t;X_0,X_1)$ for any $t\in(0,\infty)$
by setting
$K(f,t;X_0,X_1):=\inf_{f=g+h}\|g\|_{X_0}+t\|h\|_{X_1}.$
Next, we define the real interpolation space by using $K$-functionals.
To be precise, if $\theta\in(0,1)$ and $q\in(0,\infty]$, then the
\emph{real interpolation space} $(X_0,X_1)_{\theta,q}$ is defined to be the set of all
$f\in X_0+X_1$ such that
\begin{align*}
\|f\|_{(X_0,_1)_{\theta,q}}:=
\begin{cases}
\displaystyle\left\{\int_{0}^{\infty}\left[t^{-\theta}K(f,t;X_0,X_1)\right]^q\,\frac{dt}{t}
\right\}^{\frac1q} & \mbox{if } q\in(0,\infty), \\
\displaystyle\sup_{t\in(0,\infty)}t^{-\theta}K(f,t;X_0,X_1) & \mbox{if }q=\infty
\end{cases}
\end{align*}
is finite. Moreover, it always holds the interpolation inequality
\begin{align}\label{e1126}
\|f\|_{(X_0,X_1)_{\theta,q}}\le \|f\|^{1-\theta}_{X_0}\|f\|_{X_1}^{\theta}
\end{align}
for any $f\in X_0\cap X_1$ (see, for instance, \cite[p.\,49]{bl76}).
In Section \ref{sec-gn}, we apply this
inequality to obtain a sharp Gagliardo--Nirenberg
inequality in various Sobolev type spaces.

Now, we are ready to prove Theorem \ref{t953}.
\begin{proof}[Proof of Theorem \ref{t953}]
Let $\beta:=1+\frac{(s-1)p}{(p-1)n}$. Notice that $\beta\in(-\infty,1-\frac1n)\cup
(1,\infty)$ due to $s\in (-\infty,\frac1p)\cup(1,\infty)$.
In Theorem \ref{p2143}, for any function $f\in W^{1,1}_\upsilon$,
we define its wavelet coefficients
$\{\langle f,\psi_{\omega,n}\rangle\}_{\omega\in\Omega}$
by using Lebesgue integrals.
Applying Proposition \ref{wavelet}(i),
we find that this way coincides with
that used in \eqref{e2028}.
Therefore, we can define the analysis operator
\begin{align*}
T: \left\{\begin{array}{rll}
W^{1,1}_{\upsilon}+B^{s,\upsilon}_{p,p}
&\to& w\ell^1_{\beta,\upsilon}
+\ell_{\beta,\upsilon}^{p},\\
f&\mapsto&{\displaystyle
\left\{\frac{\langle f,\psi_{\omega,n}\rangle}{|I_\omega|^\beta}\right\}
_{\omega\in\Omega}}
\end{array}\right.
\end{align*}
and the synthesis operator
\begin{align*}
S:\left\{
\begin{array}{rll}
{\displaystyle\ell^1_{\beta,\upsilon}
+\ell_{\beta,\upsilon}^{p}}&\to&{\displaystyle W^{1,1}_{\upsilon}+B^{s,\upsilon}_{p,p},}\\
{\displaystyle\left\{a_\omega\right\}_{\omega\in\Omega}}
&\mapsto&{\displaystyle \sum_{\omega\in\Omega}\left|I_\omega\right|^\beta
a_\omega\psi_{\omega,n}},
\end{array}\right.
\end{align*}
where $\{\langle f,\psi_{\omega,n}\rangle\}_{\omega\in\Omega}$
is defined as in
\eqref{e2028}.
From Theorem \ref{p2143} and \eqref{e2028},
it follows that both $T$ and $S$ are well
defined and $S\circ T={\rm Id}$, where
${\rm Id}$ denotes the identity mapping on
$W^{1,1}_{\upsilon}+B^{s,\upsilon}_{p,p}$.
Using \cite[p.\,113, Theorem 5.3.1]{bl76}
with the measure space therein replaced by
$(\Omega,\{|I_\omega|^{\beta-1}\upsilon(I_\omega)\}_{\omega\in\Omega})$,
we obtain
\begin{align*}
\left(w\ell^1_{\beta,\upsilon}, \ell_{\beta,\upsilon}^{p}\right)_{\theta,q}
=\left(\ell^1_{\beta,\upsilon}, \ell_{\beta,\upsilon}^{p}\right)_{\theta,q}
=\ell_{\beta,\upsilon}^{q}.
\end{align*}
Then, applying Theorem \ref{p2143} and \eqref{e2028} again, we conclude that
$S:\ell_{\beta,\upsilon}^{q}\to
(W^{1,1}_{\upsilon},B^{s,\upsilon}_{p,p})_{\theta,q}$
and $T:(W^{1,1}_{\upsilon},B^{s,\upsilon}_{p,p})_{\theta,q}
\to\ell_{\beta,\upsilon}^{q}$
are both bounded. By these
and the conditions that $\sigma=1+\theta(s-1)$ and $\frac1q=
1-\theta(1-\frac1p)$, we find that
$\frac{(s-1)p}{p-1}=\frac{(\sigma-1)q}{q-1}$ and hence,
for any $f\in B^{\sigma,\upsilon}_{q,q}$,
\begin{align*}
\|f\|_{(W^{1,1}_{\upsilon},B^{s,\upsilon}_{p,p})_{\theta,q}}
=\|S\circ Tf\|_{(W^{1,1}_{\upsilon},B^{s,\upsilon}_{p,p})_{\theta,q}}
\lesssim \left\|\left\{\frac{\langle f,\psi_{\omega,n}\rangle}{|I_\omega|^\beta}
\right\}_{\omega\in\Omega}\right\|_{\ell_{\beta,\upsilon}^q}
\sim\|f\|_{B^{\sigma,\upsilon}_{q,q}}
\end{align*}
and, for any $f\in (W^{1,1}_{\upsilon},B^{s,\upsilon}_{p,p})_{\theta,q}$,
\begin{align*}
\|f\|_{B^{\sigma,\upsilon}_{q,q}}
\sim\left\|\left\{\frac{\langle f,\psi_{\omega,n}\rangle}
{|I_\omega|^\beta}\right\}_{\omega\in\Omega}
\right\|_{\ell_{\beta,\upsilon}^{q}}
=\left\|Tf\right\|_{\ell_{\beta,\upsilon}^{q}}
\lesssim\|f\|_{(W^{1,1}_{\upsilon},B^{s,\upsilon}_{p,p})_{\theta,q}}.
\end{align*}
Thus, the desired conclusion holds,
which then completes the proof of Theorem \ref{t953}.
\end{proof}

\section{BSVY Formula in Ball Banach Function Spaces}\label{section2}

In this section,
we apply Theorem \ref{thm-cddd}
to establish the BSVY formula in ball Banach function spaces (see Theorem
\ref{xgamma<0}).
First, we recall some basic concepts about ball Banach function spaces.
The following definition of
ball (quasi-)Banach function spaces was
introduced in \cite[Definition 2.1]{shyy2017}.

\begin{definition}\label{1659}
A quasi-Banach space $X\subset\mathscr{M}$,
equipped with a quasi-norm $\|\cdot\|_X$
which makes sense for all the functions
in $\mathscr{M}$,
is called a \emph{ball quasi-Banach function space} if $X$ satisfies that
\begin{enumerate}
\item[\textup{(i)}]
for any $f\in\mathscr{M}$,
if $\|f\|_X=0$, then $f=0$ almost everywhere;
\item[\textup{(ii)}]
if $f,g\in\mathscr{M}$
with $|g|\leq|f|$ almost everywhere,
then $\|g\|_X\leq\|f\|_X$;
\item[\textup{(iii)}]
if a sequence $\{f_m\}_{m\in\mathbb{N}}\subset\mathscr{M}$
and $f\in\mathscr{M}$ satisfy
that $0\leq f_m\uparrow f$ almost everywhere
as $m\to\infty$, then $\|f_m\|_X\uparrow\|f\|_X$ as $m\to\infty$;
\item[\textup{(iv)}]
for any ball $B:=B(x,r)$ with $x\in\mathbb{R}^n$ and $r\in(0,\infty)$,
$\mathbf{1}_B\in X$.
\end{enumerate}
Moreover, a ball quasi-Banach function space $X$ is called a
\emph{ball Banach function space} if
$X$ satisfies the following additional conditions:
\begin{enumerate}
\item[\textup{(v)}]
for any $f,g\in X$,
$\|f+g\|_X\leq\|f\|_X+\|g\|_X$;
\item[\textup{(vi)}]
for any ball $B$,
there exists a positive constant $C_{(B)}$, depending on $B$,
such that, for any $f\in X$,
$
\int_B|f(x)|\,dx\leq C_{(B)}\|f\|_X.
$
\end{enumerate}
\end{definition}

\begin{remark}\label{rm-bqbf}
\begin{enumerate}
\item[{\rm(i)}] Observe that,
  in Definition \ref{1659}(iv),
  if we replace any ball $B$
  by any bounded measurable set
  $E$, we obtain an equivalent formulation
  of ball quasi-Banach function spaces.
\item[{\rm(ii)}]
Let $X$ be a ball quasi-Banach function space.
Then, by the definition, we can easily conclude
that, for any $f\in\mathscr{M}$,
$\|f\|_X=0$ if and only if $f=0$ almost everywhere
(see also \cite[Proposition 1.2.16]{lyh2320}).
\item[{\rm(iii)}]Applying both (ii) and (iii)
of Definition \ref{1659}, we find that any ball quasi-Banach
function space $X$ has Fatou's property, that is,
for any sequence $\{f_k\}_{k\in\mathbb{N}}$ in $X$,
\begin{align*}
\left\|\liminf_{k\to\infty}\left|f_k\right|\right\|_X
\leq\liminf_{k\to\infty}\left\|f_k\right\|_X
\end{align*}
(see also \cite[Lemma 2.4]{wyy.arxiv}).
\item[{\rm(iv)}] From \cite[Proposition 1.2.36]{lyh2320}
(see also \cite[Theorem~2]{dfmn2021}),
we deduce that every ball quasi-Banach function space
is complete.
\item[{\rm(v)}] Recall that a quasi-Banach space
$X\subset\mathscr{M}$
is called a \emph{quasi-Banach function space}
if it is a ball quasi-Banach function space
and it satisfies Definition \ref{1659}(iv)
with ball therein replaced by any
measurable set of \emph{finite measure}.
Moreover, a \emph{Banach function space}
is a quasi-Banach function space satisfying both
(v) and (vi) of Definition \ref{1659}
with ball therein replaced by any measurable set
of \emph{finite measure}, which was originally
introduced in
\cite[Chapter 1, Definitions 1.1 and 1.3]{bs1988}.
It is easy to show that every quasi-Banach
function space (resp. Banach function space)
is a ball quasi-Banach
function space (resp. ball Banach function space),
and the converse is not necessary to be true.
Several examples about ball (quasi-)Banach function spaces
are given in Section \ref{S5}.

\item[${\rm (vi)}$]
In Definition~\ref{1659},
if we replace (iv)
by the following \emph{saturation property}:
\begin{enumerate}
\item[]
for any measurable set $E\subset\mathbb{R}^n$
of positive measure, there exists a measurable set $F\subset E$
of positive measure satisfying that $\mathbf{1}_F\in X$,
\end{enumerate}
then we obtain the definition of quasi-Banach function spaces
in Lorist and Nieraeth
\cite{ln23}.
Moreover, by \cite[Proposition~2.5]{zyy23}
(see also \cite[Proposition~4.22]{n23}),
we find that, if the
quasi-normed vector space $X$ satisfies
the additional assumption that
the Hardy--Littlewood maximal operator is weakly bounded on
its convexification,
then the quasi-Banach function space in \cite{ln23} coincides with
the ball quasi-Banach function space. Thus,
under this additional assumption,
working with ball quasi-Banach function
spaces in the sense of Definition~\ref{1659}
or quasi-Banach function spaces in
the sense of \cite{ln23} would yield
exactly the same results.
\end{enumerate}
\end{remark}

The following definition of the $p$-convexification of a ball
quasi-Banach function space
is precisely \cite[Definition 2.6]{shyy2017}.

\begin{definition}
Let $X$ be a ball quasi-Banach function space and $p\in(0,\infty)$.
The \emph{$p$-convexification} $X^p$ of $X$ is defined by setting
$X^p:=\{f\in\mathscr{M}:|f|^p\in X\}$
equipped with the quasi-norm $\|f\|_{X^p}:=\|\,|f|^p\|_{X}^\frac{1}{p}$
for any $f\in X^p$.
\end{definition}

Moreover, the following concept
of the associate space of a ball
Banach function space can be found in \cite[p.\,9]{shyy2017};
see \cite[Chapter 1, Section 2]{bs1988} for more details.

\begin{definition}
Let $X$ be a ball Banach function space.
The \emph{associate space} (also called the \emph{K\"othe dual}) $X'$ of $X$
is defined by setting
\begin{align*}
X':=\left\{f\in\mathscr{M}:
\|f\|_{X'}:=\sup_{\{g\in X:\|g\|_{X}=1\}}
\left\|fg\right\|_{L^1}<\infty\right\}.
\end{align*}
\end{definition}

\begin{remark}
Let $X$ be a ball Banach function space.
Then, using \cite[Proposition 2.3]{shyy2017},
we find that $X'$ is also a ball Banach function space.
\end{remark}

We now recall the concept of absolutely continuous
norms of ball Banach function spaces;
see, for instance, \cite[Chapter 1, Definition 3.1]{bs1988}
or \cite[Definition 3.2]{wyy2020}.

\begin{definition}
A ball Banach function space $X$ is said
to have an \emph{absolutely continuous norm}
if, for any $f\in X$ and any sequence $\{E_j\}_{j\in\mathbb{N}}$
of measurable sets
satisfying that $\mathbf{1}_{E_j}\to0$
almost everywhere as $j\to\infty$, one has
$\|f\mathbf{1}_{E_j}\|_X\to0$ as $j\to\infty$.
\end{definition}

To establish the
BSVY formula, we also need the following
concept of the endpoint boundedness of the
Hardy--Littlewood maximal operator $M$, which
can be found in \cite[Definition 2.14]{pyyz2023}.

\begin{definition}\label{def-epbd}
Let $X$ be a ball Banach function space.
Then $M$
is said to be \emph{endpoint bounded on $X'$}
if there exists a sequence
$\{\theta_m\}_{m\in{\mathbb N}}\subset(0,1)$
satisfying $\lim_{m\to\infty}\theta_m=1$ such that,
for any $m\in{\mathbb N}$,
$X^{\frac{1}{\theta_m}}$ is a ball Banach function space,
$M$ is bounded on $(X^{\frac{1}{\theta_m}})'$, and
$\lim_{m\to\infty}
\|M\|_{(X^{\frac{1}{\theta_m}})'\to(X^{\frac{1}{\theta_m}})'}
<\infty.$
\end{definition}

Next, we recall the
concepts of homogeneous and inhomogeneous ball Banach Sobolev spaces as follows,
which were originally introduced, respectively, in \cite[Definition 2.4]{dlyyz.arxiv}
and \cite[Definition 2.6]{dgpyyz2022}.

\begin{definition}
Let $X$ be a ball Banach function space.
\begin{enumerate}
  \item[{\rm(i)}] The \emph{homogeneous ball Banach
Sobolev space} $\dot{W}^{1,X}$
is defined to be the set of all
distributions $f$ on $\mathbb{R}^n$ such that $|\nabla f|\in X$ equipped
with the seminorm
$\|f\|_{\dot{W}^{1,X}}:=\|\,|\nabla f|\,\|_X$.
  \item[{\rm(ii)}] The \emph{inhomogeneous ball Banach Sobolev space}
$W^{1,X}$ is defined to be the set
of all $f\in X$ such that
$|\nabla f|\in X$ equipped
with the norm
$\|f\|_{W^{1,X}}:=\|f\|_X+\|\,|\nabla f|\,\|_X.$
\end{enumerate}
\end{definition}

\begin{remark}\label{loc}
Recall that the space $W^{1,1}_{\loc}$
is defined to be the set of all $f\in L^1_{\loc}$
such that $f\phi\in W^{1,1}$
for any $\phi\in C_{\rm c}^\infty$.
Then, as showed in \cite[Proposition 2.5]{dlyyz.arxiv},
for any ball Banach function space $X$,
$\dot{W}^{1,X}\subset W^{1,1}_{\loc}$.
\end{remark}

In what follows, for any $p\in [1,\infty)$
and $q\in (0,\infty)$, let
\begin{align*}\Gamma_{p,q}: =
\begin{cases}
(-\infty,-q)\cup (0,\infty) & \mbox{if}\ p = 1,\\
\mathbb{R}\setminus\{0\} & \mbox{if}\ p\in (1,\infty).
\end{cases}\end{align*}
Furthermore, for any given $r\in(0,\infty)$,
the \emph{centered ball average operator}
$\mathcal{B}_r$ is defined by setting,
for any $f\in L_{{\loc}}^1$ and $x\in\mathbb{R}^n$,
$
\mathcal{B}_r(f)(x):=\frac{1}{|B(x,r)|}\int_{B(x,r)}|f(y)|\,dy.
$
Then we have the following BSVY formula.

\begin{theorem}\label{xgamma<0}
Let $p\in [1,\infty)$, $q\in(0,\infty)$, $\gamma\in \mathbb{R}\setminus\{0\}$,
and $X$ be a ball Banach function space
having an absolutely continuous norm.
Assume that
both $n(\frac{1}{p} - \frac{1}{q})<1$ and $\gamma\in \Gamma_{p,q}$
or there exists $r\in (n,\infty)$ such that the Hardy--Littlewood
maximal operator $M$ is bounded on $X^{\frac{1}{r}}$.
Assume that
$X^{\frac1p}$ is a ball Banach function space
and that either of the following assumptions holds:
\begin{enumerate}
  \item[{\rm(a)}] $M$ is bounded on $(X^{\frac1p})'$;
  \item[{\rm(b)}] $p=1$, $M$ is endpoint bounded on $X'$,
and the centered ball average
operators $\{\mathcal{B}_r\}_{r\in(0,\infty)}$
are uniformly bounded on $X$.
\end{enumerate}
Then, for any $f\in \dot{W}^{1,X}$,
\begin{align}\label{xgamma<0e1}
\sup_{\lambda\in(0,\infty)}\lambda
\left\|\left[\int_{\mathbb{R}^n}
{\bf 1}_{E_{\lambda,\frac{\gamma}{q}}[f]}(\cdot,y)
|\cdot-y|^{\gamma-n}\,dy\right]^{\frac1q}\right\|_X
\sim\left\|\,\left|\nabla f\right|\,\right\|_X
\end{align}
with the positive equivalence constants
independent of $f$, where $E_{\lambda,\frac\gamma q}[f]$
is as in \eqref{Elambda}.
\end{theorem}

\begin{remark}\label{rem411}
\begin{enumerate}
\item[${\rm (i)}$] Let $X: = L^p$.
In this case, Theorem \ref{xgamma<0} when $p=q$ reduces to
the well-known BSVY formula obtained in \cite{bsvy.arxiv}
(see also \cite[Theorems 3 and 4]{bsvy}), which
when $p = 1$, $q \neq p$, $n\in \mathbb{N}\cap[2,\infty)$, and $\gamma\in (-\infty,-1)$
is still new.

\item[${\rm (ii)}$] Theorem \ref{xgamma<0}
gives an affirmative answer to the question in \cite[Remark 3.17(iv)]{zyy23-1}
by removing the restriction $n=1$ in assumption (d) of
\cite[Theorem 3.29]{zyy23-1}.
Moreover,
Theorem \ref{xgamma<0}
improves \cite[Theorems 3.29(i) and 3.35(i)]{zyy23-1} by expanding
the range of $q$ in the assumptions (a) and (d)
of \cite[Theorems 3.29 and 3.35]{zyy23-1} from $q\in(0,p]$ to
$q\in(0,\infty)$ with $n(\frac1p-\frac1q)<1$.
\item[{\rm(iii)}] As pointed out in
\cite[Remark 3.13(iii)]{zyy23-1}, the
assumption $n(\frac1p-\frac1q)<1$
in Theorem \ref{xgamma<0} is sharp
(see also \cite[Remark 1.2(iii)]{hlyyz25}).

\item[${\rm (iv)}$] In Theorem \ref{xgamma<0},
the assumption $\gamma\in \Gamma_{p,q}$
is sharp in the case where $X: = L^p$
and $q: = p$; see the details in \cite[Corollary 1.6]{bsvy.arxiv} for the sharpness
of $\gamma\in \Gamma_{p,q}$
with $p\in (1,\infty)$ and \cite[Proposition 6.1]{bsvy.arxiv} for
the sharpness of $\gamma\in \Gamma_{1,q}$.
\end{enumerate}
\end{remark}

We first show the following upper
estimate of the BSVY formula.

\begin{proposition}\label{upper}
Let $p\in [1,\infty)$, $q\in(0,\infty)$
satisfy $n(\frac1p- \frac1q)<1$,
and $\gamma\in\Gamma_{p,q}$.
Assume that $X$ is a ball Banach function space
satisfying that $M$ is bounded on $(X^{\frac1p})'$.
Then there exist a positive constant $C$,
independent of $\|M\|_{(X^{\frac1p})'\to (X^\frac1p)'}$,
and an increasing continuous function $\phi$
on $[0,\infty)$ such that,
for any $f\in C^\infty$
with $|\nabla f|\in C_{\rm c}$,
\begin{align}\label{uppere1}
\sup_{\lambda\in(0,\infty)}\lambda
\left\|\left[\int_{\mathbb{R}^n}
{\bf 1}_{E_{\lambda,\frac{\gamma}{q}}[f]}(\cdot,y)
|\cdot-y|^{\gamma-n}\,dy\right]^{\frac1q}\right\|_X
\le C\phi\left(\left\|M\right\|_{(X^{\frac1p})'\to
(X^{\frac1p})'}\right)
\left\|\,\left|\nabla f\right|\,\right\|_X.
\end{align}
Moreover, if $X$ has an absolutely continuous norm, then
\eqref{uppere1} holds for any
$f\in \dot{W}^{1,X}$.
\end{proposition}

To prove this proposition, we apply
the famous telescope method. Thus, we
need the following inequality;
see, for instance, \cite[p.\,59]{klv21}.

\begin{lemma}\label{Poin}
Let $M\in(0,\infty)$ and $f\in L^1_{\loc}$.
Then there exist a set $A\subset\mathbb{R}^n$ with $|A|=0$
and a positive constant $C_{(M,n)}$, depending only on
both $M$ and $n$, such that,
for any $x\in\mathbb{R}^n\setminus A$, any $r\in(0,\infty)$,
and any ball $B_1\subset B:=B(x,r)\subset MB_1$,
\begin{align*}
\left|f(x)-f_{B_1}\right|
\le C_{(M,n)}\sum_{j\in\mathbb{Z}_+}\fint_{2^{-j}B}
\left|f(y)-f_{2^{-j}B}\right|\,dy.
\end{align*}
\end{lemma}

We now use Lemma \ref{Poin} to establish
the following two key estimates of level sets.
In what follows, for any $x,y\in\mathbb{R}^n$,
let $B_{x,y}:=B(y,\frac{1}{20}|x-y|)$
and, for any $\lambda\in(0,\infty)$,
$b\in\mathbb{R}$, and $f\in L^1_{\loc}$, let
\begin{align}\label{df-e1}
E^{(1)}_{\lambda,b}[f]:=
\left\{(x,y)\in\mathbb{R}^n\times\mathbb{R}^n
:x\ne y,\ \frac{|f(x)-f_{B_{x,y}}|}{|x-y|^{1+b}}>\lambda\right\}.
\end{align}

\begin{proposition}\label{point}
Let $p\in[1,\infty)$, $q,\varepsilon,\lambda\in(0,\infty)$,
$\beta\in\mathbb{R}\setminus\{\frac1p\}$, and
$\upsilon\in L^1_{\loc}$ be nonnegative.
Then the following two statements hold.
\begin{enumerate}
\item[$\rm (i)$] If $q\in [p,\infty)$,
then there exists a positive constant $C_{(n,\beta,p,q)}$,
depending only on $n$, $\beta$, $p$, and $q$,
such that, for any $f\in L^1_{\loc}$,
\begin{align*}
&\int_{\mathbb{R}^n}
\left[\int_{\mathbb{R}^n}{\bf 1}_{
E^{(1)}_{\lambda,n(\beta-\frac1p)}[f]}(x,y)
|x-y|^{qn(\beta-\frac1p)-n}\,dy\right]^{\frac pq}
\upsilon(x)\,dx\notag\\
&\quad\le C_{(n,\beta,p,q)}
\sum_{j\in\mathbb{Z}_+}2^{jn(\beta p-1)}
\sum_{\alpha\in\{0,\frac13,\frac23\}^n}
\sum_{Q\in\mathcal{D}^\alpha_{\lambda(j),\beta+1-\frac1p}[f]}
|Q|^{\beta p-1}\upsilon(Q),
\end{align*}
where, for any
$j\in\mathbb{Z}_+$,
$\lambda(j):=c_{(n,\beta,p,\varepsilon)}\lambda
2^{j[1+n(\beta-\frac1p)-\varepsilon]}$
with a
positive constant $c_{(n,\beta,p,\varepsilon)}$
depending only on $n$, $\beta$, $p$, and $\varepsilon$.
\item[$\rm (ii)$] If $q\in (0,p)$, then there exists a positive constant
$\widetilde C:=C_{(n,\beta,p,q)}$,
depending only on $n$, $\beta$, $p$, and $q$,
such that, for any
$f\in C^\infty$ with
$|\nabla f|\in C_{\rm c}$.
\begin{align}\label{pointe7+}
&\int_{\mathbb{R}^n}
\left[\int_{\mathbb{R}^n}{\bf 1}_{
E^{(1)}_{\lambda,n(\beta-\frac1p)}[f]}(x,y)
|x-y|^{qn(\beta-\frac1p)-n}\,dy\right]^{\frac pq}
\upsilon(x)\,dx\notag\\
&\quad\le \widetilde C
\left\{\sum_{j\in\mathbb{Z}_+}2^{jqn(\beta-\frac1p)}
\sum_{\alpha\in\{0,\frac13,\frac23\}^n}\left[
\sum_{Q\in\mathcal{D}^\alpha_{\lambda(j),\beta+1-\frac1p}[f]}
|Q|^{\beta p-1}\upsilon(Q)\right]^{\frac qp}
\right\}^{\frac{p}{q}},
\end{align}
where, for any
$j\in\mathbb{Z}_+$,
$\lambda(j):=c_{(n,\beta,p,\varepsilon)}\lambda
2^{j[1+n(\beta-\frac1p)-\varepsilon]}$
with a
positive constant $c_{(n,\beta,p,\varepsilon)}$
depending only on $n$, $\beta$, $p$, and $\varepsilon$.
\end{enumerate}
\end{proposition}

\begin{proof}
For any given $f\in L^1_{\loc}$,
let the set $A$ be the same as in Lemma \ref{Poin}.
We first estimate
\begin{align*}
\int_{\mathbb{R}^n}{\bf 1}_{E^{(1)}_{\lambda,n(\beta-\frac1p)}[f]}
(\cdot,y)|\cdot-y|^{qn(\beta-\frac1p)-n}\,dy.
\end{align*}
To this end, for any $(x,y)\in E^{(1)}_{\lambda,n(\beta-
\frac1p)}[f]
\setminus (A\times\mathbb{R}^n)$,
using Lemma \ref{Poin} with $B_1:=B_{x,y}$,
$B:=B(x,2|x-y|)$, and $M:=60$, we find that
there exists a positive constant $C_{(n)}$,
depending only on $n$, such that
\begin{align}\label{pointe2}
\lambda|x-y|^{1+n(\beta-\frac1p)}
&<\left|f(x)-f_{B_{x,y}}\right|
\le C_{(n)}\sum_{j\in\mathbb{Z}_+}
\fint_{B(x,2^{-j+1}|x-y|)}
\left|f(z)-f_{B(x,2^{-j+1}|x-y|)}\right|\,dz.
\end{align}
In what follows,
let $c_0:=\frac{1-2^{-\varepsilon}}{C_{(n)}}$.
We claim that, for any $(x,y)\in E^{(1)}_{\lambda,n(\beta-\frac1p)}[f]
\setminus (A\times\mathbb{R}^n)$,
there exists $j_{x,y}\in\mathbb{Z}_+$ such that
\begin{align*}
c_0\lambda 2^{-j_{x,y}\varepsilon}|x-y|^{1+n(\beta-\frac1p)}
\le \fint_{B(x,2^{-j_{x,y}+1}|x-y|)}
\left|f(z)-f_{B(x,2^{-j_{x,y}+1}|x-y|)}\right|\,dz.
\end{align*}
Otherwise, it holds that
\begin{align*}
\frac{\lambda}{C_{(n)}}|x-y|^{1+n(\beta-\frac1p)}
&=c_0\lambda|x-y|^{1+n(\beta-\frac1p)}
\sum_{j\in\mathbb{Z}_+}2^{-j\varepsilon}\\
&>\sum_{j\in\mathbb{Z}_+}\fint_{B(x,2^{-j+1}|x-y|)}
\left|f(z)-f_{B(x,2^{-j+1}|x-y|)}\right|\,dz,
\end{align*}
which contradicts \eqref{pointe2}. This then finishes
the proof of the above claim.

In addition, from Lemma \ref{2115}(ii),
we infer that there exists a positive constant $\widetilde{C}_{(n)}$,
depending only on $n$,
such that, for any $(x,y)\in E^{(1)}_{\lambda,n(\beta
-\frac1p)}[f]\setminus (A\times\mathbb{R}^n)$,
there exist $\alpha_{x,y}\in \{0,\frac13,\frac23\}^n$
and $Q_{x,y}\in \mathcal{D}^{\alpha_{x,y}}$ satisfying that
$B(x,2^{-j_{x,y}+1}|x-y|)\subset Q_{x,y}\subset
B(x,2^{-j_{x,y}+1}\widetilde{C}_{(n)}|x-y|).$
Applying this and the above claim,
we conclude that, for any $(x,y)\in E^{(1)}_{\lambda,
n(\beta-\frac1p)}[f]
\setminus (A\times\mathbb{R}^n)$,
\begin{align*}
&2^{j_{x,y}[1+n(\beta-\frac1p)-\varepsilon]}\lambda
\left|Q_{x,y}\right|^{\beta-\frac1p+\frac1n}\\
&\quad=\lambda 2^{-j_{x,y}\varepsilon}
\left|2^{j_{x,y}}Q_{x,y}\right|^{\beta-\frac1p+\frac1n}
\sim\lambda 2^{-j_{x,y}\varepsilon}|x-y|^{n(\beta-
\frac1p)+1}\notag\\
&\quad\lesssim\fint_{B(x,2^{-j_{x,y}+1}|x-y|)}
\left|f(z)-f_{B(x,2^{-j_{x,y}+1}|x-y|)}\right|\,dz\\
&\quad\le 2\inf_{c\in\mathbb{C}}
\fint_{B(x,2^{-j_{x,y}+1}|x-y|)}
\left|f(z)-c\right|\,dz\lesssim\inf_{c\in\mathbb{C}}
\fint_{Q_{x,y}}
\left|f(z)-c\right|\,dz\le\fint_{Q_{x,y}}
\left|f(z)-f_{Q_{x,y}}\right|\,dz\\
&\quad=\left|Q_{x,y}\right|^{-2}\int_{Q_{x,y}}
\left|\int_{Q_{x,y}}
\left[f(z)-f(w)\right]\,dw\right|\,dz\notag\\
&\quad\le \left|Q_{x,y}\right|^{-1+\frac1n}
\left[\left|Q_{x,y}\right|^{-1-\frac1n}\int_{Q_{x,y}}
\int_{Q_{x,y}}|f(z)-f(w)|\,dw\,dz\right]
=\left|Q_{x,y}\right|^{-1+\frac1n}\omega_{Q_{x,y}}(f),
\end{align*}
which further implies that there exists a positive constant
$c_{(n,\beta,p,\varepsilon)}$,
depending only on $n$, $\beta$, $p$, and $\varepsilon$,
such that $Q_{x,y}\in \mathcal{D}^{\alpha_{x,y}}
_{\lambda(j_{x,y}),\beta+1-\frac1p}[f]$
with $\lambda(j_{x,y}):=c_{(n,\beta,p,\varepsilon)}
\lambda2^{j_{x,y}[1+n(\beta-\frac1p)-\varepsilon]}$ and
$\mathcal{D}^{\alpha_{x,y}}
_{\lambda(j_{x,y}),\beta+1-\frac1p}[f]$ defined the same as in
\eqref{df-ab} with $\alpha: = \alpha_{x,y}$,
$\lambda: = \lambda(j_{x,y})$,
and $b: = \beta+1-\frac1n$.
Furthermore, notice that, for any $(x,y)\in
E^{(1)}_{\lambda,n(\beta-\frac1p)}[f]
\setminus (A\times\mathbb{R}^n)$,
\begin{align}\label{pointe4}
(x,y)\in\left[2^{-j_{x,y}}B(x,2|x-y|)\right]
\times B(x,2|x-y|)\subset
Q_{x,y}\times 2^{j_{x,y}}Q_{x,y}
\end{align}
and
\begin{align}\label{pointe5}
|x-y|\sim\left|2^{j_{x,y}}Q_{x,y}\right|^{\frac1n},
\end{align}
where the positive equivalence constants depend only on $n$.
By \eqref{pointe4}, we obtain
\begin{align*}
E^{(1)}_{\lambda,n(\beta-\frac1p)}[f]
\setminus \left(A\times\mathbb{R}^n\right)
\subset\bigcup_{j\in\mathbb{Z}_+}
\bigcup_{\alpha\in\{0,\frac13,\frac23\}^n}
\bigcup_{Q\in\mathcal{D}^{\alpha}_{\lambda(j),\beta+1-\frac1p}[f]}
\left(Q\times 2^{j}Q\right).
\end{align*}
From this and \eqref{pointe5}, it further follows that,
for any $x\in\mathbb{R}^n\setminus A$,
\begin{align}\label{pointe6}
&\int_{\mathbb{R}^n}{\bf 1}_{E^{(1)}_{\lambda,n(\beta-\frac1p)}[f]}
(x,y)|x-y|^{qn(\beta-\frac1p)-n}\,dy\notag\\
&\quad\lesssim\sum_{j\in\mathbb{Z}_+}
\sum_{\alpha\in\{0,\frac13,\frac23\}^n}
\sum_{Q\in\mathcal{D}^{\alpha}_{\lambda(j),\beta+1-
\frac1p}[f]}
\int_{2^jQ}\left|2^jQ\right|^{q(\beta-\frac1p)-1}
\,dy{\bf 1}_{Q}(x)\notag\\
&\quad=\sum_{j\in\mathbb{Z}_+}
\sum_{\alpha\in\{0,\frac13,\frac23\}^n}
\sum_{Q\in\mathcal{D}^{\alpha}_{\lambda(j),\beta+1-\frac1p}[f]}
\left|2^jQ\right|^{q(\beta-\frac1p)}{\bf 1}_{Q}(x),
\end{align}
where the implicit positive constant depends only on $n$, $\beta$,
$p$, and $q$.

Next, we prove (i). Assume that $q\in [p,\infty)$.
Using this, \eqref{pointe6}, the fact that, for any $r\in(0,1]$
and $\{a_j\}_{j\in\mathbb{N}}\subset\mathbb{C}$,
$(\sum_{j\in\mathbb{N}}|a_j|)^r
\le\sum_{j\in\mathbb{N}}|a_j|^r,$
and Tonelli's theorem,
we find that
\begin{align}\label{eq2105}
&\int_{\mathbb{R}^n}
\left[\int_{\mathbb{R}^n}{\bf 1}_{E^{(1)}_
{\lambda,n(\beta-\frac1p)}}(x,y)|x-y|^{qn(\beta-\frac1p)-n}\,dy\right]
^{\frac pq}\upsilon(x)\,dx\notag\\
&\quad\lesssim\int_{\mathbb{R}^n}
\left\{\sum_{j\in\mathbb{Z}_+}\sum_{\alpha\in
\{0,\frac13,\frac23\}^n}\sum_{Q\in\mathcal{D}^{\alpha}
_{\lambda(j),\beta+1-\frac1p}[f]}
\left|2^jQ\right|^{q(\beta-\frac1p)}{\bf 1}_Q(x)\right\}
^{\frac pq}\upsilon(x)\,dx\notag\\
&\quad\le\sum_{j\in\mathbb{Z}_+}2^{jn(\beta p-1)}
\sum_{\alpha\in\{0,\frac13,\frac23\}^n}\sum_{Q\in\mathcal{D}^{\alpha}
_{\lambda(j),\beta+1-\frac1p}[f]}|Q|^{\beta p-1}\upsilon(Q),
\end{align}
where the implicit positive constant depends only on $n$, $\beta$,
$p$, and $q$. This shows (i).

Now, we prove (ii). Assume that $q\in (0,p)$.
By this, \eqref{pointe6},
and Minkowski's inequality, we find that
\begin{align}\label{gamma<0e1-z}
&\left\{\int_{\mathbb{R}^n}
\left[\int_{\mathbb{R}^n}{\bf 1}_{E^{(1)}_
{\lambda,n(\beta-\frac1p)}}(x,y)|x-y|^{qn(\beta-
\frac1p)-n}\,dy\right]
^{\frac pq}\upsilon(x)\,dx \right\}^{\frac qp}\notag\\
&\quad\le\sum_{j\in\mathbb{Z}_+}\sum_{\alpha\in
\{0,\frac13,\frac23\}^n}\left\{\int_{\mathbb{R}^n}
\left[\sum_{Q\in\mathcal{D}^{\alpha}
_{\lambda(j),\beta+1-\frac1p}[f]}
\left|2^jQ\right|^{q(\beta-\frac1p)}
{\bf 1}_Q(x)\right]^{\frac pq}\upsilon(x)\,dx\right\}
^{\frac qp}.
\end{align}
From this and Lemma \ref{lem-sd}
with $r: = q$ and $\lambda: = \lambda(j)$,
we deduce that, for any $x\in\mathbb{R}^n$,
\begin{align*}
\left[\sum_{Q\in\mathcal{D}^{\alpha}
_{\lambda(j),\beta+1-\frac1p}[f]}
\left|2^jQ\right|^{q(\beta-\frac1p)}
{\bf 1}_Q(x)\right]^{\frac pq}
&\lesssim2^{jn(\beta p -1)}\left|Q_x\right|^{\beta p-1}\\
&\le2^{jn(\beta p -1)}
\sum_{Q\in\mathcal{D}^{\alpha}_{\lambda(j),\beta
+1-\frac1p}[f]}
|Q|^{\beta p-1}{\bf 1}_Q(x),
\end{align*}
where $Q_x$ is the same as in Lemma \ref{lem-sd}.
This, combined with \eqref{gamma<0e1-z}, further implies that
\eqref{pointe7+} and hence (ii) holds.
This then finishes the proof of Proposition \ref{point}.
\end{proof}

On the other hand,
for any $\lambda\in(0,\infty)$,
$b\in\mathbb{R}$, and $f\in L^1_{\loc}$,
let
\begin{align}\label{df-e2}
E^{(2)}_{\lambda,b}[f]:=
\left\{(x,y)\in\mathbb{R}^n\times \mathbb{R}^n
:x\ne y,\ \frac{|f(y)-f_{B_{x,y}}|}{|x-y|^{1+b}}>\lambda\right\}.
\end{align}
We now turn to prove the following estimate of
the level set $E^{(2)}_{\lambda,b}[f]$.

\begin{proposition}\label{point2}
Let $p\in [1,\infty)$, $q,\varepsilon,\lambda\in(0,\infty)$,
$\beta\in \mathbb{R}\setminus\{\frac1p\}$,
$\upsilon\in A_1$, and $f\in C^\infty$ with
$|\nabla f|\in C_{\rm c}$.
Then the following two statements hold.
\begin{enumerate}
\item[$\rm (i)$] If $q\in [p,\infty)$,
then there exists a positive constant $\widetilde C:=C_{(n,\beta,p,q)}$,
depending only on $n$, $\beta$, $p$, and $q$,
such that
\begin{align*}
&\int_{\mathbb{R}^n}
\left[\int_{\mathbb{R}^n}{\bf 1}_{
E^{(2)}_{\lambda,n(\beta-\frac1p)}[f]}(x,y)
|x-y|^{qn(\beta-\frac1p)-n}\,dy\right]^{\frac pq}
\upsilon(x)\,dx\notag\\
&\quad\le \widetilde C[\upsilon]_{A_1}
\sum_{j\in\mathbb{Z}_+}2^{jn(\beta p-1)}
\sum_{\alpha\in\{0,\frac13,\frac23\}^n}
\sum_{Q\in\mathcal{D}^\alpha_{\lambda(j),\beta+1-\frac1p}[f]}
|Q|^{\beta p-1}\upsilon(Q),
\end{align*}
where, for any
$j\in\mathbb{Z}_+$,
$\lambda(j):=c_{(n,\beta,q,\varepsilon)}\lambda
2^{j[1+n(\beta-\frac1p)-\varepsilon]}$
with a
positive constant $c_{(n,\beta,q,\varepsilon)}$
depending only on $n$, $\beta$, $q$, and $\varepsilon$.
\item[$\rm (ii)$] If $q\in (0,p)$, then there exists a positive constant
$\widetilde C:=C_{(n,\beta,p,q)}$,
depending only on $n$, $\beta$, $p$, and $q$,
such that
\begin{align*}
&\int_{\mathbb{R}^n}
\left[\int_{\mathbb{R}^n}{\bf 1}_{
E^{(2)}_{\lambda,n(\beta-\frac1p)}[f]}(x,y)
|x-y|^{qn(\beta-\frac1p)-n}\,dy\right]^{\frac pq}
\upsilon(x)\,dx\notag\\
&\quad\le \widetilde C
[\upsilon]_{A_1}^{\frac{p}{q}}
\left\{\sum_{j\in\mathbb{Z}_+}2^{jqn(\beta-\frac1p)}
\sum_{\alpha\in\{0,\frac13,\frac23\}^n}\left[
\sum_{Q\in\mathcal{D}^\alpha_{\lambda(j),\beta+1-\frac1p}[f]}
|Q|^{\beta p-1}\upsilon(Q)\right]^{\frac qp} \right\}^{\frac{p}{q}},
\end{align*}
where, for any
$j\in\mathbb{Z}_+$,
$\lambda(j):=c_{(n,\beta,q,\varepsilon)}\lambda
2^{j[1+n(\beta-\frac1p)-\varepsilon]}$
with a
positive constant $c_{(n,\beta,q,\varepsilon)}$
depending only on $n$, $\beta$, $q$, and $\varepsilon$.
\end{enumerate}
\end{proposition}

To establish this estimate,
we need several
auxiliary lemmas.
The following one about
weighted Lebesgue spaces is important; see, for instance,
\cite[Section 7.1]{shyy2017} and \cite[Remarks 2.7 and 3.4]{wyy2020}.

\begin{lemma}\label{Lpv}
Let $p\in[1,\infty)$ and $\upsilon\in
A_p$. Then
\begin{enumerate}
  \item[{\rm(i)}] $L^p_\upsilon$
  is a ball Banach function space having an absolutely
  continuous norm;
  \item[{\rm(ii)}] for any $s\in(0,\infty)$,
  $(L^p_{\upsilon})^{s}=
  L^{ps}_\upsilon$;
  \item[{\rm(iii)}] $(L^p_\upsilon)'
  =L^{p'}_{\upsilon^{1-p'}}$.
\end{enumerate}
\end{lemma}

To show Proposition \ref{point2}, we also need
the following two technical lemmas about extrapolation,
the first one is precisely \cite[Lemma 4.6]{dlyyz.arxiv}
and the second one can be found in \cite[Lemma 3.3]{zyy23-1}.

\begin{lemma}\label{lem-extrap}
Let $X$ be a ball Banach function space
such that the Hardy--Littlewood maximal operator
$M$ is bounded on $X$.
For any $g\in X$ and $x\in\mathbb{R}^n$,
let $R_{X}g(x) :
= \sum_{k=0}^{\infty}\frac{M^k g(x)}{2^k \|M\|^k_{
X\to X}},$
where, for any $k\in\mathbb{N}$, $M^k$ is the $k$-fold iteration of $M$
and $M^0g(x): = |g(x)|$. Then,
for any $g\in X$,
\begin{enumerate}
\item[\rm (i)] for any $x\in \mathbb{R}^n$, $|g(x)| \le
R_{X}g(x)$;
\item[\rm (ii)] $R_{X}
g\in A_1$ and $[R_{X}g]_{A_1} \le 2
\|M\|_{X\to X}$;
\item[\rm (iii)] $\|R_{X}g\|_{X} \le 2\|g\|_
{X}$.
\end{enumerate}
\end{lemma}

\begin{lemma}\label{4.6}
Let $X$ be a ball Banach function space and $p\in[1,\infty)$.
Assume that $X^\frac{1}{p}$ is a ball Banach function space and
$M$ is bounded on $Y:=(X^\frac{1}{p})'$.
Then, for any $f\in X$,
$$
\|f\|_X\leq\sup_{\{g\in Y:\|g\|_{Y}\leq1\}}
\left[\int_{\mathbb{R}^n}
\left|f(x)\right|^pR_{Y}g(x)\,dx\right]^\frac{1}{p}
\leq2^\frac{1}{p}\|f\|_X.
$$
\end{lemma}

\begin{proof}[Proof of Proposition \ref{point2}]
For any given $f\in C^\infty$
with $|\nabla f|\in C^\infty_{\rm c}$,
let the set $A$ be the same as in Lemma \ref{Poin}.
Repeating an argument used in the proof of
Proposition \ref{point} with $x$ and $y$
replaced, respectively, by $y$ and $x$, we obtain
\begin{align*}
E^{(2)}_{\lambda,
n(\beta-\frac1p)}[f]\setminus\left(\mathbb{R}^n
\times A\right)
\subset\bigcup_{j\in\mathbb{Z}_+}
\bigcup_{\alpha\in\{0,\frac13,\frac23\}^n}
\bigcup_{Q\in\mathcal{D}^{\alpha}_{\lambda(j),\beta+1-\frac1p}[f]}
\left(2^jQ\times Q\right)
\end{align*}
and hence, for any $x\in\mathbb{R}^n$,
\begin{align}\label{pointe22}
&\int_{\mathbb{R}^n}{\bf 1}_{E^{(2)}_{\lambda,
n(\beta-\frac1p)}}(x,y)|x-y|^{qn(\beta-\frac1p)-n}\,dy\notag\\
&\quad\lesssim\sum_{j\in\mathbb{Z}_+}
\sum_{\alpha\in\{0,\frac13,\frac23\}^n}
\sum_{Q\in\mathcal{D}^{\alpha}_{\lambda(j),\beta+1-\frac1p}[f]}
\int_{Q}\left|2^jQ\right|^{q(\beta-\frac1p)-1}\,dy{\bf 1}_{2^jQ}(x)\notag\\
&\quad=\sum_{j\in\mathbb{Z}_+}
\sum_{\alpha\in\{0,\frac13,\frac23\}^n}
\sum_{Q\in\mathcal{D}^{\alpha}_{\lambda(j),\beta+1-\frac1p}[f]}
\left|2^jQ\right|^{q(\beta-\frac1p)-1}|Q|{\bf 1}_{2^jQ}(x)
\end{align}
with the implicit positive constant depending only on $n$,
$\beta$, $p$, and $q$.

Next, we prove (i). Assume that $q\in [p,\infty)$.
From this, \eqref{pointe22}, and Lemma \ref{ApProperty}(ii), similar to
the estimate \eqref{eq2105}, we easily infer that
(i) holds.

Now, we show (ii). Assume that $q\in (0,p)$.
Applying this, \eqref{pointe22},
and Minkowski's inequality, we conclude that
\begin{align}\label{gamma<22e1-z}
&{\rm I}: =\left\{\int_{\mathbb{R}^n}
\left[\int_{\mathbb{R}^n}{\bf 1}_{E^{(2)}_
{\lambda,n(\beta-\frac1p)}}(x,y)|x-y|^{qn(\beta-\frac1p)-n}\,dy\right]
^{\frac pq}\upsilon(x)\,dx \right\}^{\frac qp}\notag\\
&\quad\lesssim\sum_{j\in\mathbb{Z}_+}\sum_{\alpha\in
\{0,\frac13,\frac23\}^n}\left\{\int_{\mathbb{R}^n}
\left[\sum_{Q\in\mathcal{D}^{\alpha}
_{\lambda(j),\beta+1-\frac1p}[f]}
\left|2^jQ\right|^{q(\beta-\frac1p)-1}|Q|{\bf 1}_{2^jQ}(x)
\right]^{\frac pq}\upsilon(x)\,dx\right\}^{\frac qp}.
\end{align}
For any given $j\in\mathbb{Z}_+$ and $\alpha\in
\{0,\frac13,\frac23\}^n$ and for any $Q\in\mathcal{D}^{\alpha}
_{\lambda(j),\beta+1-\frac1p}[f]$, let $J_Q: = |2^jQ|^{q(\beta-\frac1p)-1}|Q|$.
Moreover, let $r: = \frac{p}{q}\in (1,\infty)$ and $\mu: =
\upsilon^{1-r'}$. Choose $g\in L^{r'}_\mu$
such that $\|g\|_{L^{r'}_\mu} = 1$.
Using this, Tonelli's theorem,
Lemma \ref{lem-extrap} with
$X:=L^{r'}_{\mu}$,
(ii), (v), and
(iv) of Lemma \ref{ApProperty}, and H\"{o}lder's
inequality, we find that
\begin{align*}
& \left|\int_{\mathbb{R}^n}
\sum_{Q\in\mathcal{D}^{\alpha}
_{\lambda(j),\beta+1-\frac1p}[f]}J_Q {\bf 1}_{2^jQ}(x) g(x)\,dx\right|\\
& \quad\le \sum_{Q\in\mathcal{D}^{\alpha}
_{\lambda(j),\beta+1-\frac1p}[f]}J_Q
R_{L^{r'}_\mu}g\left(2^jQ\right) \le 2^{jn}
\left[R_{L^{r'}_\mu}g\right]_{A_1}
\sum_{Q\in\mathcal{D}^{\alpha}
_{\lambda(j),\beta+1-\frac1p}[f]}J_Q
R_{L^{r'}_\mu}g(Q)\\
&\quad \lesssim 2^{jn}\|M\|_{L^{r'}_\mu
\to L^{r'}_\mu}\int_{\mathbb{R}^n}
\sum_{Q\in\mathcal{D}^{\alpha}
_{\lambda(j),\beta+1-\frac1p}[f]}J_Q {\bf 1}_{Q}(x)
R_{L^{r'}_\mu}g(x)\,dx\\
&\quad \le 2^{jn}[\mu]_{A_{r'}}^{\frac{1}{r'-1}}
\left\{\int_{\mathbb{R}^n}\left[
\sum_{Q\in\mathcal{D}^{\alpha}
_{\lambda(j),\beta+1-\frac1p}[f]}J_Q {\bf 1}_{Q}(x)
\right]^{r}\upsilon(x)\,dx\right\}^{\frac{1}{r}}
\left\|R_{L^{r'}_\mu}g
\right\|_{L^{r'}_\mu}\\
&\quad \lesssim 2^{jqn(\beta-\frac1p)}[\upsilon]_{A_1}
\left\{\int_{\mathbb{R}^n}\left[
\sum_{Q\in\mathcal{D}^{\alpha}
_{\lambda(j),\beta+1-\frac1p}[f]}|Q|^{q(\beta-\frac1p)}{\bf 1}_{Q}(x)
\right]^{\frac{p}{q}}\upsilon(x)\,dx\right\}^{\frac qp}.
\end{align*}
From this, Lemma \ref{Lpv}(iii), and \eqref{gamma<22e1-z},
it follows that
\begin{align*}
{\rm I} \lesssim [\upsilon]_{A_1}
\sum_{j\in\mathbb{Z}_+}2^{jqn(\beta-\frac1p)}\sum_{\alpha\in
\{0,\frac13,\frac23\}^n}\left\{\int_{\mathbb{R}^n}
\left[\sum_{Q\in\mathcal{D}^{\alpha}
_{\lambda(j),\beta+1-\frac1p}[f]}
|Q|^{q(\beta-\frac1p)}{\bf 1}_{Q}(x)\right]^{\frac pq}\upsilon(x)\,dx\right\}^{\frac qp},
\end{align*}
which, together with Lemma \ref{lem-sd} with $r: = q$,
further implies that the desired conclusion holds.
This then finishes the proof of
Proposition \ref{point2}.
\end{proof}

Based on these, we now prove Proposition \ref{upper}.

\begin{proof}[Proof of Proposition \ref{upper}]
We first show that there exists an increasing
continuous function $\psi$ on $[0,\infty)$
such that, for any $\upsilon\in A_1$
and $f\in C^\infty$ with $|\nabla f|
\in C_{\rm c}$,
\begin{align}\label{uppere2}
&\sup_{\lambda\in(0,\infty)}\lambda^p
\int_{\mathbb{R}^n}
\left[\int_{\mathbb{R}^n}{\bf 1}_{E_{\lambda,\frac{\gamma}{q}}[f]}
(x,y)|x-y|^{\gamma-n}\,dy\right]^{\frac pq}\upsilon(x)\,dx
\lesssim\psi\left([\upsilon]_{A_1}\right)
\int_{\mathbb{R}^n}\left|\nabla f(x)\right|^p
\upsilon(x)\,dx.
\end{align}
We only consider the case $q\in[p,\infty)$ because
the case $q\in(0,p)$ is similar and hence we omit the details.

Fix $\lambda\in(0,\infty)$ and
$f\in C^\infty$ with
$|\nabla f|\in C_{\rm c}$.
Applying \eqref{Elambda}, \eqref{df-e1}, and \eqref{df-e2},
we find that
\begin{align}\label{gamma<0e1}
&\int_{\mathbb{R}^n}
\left[\int_{\mathbb{R}^n}{\bf 1}_{E_{\lambda,
\frac{\gamma}{q}}[f]}(x,y)|x-y|^{\gamma-n}\,dy\right]
^{\frac pq}\upsilon(x)\,dx\notag\\
&\quad\lesssim\int_{\mathbb{R}^n}
\left[\int_{\mathbb{R}^n}{\bf 1}_{E^{(1)}_{\frac\lambda2,
\frac{\gamma}{q}}[f]}(x,y)|x-y|^{\gamma-n}\,dy\right]
^{\frac pq}\upsilon(x)\,dx\notag\\
&\qquad+\int_{\mathbb{R}^n}\left[\int_{\mathbb{R}^n}{\bf 1}_{E^{(2)}_{
\frac\lambda2,\frac{\gamma}{q}}
[f]}(x,y)|x-y|^{\gamma-n}\,dy\right]
^{\frac pq}\upsilon(x)\,dx\notag\\
&\quad=:{\rm I}_1+{\rm I}_2.
\end{align}

We first deal with ${\rm I}_1$.
To this end, choose $\varepsilon\in(0,1-n(\frac1p-\frac1q))$.
From both Proposition \ref{point}(i)
and Theorem \ref{thm-cddd}(i) with $\beta:=
\frac{1}{p}+\frac{\gamma}{qn}$,
the definition of $\lambda(j)$,
and the assumption $\varepsilon\in(0,1)$,
we deduce that
\begin{align}\label{gamma<0e4}
{\rm I_1}
&\lesssim\sum_{j\in\mathbb{Z}_+}2^{\frac{j\gamma p}{q}}
\sum_{\alpha\in\{0,\frac13,\frac23\}^n}
\sum_{Q\in\mathcal{D}^{\alpha}_{\frac{\lambda(j)}{2},
1+\frac{\gamma}{qn}}[f]}
|Q|^{\frac{\gamma p}{qn}}\upsilon(Q)\notag\\
&\lesssim\frac{[\upsilon]_{A_1}}{\lambda^p}
\sum_{j\in\mathbb{Z}_+}2^{\frac{j\gamma p}{q}}
2^{jp(\varepsilon-1-\frac{\gamma}{q})}
\int_{\mathbb{R}^n}\left|\nabla f(x)\right|^p
\upsilon(x)\,dx\notag\\
&\le\frac{[\upsilon]_{A_1}}{\lambda^p}
\sum_{j\in\mathbb{Z}_+}2^{jp(\varepsilon-1)}
\int_{\mathbb{R}^n}\left|\nabla f(x)\right|^p
\upsilon(x)\,dx
\sim\frac{[\upsilon]_{A_1}}{\lambda^p}
\int_{\mathbb{R}^n}\left|\nabla f(x)\right|^p
\upsilon(x)\,dx.
\end{align}
This then finishes the estimation of ${\rm I}_1$.

Next, we estimate ${\rm I}_2$.
By Propositions \ref{point2}(i) and
Theorem \ref{thm-cddd}(i) with $\beta:=\frac1p+\frac{\gamma}{qn}$,
the definition of $\lambda(j)$,
and the assumption $\varepsilon\in(0,
1-n(\frac1p-\frac1q))$,
we further conclude that
\begin{align*}
{\rm I}_2
&\le[\upsilon]_{A_1}
\sum_{j\in\mathbb{Z}_+}2^{jnp(
\frac1p+\frac{\gamma}{qn}-\frac{1}{q})}
\sum_{\alpha\in\{0,\frac13,\frac23\}^n}
\sum_{Q\in\mathcal{D}^{\alpha}_
{\frac{\lambda(j)}{2},1+\frac{\gamma}{qn}}[f]}
|Q|^{\frac{\gamma p}{qn}}\upsilon(Q)\notag\\
&\lesssim\frac{[\upsilon]_{A_1}^2}{\lambda^p}
\sum_{j\in\mathbb{Z}_+}2^{jnp(\frac1p+
\frac{\gamma}{qn}-\frac{1}{q})}
2^{jp(\varepsilon-1-\frac{\gamma}{q})}
\int_{\mathbb{R}^n}\left|\nabla f(x)\right|^p
\upsilon(x)\,dx\notag\\
&\le\frac{[\upsilon]_{A_1}^2}
{\lambda^p}\sum_{j\in\mathbb{Z}_+}
2^{jp[\varepsilon-1+n(\frac1p-\frac1q)]}
\int_{\mathbb{R}^n}\left|\nabla f(x)\right|^p
\upsilon(x)\,dx
\sim\frac{[\upsilon]_{A_1}^2}
{\lambda^p}\int_{\mathbb{R}^n}\left|\nabla f(x)\right|^p
\upsilon(x)\,dx,
\end{align*}
which completes the estimation of ${\rm I}_2$.
Combining this estimate, \eqref{gamma<0e1},
and \eqref{gamma<0e4},
we find that \eqref{uppere2} holds
with $\psi(t):=(1+t)t$ for any $t\in[0,\infty)$.

Now, we prove that \eqref{uppere1} holds
for any $f\in C^\infty$
with $|\nabla f|\in C_{\rm c}$.
Indeed, from the assumption that $M$ is bounded on $Y:=(X^{\frac1p})'$,
Lemma \ref{4.6},
\eqref{uppere2}, the fact that $\psi$
in \eqref{uppere2} is increasing, and Lemma \ref{lem-extrap}(ii),
it follows that, for any $\lambda\in(0,\infty)$ and
$f\in C^\infty$
with $|\nabla f|\in C_{\rm c}$,
\begin{align*}
&\lambda^p
\left\|\left[\int_{\mathbb{R}^n}
{\bf 1}_{E_{\lambda,\frac{\gamma}{q}}[f]}
(\cdot,y)|\cdot-y|^{\gamma-n}\,dy
\right]^{\frac1q}\right\|_{X}^p\\
&\quad\sim\sup_{\{g\in Y:\|g\|_{Y}\le 1\}}
\lambda^p\int_{\mathbb{R}^n}
\left[\int_{\mathbb{R}^n}
{\bf 1}_{E_{\lambda,\frac\gamma q}[f]}
(x,y)|x-y|^{\gamma-n}\right]^{\frac pq}
R_{Y}g(x)\,dx\notag\\
&\quad\lesssim
\sup_{\{g\in Y:\|g\|_{Y}\le 1\}}
\psi\left(\left[R_{Y}g\right]
_{A_1}\right)
\int_{\mathbb{R}^n}\left|\nabla f(x)\right|^p
R_{Y}g(x)\,dx\\
&\quad\le\psi\left(2\left\|M\right\|_{
Y\to Y}\right)
\sup_{\{g\in Y:\|g\|_{Y}\le 1\}}
\int_{\mathbb{R}^n}\left|\nabla f(x)\right|^p
R_{Y}g(x)\,dx\sim\psi\left(2\left\|M\right\|_{
Y\to Y}\right)
\left\|\,|\nabla f|\,\right\|_{X}^p.
\end{align*}
Thus, \eqref{uppere1} holds
for any $f\in C^\infty$
with $|\nabla f|\in C_{\rm c}$.

Moreover, if $X$ has an absolutely continuous norm,
then, repeating the standard extension argument used in the proof of
\cite[Theorem 3.29]{zyy23-1} with
Theorems 3.1, 3.6, 3.8, and 3.14 and Corollary 3.26
therein replaced by \eqref{uppere1},
we further conclude that \eqref{uppere1} also
holds for any $f\in \dot{W}^{1,X}$.
This then finishes the proof of Proposition \ref{upper}.
\end{proof}

Next, we show the lower estimate of Theorem \ref{xgamma<0} as follows.

\begin{proposition}\label{lower}
If $X$ is a ball Banach function space,
$q\in(0,\infty)$, $\gamma\in\mathbb{R}\setminus\{0\}$,
and $f\in \dot{W}^{1,X}$, then
\begin{align*}
\liminf_{\lambda\to L}
\left\|\left[\int_{\mathbb{R}^n}
{\bf 1}_{E_{\lambda,\frac\gamma q}[f]}
(\cdot,y)|\cdot-y|^{\gamma-n}\,dy\right]
^{\frac1q}\right\|_{X}
\ge\left[\frac{2\Gamma(\frac{q+1}{2})\pi^{\frac{n-1}{2}}}
{|\gamma|\Gamma(\frac{q+n}{2})}\right]^{\frac1q}
\left\|\,|\nabla f|\,\right\|_X,
\end{align*}
where $\lambda\to L$ means $\lambda\to\infty$
when $\gamma\in(0,\infty)$ and
$\lambda\to 0^+$ when $\gamma\in(-\infty,0)$.
\end{proposition}

\begin{proof}
By $f\in \dot{W}^{1,X}$
and Remark \ref{loc}, we obtain
$f\in W^{1,1}_{\loc}$.
We next show that, for almost every
$(x,h)\in\mathbb{R}^n\times\mathbb{R}^n$,
\begin{align}\label{lowere1}
\lim_{\delta\to0^+}\frac{f(x+\delta h)-f(x)}{\delta}
=\nabla f(x)\cdot h.
\end{align}
To this end, for any $k\in\mathbb{Z}^n$,
choose $\rho_k\in C^\infty_{\rm c}$
such that $\rho_k\equiv 1$ on $B(k,2)$.
Then, for any $k\in\mathbb{Z}^n$,
$f\rho_k\in \dot{W}^{1,1}$.
Combining this and \cite[Lemma 3.1]{bsvy.arxiv},
we further conclude that, for any given $k\in\mathbb{Z}^n$
and almost every
$(x,h)\in\mathbb{R}^n\times\mathbb{R}^n$,
\begin{align}\label{lowere2}
\lim_{\delta\to0^+}
\frac{(f\rho_k)(x+\delta h)-(f\rho_k)(x)}{\delta}
=\nabla\left(f\rho_k\right)(x)\cdot h.
\end{align}
For any $k\in\mathbb{Z}^n$, let
$A_k:=\{(x,h)\in\mathbb{R}^n\times\mathbb{R}^n:
\eqref{lowere2}\ \text{does not hold}\}$
and $A:=\bigcup_{k\in\mathbb{Z}^n}A_k$.
Then $|A|=0$. From the assumption that,
for any $k\in\mathbb{Z}^n$,
$\rho_k\in C^\infty_{\rm c}$
is equal to 1 in $B(k,2)$, \eqref{lowere2},
and the definition of
weak derivatives, we deduce that,
for any $k\in\mathbb{Z}^n$ and $(x,h)\in
[B(k,1)\times\mathbb{R}^n]\setminus A$,
\begin{align*}
\lim_{\delta\to0^+}\frac{f(x+\delta h)-f(x)}{\delta}
&=\lim_{\delta\to0^+}
\frac{(f\rho_k)(x+\delta h)-(f\rho_k)(x)}{\delta}\\
&=\nabla\left(f\rho_k\right)(x)\cdot h=
\left[\nabla f(x)\rho_k(x)
+f(x)\nabla\rho_k(x)\right]\cdot h=\nabla f(x)\cdot h,
\end{align*}
which completes the proof of \eqref{lowere1}.
Therefore,
repeating an argument similar to that used in
the proof of \cite[Lemma 3.2]{bsvy.arxiv}
with Lemma 3.1 and Fatou's lemma on $L^1$
replaced, respectively, by \eqref{lowere1}
and Remark \ref{rm-bqbf}(iii), we further
obtain the desired conclusion of the present proposition.
This then finishes the proof of Proposition \ref{lower}.
\end{proof}

We now prove Theorem \ref{xgamma<0}.

\begin{proof}[Proof of Theorem \ref{xgamma<0}]
We first assume (a) holds. Then,
combining Propositions \ref{upper} and
\ref{lower} and \cite[Theorem 3.6]{zyy23-1},
we conclude that \eqref{xgamma<0e1} holds.
Using this and repeating an argument
similar to that used in the proof of
\cite[Theorem 3.35]{zyy23-1},
we find that \eqref{xgamma<0e1} also holds when (b) holds.
This then finishes the proof of Theorem \ref{xgamma<0}.
\end{proof}

\section{Proof of Theorem \ref{cddd}
(New Characterizations of $A_p$-Weights)}\label{sec-6}

The main target of this section is to
prove Theorem \ref{cddd}.
To this end, let us first recall some related concepts.
Notice that the definitions of the (inverse)
Fourier transform are naturally extended
from $\mathcal{S}$ to
$\mathcal{S}'/\mathcal{P}$ (that is, the set of
all tempered distributions modulo polynomials).
Then, for any $\beta\in\mathbb{R}$ and
$f\in\mathcal{S}'/\mathcal{P}$,
the \emph{Riesz potential} $I_\beta f$ of $f$
is defined by setting
\begin{align}\label{df-ibeta}
I_{\beta} f:=\mathcal{F}^{-1}
\left[(2\pi|\cdot|)^{-\beta}\mathcal{F}f\right];
\end{align}
see \cite{ysy10} or
\cite[Chapter 1]{g2014-2} for more details.
Let $p\in [1,\infty)$ and $\upsilon\in L^1_{\loc}$.
The \emph{homogeneous weighted
Sobolev space} $\dot{H}^{1,p}_\upsilon$
is defined to be the set of all $f\in\mathcal{S}'/\mathcal{P}$ such that
$\|f\|_{\dot{H}^{1,p}_\upsilon}
:=\left\|I_{-1}f\right\|_{L^p_\upsilon}$ is finite.
Moreover, let
\begin{align}\label{df-Ynp}
\dot{Y}^{1,p}_{\upsilon}:=
\begin{cases}
\dot{W}^{1,p}_\upsilon&
\text{if}\ n=1\ \text{or}\ p=1,\\
\dot{H}^{1,p}_\upsilon&
\text{if}\ n\in\mathbb{N}\cap[2,\infty)\
\text{and}\ p\in(1,\infty).
\end{cases}
\end{align}

To show Theorem \ref{cddd}, we first
prove two technical propositions as follows.

\begin{proposition}\label{cddd-nece}
Let $p\in[1,\infty)$, $\beta\in\mathbb{R}
\setminus\{\frac1p\}$,
and $\upsilon\in L^1_{\loc}$ be nonnegative.
If there exists a positive constant
$C$ such that, for any $\alpha\in\{0,\frac13,\frac23\}^n$ and
$f\in \dot{Y}^{1,p}_{\upsilon}\cap L^1_{\loc}$,
\begin{align}\label{cddd-necee0}
\sup_{\lambda\in(0,\infty)}\lambda^p
\sum_{Q\in\mathcal{D}^\alpha_{\lambda,\beta+1-\frac1p}[f]}
|Q|^{\beta p-1}\upsilon(Q)\le
C\|f\|_{\dot{Y}^{1,p}_{\upsilon}}^p,
\end{align}
then $\upsilon\in A_p$.
\end{proposition}

\begin{proof}
Using Proposition \ref{point}(i)
with $q:=p$, we conclude that,
for any $\lambda\in(0,\infty)$, $\varepsilon\in (0,1),$ and
$f\in L^1_{\loc}$,
\begin{align*}
&\int_{\mathbb{R}^n}
\int_{\mathbb{R}^n}
{\bf 1}_{E^{(1)}_{\lambda,n(\beta-\frac1p)}[f]}
(x,y)|x-y|^{pn(\beta-\frac1p)-n}\upsilon(x)\,dx\,dy\\
&\quad\lesssim
\sum_{j\in\mathbb{Z}_+}2^{jn(\beta p-1)}\sum_{\alpha\in\{0,\frac13,
\frac23\}^n}\sum_{Q\in\mathcal{D}^{\alpha}_{\lambda(j),\beta+1-\frac1p}[f]}
|Q|^{\beta p-1}\upsilon(Q),
\end{align*}
where, for any $j\in\mathbb{Z}_+$,
$\lambda(j)\sim2^{j[1+n(\beta-\frac1p)-\varepsilon]}\lambda$.
This, combined with \eqref{cddd-necee0},
further implies that, for any $\lambda\in(0,\infty)$, $\varepsilon\in (0,1),$
and $f\in\dot{Y}^{1,p}_{\upsilon}
\cap L^1_{\loc}$,
\begin{align}\label{cddd-necee2}
&\lambda^p\int_{\mathbb{R}^n}\int_{\mathbb{R}^n}
{\bf 1}_{E_{\lambda,n(\beta-\frac1p)}^{(1)}[f]}(x,y)|x-y|^{pn(\beta-
\frac1p)-n}
\upsilon(x)\,dx\,dy\lesssim
\sum_{j\in\mathbb{Z}_+}2^{jp(\varepsilon-1)}
\left\|f\right\|_{\dot{Y}^{1,p}_{\upsilon}}^p
\sim\left\|f\right\|_{\dot{Y}^{1,p}_{\upsilon}}^p.
\end{align}

In what follows, let
$F_0:=B({\bf 0},6)\setminus B({\bf 0},3)$,
$F_1:=B({\bf 0},2)\setminus B({\bf 0},1)$,
$F_2:=B({\bf 0},18)\setminus B({\bf 0},9)$,
and $F_3:= B({\bf 0},48)\setminus B({\bf 0},24)$.
We now complete the present proof by considering
the following three cases for both $n$ and $p$.

\emph{Case 1: $n=1$ and $p\in(1,\infty)$.}
In this case, from \eqref{df-Ynp}, we infer that
$\dot{Y}^{1,p}_{\upsilon}(\mathbb{R})
=\dot{W}^{1,p}_{\upsilon}(\mathbb{R})$.
Applying this and \eqref{cddd-necee2} and
repeating an argument similar to that used in the proof of
\cite[Theorem 3.7]{dlyyz.arxiv}
with $E_{f}(\lambda,p)$ therein replaced by
$E^{(1)}_{\lambda,n(\beta-\frac1p)}[f]$, we
conclude that $\upsilon\in A_p(\mathbb{R})$ and hence
finish the proof of the present proposition
in this case.

\emph{Case 2: $n\in\mathbb{N}$ and $p=1$.} In this case,
$\dot{Y}^{1,1}_{\upsilon}
=\dot{W}^{1,1}_{\upsilon}$.
To prove $\upsilon\in A_1$,
we only need to show that,
for any nonnegative and radial $g\in L^1_{\loc}$,
\begin{align}\label{cddd-necee3}
\fint_{F_0}g(w)\,d w
\lesssim\frac{1}{\upsilon(F_0)}\int_{F_0}
g(w)\upsilon(w)\,d w.
\end{align}
Indeed, assume that \eqref{cddd-necee3} holds for the moment.
Let $a:=\mathop\mathrm{\,ess\,inf\,}_{x\in F_0}\upsilon(x)$.
We next claim that, for any $\varepsilon\in(0,\infty)$,
there exist $3\le c_\varepsilon<b_\varepsilon\le 6$
such that $\upsilon{\bf 1}_{B({\bf0},b_\varepsilon)
\setminus B({\bf0},c_\varepsilon)}<a+\varepsilon$.
Otherwise, there exists $\varepsilon_0\in(0,\infty)$
such that, for any $3\le c<b\le6$,
$\upsilon{\bf 1}_{B({\bf 0},b)\setminus B({\bf 0},c)}
\ge a+\varepsilon_0$.
This further implies that
$$\mathop\mathrm{\,ess\,inf\,}_{x\in F_0}\upsilon(x)
=\inf_{3\le c<b\le6}\mathop\mathrm{\,ess\,inf\,}_{x\in
B({\bf 0},b)\setminus B({\bf 0},c)}\upsilon(x)
\ge a+\varepsilon_0>a,$$
which contradicts the definition of $a$.
Thus, the above claim holds.
For any $\varepsilon\in(0,\infty)$,
let $g_{\varepsilon}:={\bf 1}_{B({\bf 0},b_\varepsilon)
\setminus B({\bf 0},c_\varepsilon)}$.
Then, by \eqref{cddd-necee3}, we obtain,
for any $\varepsilon\in(0,\infty)$,
\begin{align*}
\frac{|B({\bf 0},b_\varepsilon)
\setminus B({\bf 0},c_\varepsilon)|}{|F_0|}
\lesssim\frac{1}{\upsilon(F_0)}
\int_{B({\bf 0},b_\varepsilon)
\setminus B({\bf 0},c_\varepsilon)}\upsilon(z)\,dz
<\frac{a+\varepsilon}{\upsilon(F_0)}\left|B({\bf 0},b_\varepsilon)
\setminus B({\bf 0},c_\varepsilon)\right|,
\end{align*}
which further implies
$\frac{\upsilon(F_0)}{|F_0|}\lesssim a+\varepsilon$.
From this and the arbitrariness of $\varepsilon$,
it follows that
\begin{align}\label{cddd-necee10}
\frac{\upsilon(F_0)}{|F_0|}\lesssim a
=\mathop\mathrm{\,ess\,inf\,}_{x\in F_0}\upsilon(x).
\end{align}
Choose $Q_0\in\mathcal{Q}$ such that
$Q_0\subset F_0$ and $|Q_0|\sim|F_0|$.
Using \eqref{cddd-necee10}, we conclude that
\begin{align*}
\frac{\upsilon(Q_0)}{|Q_0|}
\lesssim\frac{\upsilon(F_0)}{|F_0|}
\lesssim\mathop\mathrm{\,ess\,inf\,}_{x\in F_0}\upsilon(x)
\le \mathop\mathrm{\,ess\,inf\,}_{x\in Q_0}\upsilon(x).
\end{align*}
This, together with the fact that \eqref{cddd-necee2}
has both the dilation invariance and the translation invariance
[that is, for any $\delta\in(0,\infty)$ and
$z\in\mathbb{R}^n$, both $\upsilon(\delta\cdot)$ and
$\upsilon(\cdot-z)$ satisfy \eqref{cddd-necee2} with
the same implicit positive constant], further implies that,
for any cube $Q\in\mathcal{Q}$,
$\frac{\upsilon(Q)}{|Q|}\lesssim
\mathop\mathrm{\,ess\,inf\,}_{x\in Q}\upsilon(x)$
and hence $\upsilon\in A_1$.

Now, we prove \eqref{cddd-necee3}.
To this end, we first assume that $g\in C^\infty$
is nonnegative and radial
and choose $\eta\in C^\infty$
such that $\eta$ is radial and
\begin{align}\label{cddd-necee17}
{\bf 1}_{F_0}\le \eta\le {\bf 1}_{B({\bf 0},8)\setminus
B({\bf 0},2)}.
\end{align}
Then there exist two infinitely differentiable
functions
$g_0,\eta_0$ on $[0,\infty)$
such that, for any $x\in\mathbb{R}^n$,
$g(x)=g_0(|x|)$ and $\eta(x)=\eta_0(|x|)$.
For any $x\in\mathbb{R}^n$, let
\begin{align}\label{cddd-necee12}
f(x):=\int_{B({\bf 0},|x|)}g(w)\eta(w)\,dw.
\end{align}
Applying the polar coordinate, we obtain,
for any $x\in\mathbb{R}^n$,
\begin{align*}
f(x)=\int_{0}^{|x|}\int_{\mathbb{S}^{n-1}}g_0(r)
\eta_0(r)r^{n-1}\,d\sigma(\xi)\,dr
=\omega_{n-1}\int_{0}^{|x|}g_0(r)\eta_0(r)r^{n-1}\,dr,
\end{align*}
which further implies that
\begin{align}\label{cddd-necee6}
\nabla f(x)=\omega_{n-1}g_0(|x|)\eta_0(|x|)|x|^{n-1}
\frac{x}{|x|}=\omega_{n-1}g(x)\eta(x)|x|^{n-2}x,
\end{align}
where $\omega_{n-1}: = 2\pi^{\frac{n}{2}}/\Gamma(\frac{n}{2})$
is the surface area of $\mathbb{S}^{n-1}$
and $\Gamma(\cdot)$ denotes the Gamma function.
Thus, $|\nabla f|\in C_{\rm c}$ and hence
$f\in \dot{W}^{1,1}_{\upsilon}
=\dot{Y}^{1,1}_{\upsilon}\cap L^1_{\loc}$.
For any $x\in A_1$, $y\in A_2$, and
$z\in B_{x,y}$,
\begin{align*}
|z|\ge|y|-|y-z|\ge|y|-\frac{1}{20}|x-y|>8,
\end{align*}
which, combined with \eqref{cddd-necee12} and
\eqref{cddd-necee17},
further implies that
\begin{align}\label{cddd-necee4}
f(z)-f(x)&>\int_{B({\bf 0},8)\setminus
B({\bf 0},2)}g(w)\eta(w)\,d w\ge\int_{F_0}g(w)\,d w
\sim\int_{F_0}g(w)\,d w |x-y|^{1+n(\beta-1)}
\end{align}
with the positive equivalence constants depending
only on $n$ and $\beta$.
On the other hand, notice that,
for any $x\in F_2$, $y\in F_1$, and $z\in B_{x,y}$,
$|z|\le|z-y|+|y|<
\frac{1}{20}|x-y|+|y|< 3,$
which, together with \eqref{cddd-necee12} and
\eqref{cddd-necee17} again, further implies that
\begin{align*}
f(x)-f(z)&>\int_{B({\bf 0},9)
\setminus B({\bf 0},3)}g(w)\eta(w)\,d w\ge\int_{F_0}g(w)\,d w
\sim\int_{F_0}g(w)\,d w |x-y|^{1+n(\beta-1)}
\end{align*}
with the positive equivalence constants depending only
on $n$ and $\beta$.
Using this and \eqref{cddd-necee4},
we find that there exists a positive
constant $C_{(n,\beta)}$, depending only on
$n$ and $\beta$, such that,
for any $(x,y)\in (F_1\times F_2)
\cup(F_2\times F_1)$,
\begin{align*}
\left|f(x)-f_{B_{x,y}}\right|
=\left|\fint_{B_{x,y}}[f(x)-f(w)]\,d w\right|
\ge C_{(n,\beta)}
\int_{F_0}g(w)\,d w|x-y|^{1+n(\beta-1)}.
\end{align*}
This further implies
$(F_1\times F_2)\cup
(F_2\times F_1)\subset
E^{(1)}_{\lambda_{(n,\beta)},n(\beta-1)}[f],$
where $\lambda_{(n,\beta)}:=C_{(n,\beta)}
\int_{F_0}g(w)\,d w$.
From this, \eqref{cddd-necee2}, the fact that
$\dot{Y}^{1,1}_{\upsilon}=\dot{W}^{1,1}_\upsilon$
in this case,
\eqref{cddd-necee6},
and \eqref{cddd-necee17}, it follows that
\begin{align}\label{cddd-necee5}
&\upsilon\left(F_1\cup F_2\right)\int_{F_0}g(w)\,d w\notag\\
&\quad\sim\lambda_{(n,\beta)}\int_{F_1\cup F_2}
\upsilon(w)\,d w\lesssim\lambda_{(n,\beta)}
\int_{\mathbb{R}^n}\int_{\mathbb{R}^n}
{\bf 1}_{E^{(1)}_{\lambda_\beta,n(\beta-1)}[f]}
(x,y)|x-y|^{(\beta-2)n}\upsilon(x)\,dx\,dy\notag\\
&\quad\lesssim\int_{\mathbb{R}^n}\left|\nabla f(w)\right|
\upsilon(w)\,dw
\lesssim\int_{B({\bf 0},8)\setminus B({\bf 0},2)}
g(w)\upsilon(w)\,dw.
\end{align}
Next, fix a nonnegative and radial
$g\in L^1_{\loc}$ and
choose a nonnegative and radial
$\varphi\in C^\infty_{\rm c}$
such that $\int_{\mathbb{R}^n}\varphi(x)\,dx=1$.
For any $\varepsilon\in(0,\infty)$ and $x\in \mathbb{R}^n$,
let $\varphi_\varepsilon(x):=\frac{1}{\varepsilon^{n}}
\varphi(\frac{x}{\varepsilon})$ and
$g_\varepsilon(x):=(g{\bf 1}_{F_0})\ast
\varphi_{\varepsilon}(x).$
Therefore, $\{g_\varepsilon\}_{\varepsilon\in(0,\infty)}$
in $C^\infty$
are all nonnegative and radial and hence \eqref{cddd-necee5}
holds with $g:=g_{\varepsilon}$ for any $\varepsilon\in(0,\infty)$.
By \cite[Corollary 2.9]{D2000}, we conclude that
$\lim_{\varepsilon\to0^+}g_{\varepsilon}=g$
almost everywhere in $F_0$. Moreover, applying
Young's inequality, we find that, for any $\varepsilon\in(0,\infty)$,
\begin{align*}
\left\|g_{\varepsilon}\right\|_{L^\infty}
\le\left\|\varphi_{\varepsilon}\right\|_{L^1}
\left\|g{\bf 1}_{F_0}\right\|_{L^\infty}
=\left\|g{\bf 1}_{F_0}\right\|_{L^\infty}.
\end{align*}
From this, the
Lebesgue dominated convergence theorem, and
\eqref{cddd-necee5} with $g:=g_\varepsilon$ for any $\varepsilon\in(0,\infty)$,
we deduce that
\begin{align}\label{cddd-necee8}
\upsilon(F_1\cup F_2)\fint_{F_0}g(w)\,dw
&\sim\upsilon(F_1\cup F_2)\int_{F_0}g(w)\,dw\lesssim
\int_{F_0}g(w)\upsilon(w)\,dw.
\end{align}
Applying this and the fact that \eqref{cddd-necee2}
has the dilation invariance again,
we conclude that, for any nonnegative and radial
$g\in L^1_{\loc}$,
$$\upsilon(F_0\cup[B({\bf 0},54)
\setminus B({\bf 0},27)])\fint_{F_2}g(w)\,d w
\lesssim\int_{F_2}g(w)\upsilon(w)\,dw.$$
Letting $g:={\bf 1}_{F_2}$, we obtain
$$\upsilon(F_0)
\le\upsilon(F_0\cup [B({\bf 0},54)
\setminus B({\bf 0},27)])
\lesssim\upsilon(F_2)\le\upsilon(F_1\cup F_2).$$
From this and \eqref{cddd-necee8},
we infer \eqref{cddd-necee3}.
This then finishes the proof of the present proposition
in this case.

\emph{Case 3: $n\in \mathbb{N}\cap[2,\infty)$ and $p\in(1,\infty)$.} In this case,
$\dot{Y}^{1,p}_{\upsilon}=\dot{H}^{1,p}_{\upsilon}$.
Let $g\in C^\infty$
be nonnegative and let
$\eta\in C^\infty$
satisfy \eqref{cddd-necee17}.
Let $f:=I_{1}(g\eta)$, where
$I_1$ is the same as in \eqref{df-ibeta} with
$\beta$ replaced by $1$.
Then, by \cite[pp.\,10--11]{g2014-2},
we find that $f\in L^1_{\loc}$
and, for any $x\in\mathbb{R}^n$,
\begin{align}\label{cddd-necee18}
f(x)=\frac{\Gamma(\frac{n-1}{2})}{2\pi^{\frac n2}
\Gamma(\frac12)}
\int_{\mathbb{R}^n}\frac{g(w)\eta(w)}{|x-w|^{n-1}}
\,d w,
\end{align}
where $\Gamma(\cdot)$ denotes the Gamma function.
In addition, applying the semigroup
property of Riesz potentials (see, for instance,
\cite[p.\,9]{g2014-2}) and \eqref{cddd-necee17}, we conclude that
\begin{align}\label{cddd-necee19}
\left\|f\right\|_{\dot{Y}^{1,p}_{\upsilon}}^p
&=\left\|I_{-1}I_1(g\eta)\right\|_{L^p_\upsilon}
=\int_{\mathbb{R}^n}\left|g(w)\eta(w)
\right|^p\upsilon(w)\,dw\lesssim\int_{B({\bf 0},8)\setminus
B({\bf 0},2)}|g(w)|^p\upsilon(w)\,d w.
\end{align}

Observe that, for any $x\in F_1$, $y\in F_3$, $z\in B_{x,y}$,
and $w\in B({\bf 0},8)\setminus
B({\bf 0},2)$,
$|x-w|\le|x|+|w|<10$
and $|z-w|\ge|y-w|-|y-z|\ge
|y|-|w|-\frac{1}{20}|x-y|
>13.$
Using these, \eqref{cddd-necee18}, and
\eqref{cddd-necee17},
we find that, for any
$x\in F_1$, $y\in F_3$, and $z\in B_{x,y}$,
\begin{align}\label{cddd-necee20}
f(x)-f(z)
&\sim\int_{B({\bf 0},8)\setminus B({\bf 0},2)}
g(w)\eta(w)\left(
\frac{1}{|x-w|^{n-1}}-\frac{1}{|z-w|^{n-1}}\right)
\,d w\notag\\
&\gtrsim\int_{B({\bf 0},8)\setminus B({\bf 0},2)}
g(w)\eta(w)\,d w
\ge \int_{F_0}g(w)\,d w
\sim\int_{F_0}g(w)\,d w |x-y|^{1+n(\beta-\frac1p)},
\end{align}
where the implicit positive constants
depend only on $n$, $p$, and $\beta$.
In addition, for any $x\in F_3$, $y\in F_1$, $z\in B_{x,y}$,
and $w\in B({\bf 0},8)\setminus
B({\bf 0},2)$, we have $|x-w|\ge|x|-|w|>16$
and
\begin{align*}
|z-w|\le |z-y|+|y-w|<\frac{1}{20}
|x-y|+|y|+|w|<13.
\end{align*}
Therefore, from \eqref{cddd-necee18} and
\eqref{cddd-necee17} again,
it follows that,
for any $x\in F_3$, $y\in F_1$, and $z\in B_{x,y}$,
\begin{align*}
f(z)-f(x)
&\sim\int_{B({\bf 0},8)\setminus B({\bf 0},2)}
g(w)\eta(w)\left(
\frac{1}{|z-w|^{n-1}}-\frac{1}{|x-w|^{n-1}}\right)
\,d w\notag\\
&\gtrsim\int_{B({\bf 0},8)\setminus B({\bf 0},2)}
g(w)\eta(w)\,d w
\ge \int_{F_0}g(w)\,d w
\sim\int_{F_0}g(w)\,d w |x-y|^{1+n(\beta-\frac1p)}
\end{align*}
with the implicit positive constants depending only
on $n$, $p$, and $\beta$.
Combining this and \eqref{cddd-necee20}, we further
conclude that there exists a positive constant
$C_{(n,p,\beta)}$, depending only on $n$, $p$,
and $\beta$, such that, for any
$(x,y)\in(F_1\times F_3)\cup(F_3\times F_1)$,
\begin{align*}
\left|f(x)-f_{B_{x,y}}\right|
=\left|\fint_{B_{x,y}}[f(x)-f(w)]\,d w\right|
\ge C_{(n,p,\beta)}\int_{F_0}g(w)\,d w|x-y|^{1+n
(\beta-\frac1p)}.
\end{align*}
This then implies that
$(F_1\times F_3)\cup(F_3\times F_1)\subset
E^{(1)}_{\lambda_{(n,p,\beta)},n(\beta-\frac1p)}[f]$
with $\lambda_{(n,p,\beta)}:=
C_{(n,p,\beta)}\int_{F_0}g(w)\,d w$.
Applying this and \eqref{cddd-necee19}, similar
to \eqref{cddd-necee5}, we obtain
\begin{align*}
\upsilon(F_1\cup F_3)
\left[\fint_{F_0}g(w)\,d w\right]^p
\lesssim\int_{B({\bf 0},8)\setminus
B({\bf 0},2)}\left[g(w)\right]^p
\upsilon(w)\,d w.
\end{align*}
By this and an argument similar to that
used in the estimation of \eqref{cddd-necee8},
we further find that, for any nonnegative
$g\in L^1_{\loc}$,
\begin{align}\label{cddd-necee21}
\upsilon(F_1\cup F_3)
\left[\fint_{F_0}g(w)\,d w\right]^p
\lesssim\int_{F_0}\left[g(w)\right]^p
\upsilon(w)\,d w.
\end{align}
Letting $g:={\bf 1}_{F_0}$,
we obtain $\upsilon(F_1\cup F_3)
\lesssim\upsilon(F_0)$.
This further implies that
$\upsilon(B({\bf 0},2)\setminus
B({\bf 0},1))
\lesssim \upsilon(B({\bf 0},6)\setminus
B({\bf 0},3))$
and
$\upsilon(B({\bf 0},48)\setminus
B({\bf 0},24))
\lesssim \upsilon(B({\bf 0},6)\setminus
B({\bf 0},3)).$
From these and the fact that \eqref{cddd-necee2}
has the dilation invariance, we deduce that
\begin{align*}
\upsilon\left(B({\bf 0},16)\setminus
B({\bf 0},8)\right)
\lesssim\upsilon\left(B({\bf 0},2)\setminus
B({\bf 0},1)\right)
\lesssim \upsilon\left(B({\bf 0},6)\setminus
B({\bf 0},3)\right)
\end{align*}
and hence
\begin{align}\label{cddd-necee22}
\upsilon\left(B\left({\bf 0},\frac{16}{3}\right)\setminus
B\left({\bf 0},\frac{8}{3}\right)\right)
\lesssim\upsilon\left(B({\bf 0},2)\setminus
B({\bf 0},1)\right).
\end{align}
Choose a cube $Q_0\subset B({\bf 0},5)\setminus
B({\bf 0},4)$ such that
$|Q_0|\sim 1$. Then, applying
the fact that $$B({\bf 0},5)\setminus
B({\bf 0},4)\subset B\left({\bf 0},\frac{16}{3}\right)\setminus
B\left({\bf 0},\frac{8}{3}\right)$$ and \eqref{cddd-necee22},
we further obtain
\begin{align*}
\upsilon(Q_0)\le
\upsilon\left(B\left({\bf 0},\frac{16}{3}\right)\setminus
B\left({\bf 0},\frac{8}{3}\right)\right)\lesssim
\upsilon\left(B({\bf 0},2)\setminus
B({\bf 0},1)\right)\le
\upsilon(A_1\cup A_3).
\end{align*}
Using this and replacing
$g$ by $g{\bf 1}_{Q_0}$ in \eqref{cddd-necee21},
we conclude that, for any nonnegative
$g\in L^1_{\loc}$,
\begin{align*}
\left[\fint_{Q_0}
g(w)\, d w \right]^p
\lesssim\frac{1}{\upsilon(Q_0)}
\int_{Q_0}\left[g(w)\right]^p
\upsilon(w)\,d w.
\end{align*}
This, combined with the fact that \eqref{cddd-necee2}
has both the dilation invariance and the translation invariance,
further implies that, for any cube
$Q\in \mathcal{Q}$ and any nonnegative
$g\in L^1_{\loc}$,
\begin{align*}
\left[\fint_{Q}
g(w)\, d w \right]^p
\lesssim\frac{1}{\upsilon(Q)}
\int_{Q}\left[g(w)\right]^p
\upsilon(w)\,d w.
\end{align*}
From this and Lemma \ref{ApProperty}(iii),
we infer that $\upsilon\in A_p$
and hence the present proposition in this case.
This then finishes the proof of
Proposition \ref{cddd-nece}.
\end{proof}

\begin{proposition}\label{bsvy-nece}
Let $p\in[1,\infty)$, $q\in(0,\infty)$,
$\gamma\in\mathbb{R}$, and $\upsilon$
be a nonnegative locally integrable function
on $\mathbb{R}^n$. If there exists a
positive constant $C$ such that, for any
$f\in \dot{Y}^{1,p}_{\upsilon}
\cap L^1_{\loc}$,
\begin{align*}
&\sup_{\lambda\in(0,\infty)}
\lambda^p\int_{\mathbb{R}^n}
\left[\int_{\mathbb{R}^n}{\bf 1}_{E_{\lambda,
\frac{\gamma}{q}}[f]}(x,y)|x-y|^{\gamma-n}\,dy
\right]^{\frac pq}\upsilon(x)\,dx
\le C\|f\|_{\dot{Y}^{1,p}_{\upsilon}}^p,
\end{align*}
then $\upsilon\in A_p$.
\end{proposition}

\begin{proof}
We prove the present proposition
by considering the following two cases for $n$ and $p$.

\emph{Case 1: $n=1$ and $p\in(1,\infty)$.}
In this case,
repeating an argument similar to that used in the proof of
\cite[Theorem 3.7]{dlyyz.arxiv}
with $E_{f}(\lambda,p)$ therein replaced by
$E_{\lambda,\frac{\gamma}{q}}[f]$, we
find that $\upsilon\in A_p(\mathbb{R})$ and hence
finish the proof of the present proposition
in this case.

\emph{Case 2: $n\in\mathbb{N}$ and $p=1$
or $n,p\in(1,\infty)$.} In this case,
repeating an argument similar to that used
in Cases 2 and 3 of the proof of
Proposition \ref{cddd-nece} with $f_{B_{x,y}}$
replaced by $f(y)$, we obtain $\upsilon\in A_p$.
Therefore, the present proposition holds also in this case.
This then finishes the proof of Proposition \ref{bsvy-nece}.
\end{proof}

To obtain Theorem \ref{cddd},
we also need the following conclusion, which
easily follows from \cite[Theorem 2.8 and Remark 2.9]{b82};
we omit the details.

\begin{lemma}\label{w and h}
If $p\in(1,\infty)$ and
$\upsilon\in A_p$,
then
$\dot{H}^{1,p}_\upsilon=\dot{W}^{1,p}_{\upsilon}$
with equivalent seminorms.
\end{lemma}

Via the above preparations,
we now show Theorem \ref{cddd}.

\begin{proof}[Proof of Theorem \ref{cddd}]
By Theorem \ref{thm-cddd},
Proposition \ref{cddd-nece}, and Lemma \ref{w and h},
we find that (i) and (ii) are equivalent.
In addition, from Proposition \ref{bsvy-nece},
we deduce that, if (iii) holds, then (i) holds.

Therefore, it suffices to prove that (i) implies (iii).
Assume $\upsilon\in A_p$, $q\in(0,\infty)$ satisfies
$n(1-\frac1q)<1$, and
$\gamma\in\Gamma_{p,q}$.
If $p=1$, then, using Lemmas \ref{ApProperty}
and \ref{Lpv}, we conclude that
the weighted Lebesgue space $L^1_\upsilon$
under consideration satisfies
Proposition \ref{upper} with
$p:=1$ and hence, for any
$f\in \dot{W}^{1,1}_\upsilon$,
\eqref{cddd-necee11} holds.
On the other hand, if $p\in(1,\infty)$,
then, from (iv) and (vi) of Lemma \ref{ApProperty} and
the assumption $n(1-\frac1q)<1$, we infer that
there exists $r\in(1,\infty)$ such that
$\upsilon\in A_{\frac pr}$ and
$n(\frac1r-\frac1q)<1$. Applying this,
(iv) and (v) of Lemma \ref{ApProperty}, and
Lemma \ref{Lpv}, we find that
the space $L^p_\upsilon$
under consideration satisfies
Proposition \ref{upper} with $p$ therein replaced
by $r$ and hence \eqref{cddd-necee11} holds
for any $f\in \dot{W}^{1,p}_\upsilon$.
Thus, from Lemma \ref{w and h},
it follows that
$\dot{Y}^{1,p}_{\upsilon}
=\dot{W}^{1,p}_\upsilon$
and hence (i) holds.
By this, we conclude that (i) implies (iii),
which then completes the proof of Theorem \ref{cddd}.
\end{proof}

\section{Gagliardo--Nirenberg Type Inequalities}\label{sec-gn}

In this section, we aim to apply Theorems \ref{t953}
and \ref{xgamma<0} to establish three Gagliardo--Nirenberg type inequalities
in the framework of ball Banach function spaces,
which are sharp or improve the corresponding known results.

Assume that $0\le s_1< s< s_2\le 1$, $1\le p_1,p,p_2\le \infty$, and $\theta\in(0,1)$
satisfy
$s=\theta s_1+(1-\theta)s_2$ and $\frac1p=\frac{\theta}{p_1}+\frac{1-\theta}{p_2}.$
Gagliardo--Nirenberg inequalities focus on
the validity of the following inequalities:
for any $f\in W^{s_1,p_1}\cap W^{s_2,p_2}$,
\begin{align}\label{gn}
\|f\|_{W^{s,p}}\lesssim \|f\|_{W^{s_1,p_1}}^\theta
\|f\|_{W^{s_2,p_2}}^{1-\theta},
\end{align}
where $W^{s,p}$, $W^{s_1,p}$, and $W^{s_2,p}$ denote
the (fractional) Sobolev spaces
(see, for example, \cite{bm18}).
In what follows,
we only consider the critical case $s_2=p_2=1$.
In this case, Cohen et al.\ \cite[Theorem 1.5]{cddd03}
obtained \eqref{gn} when $s_1\in(0,\frac{1}{p_1})$ by using the
almost wavelet characterization of $W^{1,1}$.
Moreover, Brezis and Mironescu \cite{bm18} showed \eqref{gn}
fails if $s_1\ge \frac{1}{p_1}$, that is, the assumption
on the smoothness exponent $s_1$
considered by Cohen et al. \cite{cddd03} is necessary.
Notably, a suitable substitute for \eqref{gn}
in the range $s_1\ge \frac{1}{p_1}$ has recently been investigated by
Brezis et al. \cite{bvy2021} based on BSVY formulae.
In addition, when $s_1=0$, it is known that \eqref{gn} holds if and only if
$p_1\in [1,\infty)$. Also, Brezis et al. \cite{bvy2021}
established a good substitute for \eqref{gn} when $p_1=\infty$
via BSVY formulae.

In this section, applying both the sharp real interpolation
between critical weighted Sobolev and weighted
Besov spaces (Theorem \ref{t953}) and
the sharp BSVY formula in
ball Banach function spaces (Theorem \ref{xgamma<0}),
we establish several variants of
the aforementioned critical Gagliardo--Nirenberg type inequalities
in the framework of ball Banach function spaces.

We first recall the concept of
Triebel--Lizorkin spaces based on general ball Banach function spaces.
Let $\{\theta_j\}_{j\in\mathbb{Z}_+}$ be a
standard inhomogeneous smooth dyadic
decomposition of unity (see, for instance, \cite[(3.3)]{syy24}),
$X$ a ball Banach function space defined in Definition \ref{1659},
$s\in\mathbb{R}$, and $q\in(0,\infty)$.
The \emph{Triebel--Lizorkin space} $F^s_{X,q}$,
based on $X$, is defined to be the set of all
$f\in\mathcal{S}'$ such that
$\|f\|_{F^s_{X,q}}:=\|(\sum_{j\in\mathbb{Z}_+}
2^{jsq}|\theta_j\ast f|^q)^{\frac1q}\|_{X}<\infty;$
we refer to \cite[Section 3]{fsyy25} for its real-variable
characterizations.
Now, we establish the
following inequality, which extends Cohen et al.
\cite[Theorem 1.5]{cddd03} to various Sobolev type spaces.

\begin{theorem}\label{c1016}
Let $s,s_1\in\mathbb{R}$, $p,p_1\in(1,\infty)$, and
$\theta\in(0,1)$ satisfy
\begin{align}\label{idd}
s=1-\theta+\theta s_1\ \text{and}\ \frac{1}{p}=1-\theta+\frac{\theta}{p_1}.
\end{align}
Assume that
$X$ is a ball Banach function space such that the
Hardy--Littlewood maximal operator
$M$ is bounded or endpoint bounded on $X'$.
If $s_1\in(-\infty,\frac{1}{p_1})\cup(1,\infty)$,
then there exists a positive constant $C$
such that, for any $f\in W^{1,X}\cap F^{s_1}_{X^{p_1},p_1}$,
\begin{align}\label{e853}
\|f\|_{F^{s}_{X^p,p}}
\le C\|f\|_{W^{1,X}}^{1-\theta}\|f\|_{F^{s_1}_{X^{p_1},p_1}}^\theta.
\end{align}
\end{theorem}
\begin{proof}
Fix $f\in W^{1,X}\cap F^{s_1}_{X^{p_1},p_1}$. We first assume that
$M$ is bounded on $X'$. Then, by Lemma \ref{lem-extrap}(ii), we conclude that, for any $g\in X'$
with $\|g\|_{X'}\leq1$, $\upsilon_g:=R_{X'}g\in A_1$.
From this, Lemma \ref{4.6}, Theorem \ref{t953}, \eqref{e1126},
and Lemma \ref{lem-extrap} again, we deduce that
\begin{align}\label{e2139}
\|f\|_{F^{s}_{X^p,p}}
&\le\sup_{\{g\in X':\|g\|_{X'}\le 1\}}\|f\|_{B^{s,\upsilon_g}_{p,p}}
\lesssim \sup_{\{g\in X':\|g\|_{X'}\le 1\}}
\|f\|_{W^{1,1}_{\upsilon_g}}^{1-\theta}\|f\|_{B^{s_1,\upsilon_g}_{p_1,p_1}}^{\theta}\notag\\
&\le \left(\sup_{\{g\in X':\|g\|_{X'}\le 1\}}
\|f\|_{W^{1,1}_{\upsilon_g}}\right)^{1-\theta}
\left(\sup_{\{g\in X':\|g\|_{X'}\le 1\}}\|f\|_{B^{s_1,\upsilon_g}_{p_1,p_1}}\right)^{\theta}
\lesssim\|f\|_{W^{1,X}}^{1-\theta}\|f\|_{F^{s_1}_{X^{p_1},p_1}}^\theta.
\end{align}
Next, we assume that $M$ is endpoint bounded on $X'$ and
let $\{\theta_m\}_{m\in\mathbb{N}}$ be as in
Definition \ref{def-epbd}. Applying \eqref{e2139}
with $X$ replaced by $X^{\frac{1}{\theta_m}}$ for any $m\in\mathbb{N}$
and an argument similar to that used in the proof in \cite[p.\,56]{dlyyz.arxiv},
we find that \eqref{e2139} still holds in this case, which then
completes the proof of Theorem \ref{c1016}.
\end{proof}

\begin{remark}\label{rem52}
\begin{enumerate}
\item[(i)] In Theorem \ref{c1016},
the assumption $(-\infty,\frac{1}{p_1})\cup(1,\infty)$
on the smoothness exponent $s_1$ is sharp.
Indeed, let $X:=L^1$.
Then, in this case, by \cite[Theorem 1]{bm18}, we find that
\eqref{e853} fails if $s_1\in[\frac{1}{p_1},1)$.

\item[(ii)]Theorem \ref{c1016} when $X:=L^1$ reduces to \cite[Theorem 1.5]{cddd03}
and, to our best knowledge, the other cases of Theorem \ref{c1016} are new.

\item[(iii)] It is still unclear whether \eqref{e853} holds when $s_1=1$.
\end{enumerate}
\end{remark}

Furthermore,
repeating the proof of \cite[Theorem 4.4]{zyy23-1}
with Theorem 3.35 therein replaced by
Theorem \ref{xgamma<0} here, we obtain the following inequality,
which removes the restriction $s_1<\frac{1}{p_1}$
of Theorem \ref{c1016} by involving the weak norms.

\begin{theorem}\label{2256}
Let $\gamma\in \Gamma_{1,1}$, $s,s_1\in[0,1)$,
$p,p_1\in(1,\infty)$, and $\theta\in(0,1)$ satisfy \eqref{idd}. Assume that
$X$ is a ball Banach function space satisfying all the assumptions of
Theorem \ref{xgamma<0} with $p:=q:=1$ therein and the above $\gamma$.
Then there exists a positive constant $C$ such that,
for any $f\in \dot{W}^{1,X}$,
\begin{align}\label{939}
&\sup_{\lambda\in(0,\infty)}\lambda
\left\|\int_{\mathbb{R}^n}
\mathbf{1}_{E_{\lambda,\frac{\gamma}{p}+s-1}[f]}(\cdot,y)
\left|\cdot-y\right|^{\gamma-n}\,dy\right\|_{X}^{\frac{1}{p}}\notag\\
&\quad\leq C
\sup_{\lambda\in(0,\infty)}\lambda
\left\|\int_{\mathbb{R}^n}
\mathbf{1}_{E_{\lambda,\frac{\gamma}{p_1}+s_1-1}[f]}(\cdot,y)
\left|\cdot-y\right|^{\gamma-n}\,dy\right\|_{X}^{\frac{\theta}{p_1}}
\left\|\,\left|\nabla f\right|\,\right\|_X^{1-\theta}.
\end{align}
\end{theorem}

On the other hand,
repeating the proof of \cite[Theorem 4.1]{zyy23-1}
with Theorem 3.35 therein replaced by
Theorem \ref{xgamma<0} here, we also obtain
the following Gagliardo--Nirenberg type inequality
with $s_1=0$.

\begin{theorem}\label{2255}
Let $\gamma\in \Gamma_{1,1}$, $s,\theta\in(0,1)$, and $p,p_1\in[1,\infty]$
satisfy \eqref{idd}.
Assume that $X$ is a ball Banach function space
satisfying all the assumptions of
Theorem \ref{xgamma<0} with $p:=q:=1$ therein and the above $\gamma$.
Then the following statements hold.
\begin{enumerate}
\item[\textup{(i)}]
If $p_1\in[1,\infty)$, then there exists a positive constant $C$ such that,
for any $f\in\dot{W}^{1,X}$,
\begin{align}\label{2053}
\sup_{\lambda\in(0,\infty)}\lambda
\left\|\int_{\mathbb{R}^n}
\mathbf{1}_{E_{\lambda,\frac{\gamma}{p}+s-1}[f]}(\cdot,y)
\left|\cdot-y\right|^{\gamma-n}\,dy\right\|_{X}^\frac{1}{p}
\leq C\|f\|_{X^{p_1}}^{\theta}\left\|\,\left|\nabla f\right|\,\right\|_X^{1-\theta}.
\end{align}
\item[\textup{(ii)}]
If $p_1=\infty$, then there exists a
positive constant $C$ such that,
for any $f\in\dot{W}^{1,X}$,
\begin{align}\label{2054}
\sup_{\lambda\in(0,\infty)}\lambda
\left\|\int_{\mathbb{R}^n}
\mathbf{1}_{E_{\lambda,\frac{\gamma}{p}+s-1}[f]}(\cdot,y)
\left|\cdot-y\right|^{\gamma-n}\,dy\right\|_{X}^\frac{1}{p}
\leq C\|f\|_{L^\infty}^{\theta}\left\|\,\left|\nabla f\right|\,\right\|_X^{1-\theta}.
\end{align}
\end{enumerate}
\end{theorem}

\begin{remark}\label{rm48}
\begin{enumerate}
\item[(i)] In Theorem \ref{2255}(i), it is not necessary
to assume that $X^{p_1}$ is a ball Banach function space
as in \cite[Theorem 4.1(i)]{zyy23-1} because it must be true if
$X$ is a ball Banach function space and $p_1\in[1,\infty)$ (see,
for instance, \cite[Section 1.d]{lt79}).

\item[(ii)] Theorems \ref{2256} and \ref{2255}
when $\gamma\in (-\infty,-1)$ and
$n\in\mathbb{N}\cap[2,\infty)$ are new and,
in the other cases, reduce to
\cite[Theorems 4.4 and 4.1]{zyy23-1},
respectively.
\end{enumerate}
\end{remark}

\section{Applications to Specific Function Spaces}
\label{S5}

The main target of this section is to apply
Theorems \ref{xgamma<0}, \ref{c1016},
\ref{2256}, and \ref{2255}
to obtain various Gagliardo--Nirenberg type inequalities and
BSVY formulae in specific function spaces,
including weighted Lebesgue spaces
(see Subsection \ref{5.4}),
(Bourgain--)Morrey type spaces (see Subsection \ref{5.1}),
variable Lebesgue spaces (see Subsection \ref{5.3}),
Orlicz and Orlicz-slice
spaces (see Subsection \ref{5.5}), and
local and global generalized Herz spaces
(see Subsection \ref{5.2}). All the results
are new or essentially improve those corresponding ones obtained
in \cite{zyy23-1} by removing the restriction $n=1$
and making the range of the exponent $q$ sharp.

\subsection{Weighted Lebesgue Spaces}\label{5.4}

Let $r\in(0,\infty)$ and $\upsilon$
be a nonnegative locally integrable function on
$\mathbb{R}^n$.
As pointed out in \cite[p.\,86]{shyy2017},
the weighted Lebesgue space $L^r_\upsilon$
is a ball quasi-Banach function space,
but it may not be a Banach function space.
If $r\in[1,\infty)$ and $\upsilon\in A_r$,
then define
$r_\upsilon:=
\inf\{r\in[1,\infty):
\upsilon\in A_r\}.$

Now, we obtain the BSVY formula and
Gagliardo--Nirenberg type inequalities in weighted Lebesgue spaces
as follows.
\begin{theorem}\label{1657}
Let $r\in[1,\infty)$ and $\upsilon\in A_r$.
\begin{enumerate}
\item[{\rm (i)}] If $q\in(0,\infty)$ satisfies $n(\frac{r_\upsilon}{r}-
\frac1q)<1$ and $\gamma\in \Gamma_{r,q}$,
then, for any $f\in \dot{W}^{1,r}_\upsilon$,
\eqref{xgamma<0e1} holds with $X:=L^r_\upsilon$.
\item[{\rm (ii)}] Let all the symbols be the same as in Theorem \ref{c1016}.
If $s_1\in(-\infty,\frac{1}{p_1})\cup(1,\infty)$,
then, for any $f\in F^{s_1}_{L^{rp_1}_\upsilon,p_1}\cap W^{1,r}_\upsilon$,
\eqref{e853} holds with $X:=L^r_\upsilon$.
\item[{\rm (iii)}]Let all the symbols be the same as in Theorem \ref{2256}.
Then, for any $f\in\dot{W}^{1,r}_\upsilon$, \eqref{939} holds with
$X:=L^r_\upsilon$.
\item[{\rm (iv)}]Let all the symbols be the same as in Theorem \ref{2255} and
$f\in\dot{W}^{1,r}_\upsilon$.
If $p_1\in[1,\infty)$, then \eqref{2053} holds with $X:=L^r_\upsilon$ and,
if $p_1=\infty$, then \eqref{2054} holds with $X:=L^r_\upsilon$.
\end{enumerate}
\end{theorem}

\begin{proof}
We first show (i). Repeating an argument similar to that used in
the proof that (i) implies (iii) of
Theorem \ref{cddd} and replacing $p$ therein by
$r$ and replacing $r$ therein by
$\frac{r}{r_\upsilon+\varepsilon}$ with
$\varepsilon\in(0,r-r_\upsilon)$, we conclude that
the weighted Lebesgue space $L^r_\upsilon$
under consideration satisfies all the assumptions
of Theorem \ref{xgamma<0} with $p:=\frac{r}{r_\upsilon+\varepsilon}$
therein and $q,\gamma$ as in (i).
Thus, using Theorem \ref{xgamma<0}, we obtain (i).

Next, from the proof of \cite[Theorem 5.14]{zyy23-1},
we infer that the weighted Lebesgue space
$L^r_{\upsilon}$ under consideration satisfies
all the assumptions of Theorem \ref{xgamma<0} with $p:=q:=1$ therein
and $\gamma\in\Gamma_{1,1}$.
Thus, Theorems \ref{c1016}, \ref{2256}, and \ref{2255} respectively
imply (ii), (iii), and (iv),
which then completes the proof of Theorem \ref{1657}.
\end{proof}
\begin{remark}
\begin{enumerate}
  \item[{\rm(i)}] Theorem \ref{1657}(i) when
  $r=1$ and $\gamma\in(-\infty,-q)$ or
  $r\in(1,\infty)$ and $q\in(r,\infty)$ is new
  and, in the other cases, coincides with \cite[Theorem 5.14]{zyy23-1}.
  To the best of our knowledge, Theorem \ref{1657}(ii) is completely new.
  \item[{\rm(ii)}]
  Both (iii) and (iv) of Theorem \ref{1657}
  when $r = 1$, $\gamma\in (-\infty,-1)$, and
  $n\in\mathbb{N}\cap[2,\infty)$ are new and,
  in the other cases, coincide with \cite[Theorems 5.15 and 5.16]{zyy23-1}.
\end{enumerate}
\end{remark}

\subsection{(Bourgain--)Morrey Type Spaces}\label{5.1}

In this subsection, we investigate the
BSVY formulae and Gagliardo--Nirenberg type inequalities
in (Bourgain--)Morrey type
spaces, including Morrey, Bourgain--Morrey, Besov--Bourgain--Morrey,
and Triebel--Lizorkin--Bourgain--Morrey spaces.

First, we give a short introduction
on the history of Bourgain--Morrey type spaces.
Recall that Morrey spaces [see Definition
\ref{df-43}(i) for the definition] were originally
introduced by Morrey \cite{m1938} in 1938
to study the regularity of the solution of
partial differential equations. Nowadays,
these spaces have proved important in the theory of
partial differential equations, potential theory,
and harmonic analysis; we refer to
\cite{hms2017,hms2020,hss2018,hs2017} and the monographs
\cite{a2015,sdh20201,sdh20202,ysy10}.
On the other hand, in order to study
the Bochner--Riesz multiplier
problems in ${\mathbb R}^3$, Bourgain \cite{b91}
introduced a new function space
which is a special case of Bourgain--Morrey spaces.
After that, to explore some problems on
nonlinear Schr\"{o}dinger equations, Masaki \cite{m16-09234}
introduced Bourgain--Morrey spaces for
the full range of exponents [see Definition
\ref{df-43}(ii) for the definition]. Later on,
Bourgain--Morrey spaces play important roles
in the study of some linear and nonlinear
partial differential equations (see,
for example, \cite{BV07,b98,kpv00,ms18,ms18-2,mvv96,mvv99})
and, moreover, their several fundamental real-variable
properties were recently revealed by Hatano
et al. \cite{hnsh22}.
Very recently, via combining the structures
of both Besov spaces (or Triebel--Lizorkin spaces)
and Bourgain--Morrey spaces,
Zhao et al. \cite{zstyy23} and Hu et al. \cite{hly23}
introduced Besov--Bourgain--Morrey spaces
[see Definition
\ref{df-43}(iii) for the definition] and
Triebel--Lizorkin--Bourgain--Morrey spaces
[see Definition
\ref{df-43}(iv) for the definition], respectively.

Recall that the \emph{dyadic cube}
$Q_{j,m}$ of $\mathbb{R}^n$ with $j\in\mathbb{Z}$ and $m\in\mathbb{Z}^n$
is defined by setting
$Q_{j,m}: = 2^j[m+ (0,1]^n ].$
We now present the definitions of aforementioned
Bourgain--Morrey type spaces
as follows (see, for example, \cite{hnsh22,hly23,zstyy23}).

\begin{definition}\label{df-43}
Let $0<t\le u\le r\le\infty$ and $\tau\in(0,\infty]$.
\begin{enumerate}
  \item[{\rm(i)}] The \emph{Morrey space}
  $\mathcal{M}^u_t$
is defined to be the
set of all $f\in L^t_{\loc}$ such that
\begin{equation*}
\|f\|_{\mathcal{M}^u_t}:=
\sup_{j\in\mathbb{Z},\,m\in\mathbb{Z}^n}
\left|Q_{j,m}\right|^{\frac 1u-\frac 1t}
\left\|f{\mathbf{1}}_{Q_{j,m}}
\right\|_{L^{t}}<\infty.
\end{equation*}
  \item[{\rm(ii)}] The \emph{Bourgain--Morrey space}
$\mathcal{M}^u_{t,r}$ is defined to be the
set of all $f\in L^t_{\loc}$ such that
\begin{equation*}
\|f\|_{\mathcal{M}^u_{t,r}}
:=\left\{\sum_{j\in{\mathbb Z},\,m\in{\mathbb Z}^n}
\left[\left|Q_{j,m}\right|^{\frac{1}{u}-\frac{1}{t}}
\left\|f{\mathbf{1}}_{Q_{j,m}}
\right\|_{L^{t}}\right]^r\right\}^{\frac{1}{r}},
\end{equation*}
with the usual modification made when $r=\infty$, is finite.
  \item[{\rm(iii)}] The \emph{Besov--Bourgain--Morrey space
  $\mathcal{M}\dot{B}_{t,r}^{u,\tau}$} is defined to be the set
of all $f\in L^t_{\loc}$ such that
\begin{align*}
\|f\|_{\mathcal{M}\dot{B}_{t,r}^{u,\tau}}:=
\left\{\sum_{j\in{\mathbb Z}}
\left[\sum_{m\in{\mathbb Z}^n}
\left\{\left|Q_{j, m}\right|^{\frac{1}{u}-\frac{1}{t}}
\left\|f{\mathbf{1}}_{Q_{j,m}}\right\|_{L^{t}}
\right\}^r \right]^\frac{\tau}{r}\right\}^\frac{1}{\tau},
\end{align*}
with the usual modifications made when $r=\infty$ or $\tau=\infty$,
is finite.
\item[{\rm(iv)}] The \emph{Triebel--Lizorkin--Bourgain--Morrey space
$\mathcal{M}\dot{F}_{t,r}^{u,\tau}$} is defined to be the set
of all $f\in L^t_{\loc}$ such that
\begin{align*}
\|f\|_{\mathcal{M}\dot{F}_{t,r}^{u,\tau}}
:=\left(\int_{\mathbb{R}^n}
\left\{\int_{0}^{\infty}\left[\xi^{n(\frac1u
-\frac1t-\frac1r)}\left\|f{\bf 1}_{B(y,\xi)}
\right\|_{L^t}\right]^\tau\,
\frac{d\xi}{\xi}\right\}^{\frac r\tau}\,dy\right)^{\frac1r},
\end{align*}
with the usual modifications made when $r=\infty$ or $\tau=\infty$,
is finite.
\end{enumerate}
\end{definition}

\begin{remark}\label{bbm-rm}
Let $0<t\le u\le r\le\infty$ and
$\tau\in(0,\infty]$.
\begin{enumerate}
  \item[{\rm(i)}] It is obvious that
  $\mathcal{M}^{u}_{t,\infty}=\mathcal{M}^u_t$
  and $\mathcal{M}\dot{B}^{u,r}_{t,r}
  =\mathcal{M}^u_{t,r}$.
  Moreover, from \cite[Proposition 3.6(iii)]{hly23},
  we deduce that $\mathcal{M}\dot{F}^{u,r}_{t,r}
  =\mathcal{M}^u_{t,r}$.
  \item[{\rm(ii)}] Applying
\cite[Theorem 2.9]{zstyy23},
we find that $f\in\mathcal{M}\dot{B}^{u,\tau}_{t,r}$
if and only if $f\in L^t_{{\loc}}$
and
\begin{align}\label{bbm-rme1}
\left\|f\right\|_{\mathcal{M}\dot{B}^{u,\tau}
_{t,r}}^\star
:=\left(\int_{0}^{\infty}\left\{\int_{\mathbb{R}^n}
\left[\xi^{n(\frac1u-\frac1t-\frac1r)}
\left\|f{\bf 1}_{B(y,\xi)}\right\|_{L^t}
\right]^r\,dy\right\}^{\frac \tau r}\,\frac{d\xi}{\xi}\right)
^{\frac1\tau},
\end{align}
with the usual modifications made when
$r=\infty$ or $\tau=\infty$,
is finite. Moreover, $\|\cdot\|_{\mathcal{M}\dot{B}
^{u,\tau}_{t,r}}^\star$ is
an equivalent quasi-norm of $\mathcal{M}\dot{B}
^{u,\tau}_{t,r}$.
  \item[{\rm(iii)}] By \cite[p.\,87]{shyy2017},
we conclude that, if $1\leq t\leq u<\infty$,
then the Morrey space $\mathcal{M}_t^u$
is a ball Banach function space, but may not be a Banach function space.
Moreover, as proved in \cite[Lemma 4.10]{zyy23},
if $1\le t<u<r\le\infty$ and $\tau\in[1,\infty)$
or if $1\le t\le u\le r\le \tau=\infty$, then
the Besov--Bourgain--Morrey space
$\mathcal{M}\dot{B}^{u,\tau}_{t,r}$
is also a ball Banach function space. In addition,
applying a similar argument to that used in the proof of
\cite[Lemma 4.10]{zyy23}, we can further show that,
if $1\le t<u<r\le \infty$ and $\tau\in(0,\infty)$
or if $1\le t<u\le r<\tau=\infty$,
then the Triebel--Lizorkin--Bourgain--Morrey space
$\mathcal{M}\dot{F}^{u,\tau}_{t,r}$
is a ball Banach function space.
\end{enumerate}
\end{remark}

We now obtain the following conclusions for
Bourgain--Morrey type spaces via the corresponding
ones for weighted Sobolev spaces in Theorem \ref{1657}.

\begin{theorem}\label{1524}
Let $1\le t<u< r\le\infty$, $\tau\in(0,\infty]$, $A\in\{B,F\}$,
and $f\in \dot{W}^{1,\mathcal{M}\dot{A}^{u,\tau}_{t,r}}$.
\begin{enumerate}
\item[{\rm (i)}] Let $q\in(0,\infty)$ satisfy $n(\frac1t-\frac1q)<1$ and
$\gamma\in\Gamma_{t,q}$. Then \eqref{xgamma<0e1} holds with $X:=\mathcal{M}\dot{A}^{u,\tau}_{t,r}$.
\item[{\rm (ii)}]Let all the symbols be the same as in Theorem \ref{c1016}.
Further assume $f\in F^{s_1}_{\mathcal{M}\dot{A}^{up_1,\tau p_1}_{tp_1,rp_1},p_1}\cap
W^{1,\mathcal{M}\dot{A}^{u,\tau}_{t,r}}$.
If $s_1\in(-\infty,\frac{1}{p_1})\cup(1,\infty)$,
then \eqref{e853} holds with $X:=\mathcal{M}\dot{A}^{u,\tau}_{t,r}$.
\item[{\rm (iii)}]Let all the symbols be the same as in Theorem \ref{2256}.
Then \eqref{939} holds with
$X:=\mathcal{M}\dot{A}^{u,\tau}_{t,r}$.
\item[{\rm (iv)}]Let all the symbols be the same as in Theorem \ref{2255}.
If $p_1\in[1,\infty)$, then \eqref{2053} holds with $X:=\mathcal{M}\dot{A}^{u,\tau}_{t,r}$
and, if $p_1=\infty$, then \eqref{2054} holds with $X:=\mathcal{M}\dot{A}^{u,\tau}_{t,r}$.
\end{enumerate}
\end{theorem}

\begin{remark}
In Theorem \ref{1524},
if $r=\tau=\infty$, then the Bourgain--Morrey type
space $\mathcal{M}\dot{A}^{u,\tau}_{t,r}$
under consideration coincides with the Morrey
space $\mathcal{M}^{u}_t$;
in this case, Theorem \ref{1524}(i) improves
\cite[Theorem 5.1]{zyy23-1} because,
when $t=1$ and $\gamma\in(-\infty,-q)$ or when
$t\in(1,n]$, $\gamma\in(-\infty,0)$,
and $q\in[t,\infty)$, Theorem \ref{1524}(i) is new and,
in the other cases, coincides with \cite[Theorem 5.1]{zyy23-1}.
Moreover, both (iii) and (iv) of Theorem \ref{1524}
respectively improve \cite[Theorems 5.3 and 5.2]{zyy23-1} because,
when $t=1$, $\gamma\in(-\infty,-1)$,
and $n\in\mathbb{N}\cap[2,\infty)$,
(iii) and (iv) of Theorem \ref{1524} are new and, in the other cases,
extend \cite[Theorems 5.3 and 5.2]{zyy23-1} from any $f\in C^1$ with $|\nabla f|\in C_{\rm c}$
to any $f\in M^u_t$. Even in this case, Theorem \ref{1524}(ii) is new.
In addition,
to the best of our knowledge, in the other cases
different from $r=\tau=\infty$,
Theorem \ref{1524} is completely new.
\end{remark}

To prove this theorem, we first
establish the following equivalent characterization
of Bourgain--Morrey type spaces,
which coincides with the well-known characterization
of Morrey spaces when $r=\tau=\infty$;
see, for instance, \cite[Proposition 285]{sdh20201}.

\begin{proposition}\label{equi-bm}
Let $0<t< u< r\le\infty$,
$\tau\in(0,\infty]$,
$A\in\{B,F\}$, and
$\theta\in(0,\min\{1,
t(\frac1u-\frac1r)\})$ and, for any $y\in\mathbb{R}^n$
and $\xi\in(0,\infty)$, let
$\upsilon_{y,\xi}:=[M({\bf 1}
_{B(y,\xi)})]^{1-\theta}$.
Then $f\in\mathcal{M}\dot{A}^{u,\tau}_{t,r}$
if and only if $f\in L^t_{\loc}$
and $[f]_{\mathcal{M}\dot{A}^{u,\tau}_{t,r}}<\infty$,
where
\begin{align}\label{df-mb}
[f]_{\mathcal{M}\dot{B}^{u,\tau}_{t,r}}
:=\left[\int_{0}^{\infty}\left\{\int_{\mathbb{R}^n}
\left[\xi^{n(\frac1u-\frac1t-\frac1r)}
\left\|f\right\|_{L^t_{\upsilon_{y,\xi}}}
\right]^r\,dy\right\}^{\frac \tau r}\,\frac{d\xi}{\xi}\right]
^{\frac1\tau}
\end{align}
and
\begin{align}\label{df-mf}
[f]_{\mathcal{M}\dot{F}^{u,\tau}_{t,r}}
:=\left[\int_{\mathbb{R}^n}
\left\{\int_{0}^{\infty}\left[\xi^{n(\frac1u
-\frac1t-\frac1r)}\left\|f\right\|_{L^t
_{\upsilon_{y,\xi}}}\right]^\tau\,
\frac{d\xi}{\xi}\right\}^{\frac r\tau}\,dy\right]^{\frac1r}
\end{align}
with the usual modifications made when
$r=\infty$ or $\tau=\infty$.
Moreover, $[\cdot]_{\mathcal{M}\dot{A}^{u,\tau}_{t,r}}$
is an equivalent quasi-norm of
$\mathcal{M}\dot{A}^{u,\tau}_{t,r}$ with
the positive equivalence constants depending
only on $n,t,u,r,\tau$, and $\theta$.
\end{proposition}

\begin{proof}
We only consider the case $A=B$ because the proof of the
case $A=F$ is similar and hence we omit the details.
From Remark \ref{bbm-rm}(ii), we infer that
it suffices to show that, for any
$f\in L^t_{{\loc}}$,
\begin{align}\label{equi-bme1}
[f]_{\mathcal{M}\dot{B}^{u,\tau}_{t,r}}
\sim\left\|f\right\|_
{\mathcal{M}\dot{B}^{u,\tau}_{t,r}}^\star,
\end{align}
where $\|\cdot\|_{\mathcal{M}\dot{A}^{u,\tau}_{t,r}}^\star$
is the same as in \eqref{bbm-rme1}.
Indeed, using the fact that, for any
$g\in L^1_{{\loc}}$,
$|g|\le M(g)$ almost everywhere
in $\mathbb{R}^n$,
we immediately obtain, for any $f\in L^t_{{\loc}}$,
$\left\|f\right\|_
{\mathcal{M}\dot{B}^{u,\tau}_{t,r}}^\star
\le [f]_{\mathcal{M}\dot{B}^{u,\tau}_{t,r}}$.

Conversely, we next prove
$[f]_{\mathcal{M}\dot{B}^{u,\tau}_{t,r}}
\lesssim\left\|f\right\|_
{\mathcal{M}\dot{B}^{u,\tau}_{t,r}}^\star$
for any $f\in L^t_{{\loc}}$.
To do so, fix $f\in L^t_{{\loc}}$.
Then, by the fact that, for any
$g\in L^1_{{\loc}}$,
$|g|\le M(g)$ almost everywhere
in $\mathbb{R}^n$ again, we conclude that,
for any $y\in\mathbb{R}^n$ and $\xi\in(0,\infty)$,
\begin{align}\label{equi-bme2}
\|f\|_{L^t_{\upsilon_{y,\xi}}}^t
&=\int_{B(y,2\xi)}|f(x)|^t\left[M
\left({\bf 1}_{B(y,\xi)}\right)(x)\right]^{1-\theta}\,dx
+\sum_{k=2}^{\infty}
\int_{B(y,2^k\xi)\setminus B(y,2^{k-1}\xi)}\cdots\notag\\
&\le\int_{B(y,2\xi)}|f(x)|^t\,dx
+\sum_{k=2}^{\infty}\int_{B(y,2^k\xi)\setminus B(y,2^{k-1}\xi)}
|f(x)|^t\left[M
\left({\bf 1}_{B(y,\xi)}\right)(x)\right]^{1-\theta}\,dx.
\end{align}
Now, fix $y\in\mathbb{R}^n$, $\xi\in(0,\infty)$,
$k\in\mathbb{N}\cap[2,\infty)$, and
$x\in B(y,2^k\xi)\setminus B(y,2^{k-1}\xi)$.
Choose a ball $B:=B(x_B,r_B)$
such that $x\in B$ and $|B(y,\xi)\cap B|\ne 0$,
where $x_B\in\mathbb{R}^n$ and $r_B\in(0,\infty)$.
Then there exists $z\in B(y,\xi)\cap B $ and hence
$2r_B\ge|x-x_B|+|x_B-z|\ge|x-z|\ge|x|-|z|>2^{k-2}\xi.$
This further implies that $|B|\gtrsim(2^k \xi)^n$ and hence
\begin{align*}
M\left({\bf 1}_{B(y,\xi)}\right)(x)
=\sup_{B\ni x}\frac{|B(y,\xi)\cap B|}{|B|}
\lesssim\frac{\xi^n}{(2^k\xi)^n}=2^{-kn},
\end{align*}
where the implicit positive constant depends only
on $n$. From this and \eqref{equi-bme2},
we deduce that, for any $y\in\mathbb{R}^n$ and
$\xi\in(0,\infty)$,
\begin{align*}
\|f\|_{L^t_{\upsilon_{y,\xi}}}^t
\lesssim\left\|f{\bf 1}_{B(y,2\xi)}\right\|_{L^t}^t
+\sum_{k=2}^{\infty}2^{-kn(1-\theta)}
\left\|f{\bf 1}_{B(y,2^k\xi)}\right\|_{L^t}^t.
\end{align*}
Using this, a change of variables, and
the assumption $\theta<t(\frac1u-\frac1r)$
and letting $d:=\min\{1,t,r,\tau\},$
we find that
\begin{align*}
[f]_{\mathcal{M}\dot{B}^{u,\tau}_{t,r}}^d
&\lesssim\sum_{k=1}^{\infty}2^{-\frac{kn(1-\theta)d}{t}}
\left(\int_{0}^{\infty}\left\{\int_{\mathbb{R}^n}
\left[\xi^{n(\frac1u-\frac1t-\frac1r)}
\left\|f{\bf 1}_{B(y,2^k\xi)}\right\|_{L^t}
\right]^r\,dy\right\}^{\frac \tau r}\,\frac{d\xi}{\xi}\right)
^{\frac d\tau}\\
&=\sum_{k=1}^{\infty}2^{-\frac{kn(1-\theta)d}{t}}
2^{-kn(\frac1u-\frac1t-\frac1r)d}
\left(\int_{0}^{\infty}\left\{\int_{\mathbb{R}^n}
\left[\xi^{n(\frac1u-\frac1t-\frac1r)}
\left\|f{\bf 1}_{B(y,\xi)}\right\|_{L^t}
\right]^r\,dy\right\}^{\frac \tau r}\,\frac{d\xi}{\xi}\right)
^{\frac d\tau}\\
&=\left[\|f\|_{\mathcal{M}\dot{B}^{u,\tau}_{t,r}}^\star\right]^d\sum_{k=1}^{\infty}
2^{-kn(\frac1u-\frac1r-\frac\theta t)}\sim
\left[\|f\|_{\mathcal{M}\dot{B}^{u,\tau}_{t,r}}^\star\right]^d.
\end{align*}
This then finishes the proof of \eqref{equi-bme1} and hence
Proposition \ref{equi-bm}.
\end{proof}

Next, we show Theorem \ref{1524}.

\begin{proof}[Proof of Theorem \ref{1524}]
Fix $f\in \dot{W}^{1,\mathcal{M}\dot{A}^{u,\tau}_{t,r}}$.
Then $|\nabla f|\in \mathcal{M}\dot{A}^{u,\tau}_{t,r}$.
We now prove (i).
Choose $\theta\in(0,\min\{1,t(\frac{1}{u}
-\frac1r)\})$
and, for any $y\in\mathbb{R}^n$ and $\xi\in(0,\infty)$,
let $\upsilon_{y,\xi}$ be as in Proposition \ref{equi-bm}.
Then, applying Proposition \ref{equi-bm},
we find that, for any $g\in \mathcal{M}\dot{A}^{u,\tau}_{t,r}$,
\begin{align}\label{1524e2}
\|g\|_{\mathcal{M}\dot{A}^{u,\tau}_{t,r}}
\sim[g]_{\mathcal{M}\dot{A}^{u,\tau}_{t,r}}
\end{align}
with the positive equivalence constants depending
only on $n,t,u,r,\tau$, and $\theta$,
where $[\cdot]_{\mathcal{M}\dot{A}^{u,\tau}_{t,r}}$
is as in \eqref{df-mb} when $A=B$ or as in
\eqref{df-mf} when $A=F$.
From this and $|\nabla f|\in \mathcal{M}\dot{A}^{u,\tau}_{t,r}$,
it follows that, for almost every $y\in\mathbb{R}^n$
and $\xi\in(0,\infty)$,
$\|\,|\nabla f|\,\|_{L^t_{\upsilon_{y,\xi}}}
<\infty$,
which further implies $f\in\dot{W}^{1,t}_{\upsilon_{y,\xi}}$.

We now prove that, for any $y\in\mathbb{R}^n$
and $\xi\in(0,\infty)$,
the weighted Lebesgue space $L^t_{\upsilon_{y,\xi}}$
satisfies all the assumptions of Proposition \ref{upper}
with $p:=t$ therein and $q,\gamma$ as in Theorem \ref{1524}(i).
Indeed, fix $y\in\mathbb{R}^n$ and $\xi\in(0,\infty)$. Then, using
\cite[Theorem 7.7]{D2000}, we conclude that
$\upsilon_{y,\xi}\in A_1$ and
$[\upsilon_{y,\xi}]_{A_1}$
depends only on both $n$ and $\theta$.
Therefore, by the proof of \cite[Theorem 5.14]{zyy23-1},
we find that
\begin{enumerate}
  \item[{\rm(i)}] the weighted Lebesgue space
  $L^t_{\upsilon_{y,\xi}}$ is a ball Banach
  function space having an absolutely continuous
  norm;
  \item[{\rm(ii)}] the Hardy--Littlewood maximal operator
  $M$ is bounded on $[(L^t_{\upsilon_{y,\xi}})^
  \frac1t]'$
  and
  \begin{align}\label{1524e3}
  \left\|M\right\|_{[(L^t_{\upsilon_{y,\xi}})^
  \frac1t]'
  \to[(L^t_{\upsilon_{y,\xi}})^
  \frac1t]'}\lesssim
  [\upsilon_{y,\xi}]_{A_1}\lesssim1,
  \end{align}
  where the implicit positive constants are independent of
  both $y$ and $\xi$.
\end{enumerate}
These further imply that $L^t_{\upsilon_{y,\xi}}$
satisfies all the assumptions of Proposition \ref{upper}
with $p:=t$ therein and $q,\gamma$ as in Theorem \ref{1524}(i).
Combining this and \eqref{1524e3} again,
we obtain
\begin{align*}
\sup_{\lambda\in(0,\infty)}\lambda
\left\|\left[\int_{\mathbb{R}^n}
{\bf 1}_{E_{\lambda,\frac{\gamma}{q}}[f]}(\cdot,y)
|\cdot-y|^{\gamma-n}\,dy\right]^{\frac1q}\right\|_
{L^t_{\upsilon_{y,\xi}}}
\lesssim\left\|\,\left|\nabla f\right|\,\right\|_
{L^t_{\upsilon_{y,\xi}}}
\end{align*}
with the implicit positive constant independent
of $y$ and $\xi$.
From this, \eqref{1524e2}, Remark \ref{bbm-rm}(iii),
and Proposition \ref{lower},
we further deduce (i).

Moreover, applying an argument similar
to that used in the proof of (i), we easily find that
$L^t_{\upsilon_{y,\xi}}$
satisfies all the assumptions of Theorems \ref{c1016},
\ref{2256}, and \ref{2255}.
This, together with \eqref{1524e2} and H\"older's inequality,
further implies the remaining conclusions,
which then completes the proof of Theorem \ref{1524}.
\end{proof}

\subsection{Variable Lebesgue Spaces}\label{5.3}

Let $u:\mathbb{R}^n\to[0,\infty)$ be a
measurable function and let
$u_-:={\mathop\mathrm{\,ess\,inf\,}}_{x\in\mathbb{R}^n}u(x)$ and
$u_+:={\mathop\mathrm{\,ess\,sup\,}}_{x\in\mathbb{R}^n}u(x).$
A function $u:\mathbb{R}^n\to[0,\infty)$ is said to be \emph{globally
log-H\"older continuous} if there exist
$u_{\infty}\in\mathbb{R}$ and a positive constant $C$ such that, for any
$x,y\in\mathbb{R}^n$,
\begin{equation*}
|u(x)-u(y)|\le \frac{C}{\log(e+\frac{1}{|x-y|})}\ \ \text{and}\ \
|u(x)-u_\infty|\le \frac{C}{\log(e+|x|)}.
\end{equation*}
Then the \emph{variable Lebesgue space
$L^{u(\cdot)}$}, associated with the function
$u:\mathbb{R}^n\to[0,\infty)$, is defined to be the set
of all $f\in\mathscr{M}$ such that
\begin{equation*}
\|f\|_{L^{u(\cdot)}}:=\inf\left\{\lambda
\in(0,\infty):\int_{\mathbb{R}^n}\left[\frac{|f(x)|}
{\lambda}\right]^{u(x)}\,dx\le1\right\}
\end{equation*}
is finite.
By Diening et al.
\cite[Lemma 3.2.6 and Theorem 3.2.13]{dhhr2011},
we easily find that, when $u(\cdot):\mathbb{R}^n\to(0,\infty)$,
the variable Lebesgue space $L^{u(\cdot)}$
is a quasi-Banach function space and,
when $1\le u_-\le u_+<\infty$,
$L^{u(\cdot)}$ is a Banach function space
and hence a ball Banach function space (see also
\cite[Subsection 7.8]{shyy2017}).
For more studies about variable Lebesgue spaces,
we refer to
\cite{cf2013,cw2014,dhr2009,kr1991,ns2012,n1950,n1951}.

We establish the following
BSVY formula and Gagliardo--Nirenberg type inequalities in variable Lebesgue spaces.

\begin{theorem}\label{1022}
Let $u:\mathbb{R}^n\to(0,\infty)$ be globally
log-H\"older continuous satisfying
$1\le u_-\le u_+<\infty$ and
$f\in \dot{W}^{1,L^{u(\cdot)}}$.
\begin{enumerate}
\item[{\rm(i)}] Let $q\in(0,\infty)$ satisfy
$n(\frac{1}{u_-}-\frac1q)<1$
and let $\gamma\in\Gamma_{u_-,q}$. Then \eqref{xgamma<0e1} holds
with $X:=L^{u(\cdot)}$.
\item[{\rm (ii)}] Let all the symbols be the same as in Theorem \ref{c1016}.
Further assume $f\in F^{s_1}_{L^{p_1u(\cdot)},p_1}\cap W^{1,L^{u(\cdot)}}$.
If $s_1\in(-\infty,\frac{1}{p_1})\cup(1,\infty)$, then
\eqref{e853} holds with $X:=L^{u(\cdot)}$.
  \item[{\rm(iii)}]Let all the symbols be the same as in Theorem \ref{2256}.
Then \eqref{939} holds with
$X:=L^{u(\cdot)}$.
  \item[{\rm(iv)}] Let all the symbols be the same as in Theorem \ref{2255}.
If $p_1\in[1,\infty)$, then
\eqref{2053} holds with $X:=L^{u(\cdot)}$.
If $p_1=\infty$, then
\eqref{2054} holds with $X:=L^{u(\cdot)}$.
\end{enumerate}
\end{theorem}

\begin{proof}
Applying the proof of \cite[Theorem 5.10]{zyy23-1},
we conclude that all the assumptions of Theorems \ref{xgamma<0},
\ref{c1016}, \ref{2256}, and \ref{2255}
hold for the variable Lebesgue space
$L^{u(\cdot)}$ under consideration.
By this, we further find that, for any
$f\in\dot{W}^{1,L^{u(\cdot)}}$,
Theorems \ref{xgamma<0}, \ref{c1016},
\ref{2256}, and \ref{2255}
hold with $X:=L^{u(\cdot)}$.
This then finishes the proof of Theorem \ref{1022}.
\end{proof}

\begin{remark}
\begin{enumerate}
  \item[{\rm(i)}] Theorem \ref{1022}(i) when
  $u_-=1$ and $\gamma\in(-\infty,-q)$ or when
  $u_-\in(1,n]$ and $q\in[u_-,\infty)$ is new
  and, in the other cases, coincides with \cite[Theorem 5.10]{zyy23-1}.
  \item[{\rm(ii)}] Both (iii) and (iv) of
  Theorem \ref{1022}
  when $u_-= 1$, $\gamma\in (-\infty,-1)$, and
  $n\in\mathbb{N}\cap[2,\infty)$ are new and,
  in the other cases, coincide, respectively,
  with \cite[Theorems 5.12 and 5.11]{zyy23-1}.
  To the best of our knowledge, Theorem \ref{1022}(ii) is completely new.
\end{enumerate}
\end{remark}

\subsection{Orlicz and Orlicz-Slice Spaces}\label{5.5}

Recall that a nondecreasing function $\Phi:[0,\infty)
\to[0,\infty)$ is called an \emph{Orlicz function}
if $\Phi$ satisfies that
\begin{enumerate}
\item[(i)]
$\Phi(0)= 0$;
\item[(ii)]
for any $t\in(0,\infty)$,
$\Phi(t)\in(0,\infty)$;
\item[(iii)]
$\lim_{t\to\infty}\Phi(t)=\infty$.
\end{enumerate}
An Orlicz function $\Phi$ is
said to be of \emph{lower}
(resp. \emph{upper}) \emph{type} $u$ for some
$u\in\mathbb{R}$ if there exists a positive constant
$C_{(u)}$ such that,
for any $t\in[0,\infty)$ and
$s\in(0,1)$ [resp. $s\in[1,\infty)$],
$\Phi(st)\le C_{(p)} s^u\Phi(t).$
In what follows, we always assume that
$\Phi:[0,\infty)\to [0,\infty)$
is an Orlicz function with both positive lower
type $u_{\Phi}^-$ and positive upper type $u_{\Phi}^+$.
Then the \emph{Orlicz space $L^\Phi$}
is defined to be the set of all $f\in\mathscr{M}$
such that
\begin{equation*}
\|f\|_{L^\Phi}:=\inf\left\{\lambda\in
(0,\infty):\int_{\mathbb{R}^n}\Phi\left(\frac{|f(x)|}
{\lambda}\right)\,dx\le1\right\}
\end{equation*}
is finite.
By the definition, we can easily show
that $L^\Phi$
is a quasi-Banach function space
(see \cite[Section~7.6]{shyy2017}).
For more studies about Orlicz spaces,
we refer to \cite{dfmn2021,ns2014,rr2002}.

Using Theorems \ref{xgamma<0}, \ref{c1016},
\ref{2256}, and \ref{2255},
we obtain the following BSVY formula
and Gagliardo--Nirenberg type
inequalities in Orlicz spaces.

\begin{theorem}\label{1912}
Let $\Phi$ be an Orlicz function with both
positive lower type $u^-_{\Phi}$
and positive upper type $u^+_\Phi$.
Let $1\leq u^-_{\Phi}\leq u^+_{\Phi}<\infty$
and $f\in \dot{W}^{1,L^\Phi}$.
\begin{enumerate}
\item[{\rm(i)}] Let $q\in(0,\infty)$ satisfy
$n(\frac{1}{u^-_\Phi}-\frac1q)<1$
and let $\gamma\in \Gamma_{u^-_\Phi,q}$. Then \eqref{xgamma<0e1} holds
with $X:=L^{\Phi}$.
\item[{\rm(ii)}] Let all the symbols be the same as in Theorem \ref{c1016}.
Further assume $f\in F^{s_1}_{L^{\Phi_{p_1}},p_1}\cap W^{1,L^{\Phi}}$,
where $\Phi_{p_1}(t):=\Phi(t^{p_1})$ for any $t\in[0,\infty)$.
If $s_1\in(-\infty,\frac{1}{p_1})\cup(1,\infty)$, then
\eqref{e853} holds with $X:=L^{\Phi}$.
\item[{\rm(iii)}]Let all the symbols be the same as in Theorem \ref{2256}.
Then \eqref{939} holds with $X:=L^{\Phi}$.
\item[{\rm(iv)}]Let all the symbols be the same as in Theorem \ref{2255}.
If $p_1\in[1,\infty)$, then
\eqref{2053} holds with $X:=L^{\Phi}$.
If $p_1=\infty$, then
\eqref{2054} holds with $X:=L^{\Phi}$.
\end{enumerate}
\end{theorem}

\begin{proof}
From the proof of \cite[Theorem 5.23]{zyy23-1},
it follows that the Orlicz space
$L^{\Phi}$ under consideration satisfies
all the assumptions of Theorems \ref{xgamma<0}, \ref{c1016},
\ref{2256}, and \ref{2255}.
This then further implies that Theorems \ref{xgamma<0},
\ref{c1016}, \ref{2256}, and \ref{2255} hold with
$X:=L^{\Phi}$,
which completes the proof of Theorem \ref{1912}.
\end{proof}

\begin{remark}
\begin{enumerate}
  \item[{\rm(i)}] Theorem \ref{1912}(i) when
  $u^-_\Phi=1$ and $\gamma\in(-\infty,-q)$ or when
  $u^-_\Phi\in(1,n]$ and $q\in[u^-_\Phi,\infty)$ is new
  and, in the other cases, coincides with \cite[Theorem 5.23]{zyy23-1}.
  \item[{\rm(ii)}] Both (iii) and (iv) of
  Theorem \ref{1912}
  when $u^-_\Phi= 1$, $\gamma\in (-\infty,-1)$, and
  $n\in\mathbb{N}\cap[2,\infty)$ are new and,
  in the other cases, coincide, respectively,
  with \cite[Theorems 5.25 and 5.24]{zyy23-1}.
  To the best of our knowledge, Theorem \ref{1912}(ii) is completely new.
\end{enumerate}
\end{remark}

Moreover, for any given Orlicz function $\Phi$ and
$t,r\in(0,\infty)$,
the \emph{Orlicz-slice space}
$(E_\Phi^r)_t$ is defined to be the set of all
$f\in\mathscr{M}$ such that
\begin{equation*}
\|f\|_{(E_\Phi^r)_t} :=\left\{\int_{\mathbb{R}^n}
\left[\frac{\|f\mathbf{1}_{B(x,t)}\|_{L^\Phi}}
{\|\mathbf{1}_{B(x,t)}\|_{L^\Phi}}\right]
^r\,dx\right\}^{\frac{1}{r}}<\infty.
\end{equation*}
The Orlicz-slice spaces were first introduced in
\cite{zyyw2019} as a generalization of both
the slice space of Auscher and Mourgoglou
\cite{am2019,ap2017} and the Wiener amalgam space
in \cite{h2019,h1975,kntyy2007}. From both
\cite[Lemma 2.28]{zyyw2019} and \cite[Remark 7.41(i)]{zwyy2021},
we deduce that the Orlicz-slice space $(E_\Phi^r)_t$ is a
ball Banach function space, but in general is not a
Banach function space.

Applying Theorems \ref{xgamma<0}, \ref{c1016},
\ref{2256}, and \ref{2255},
we find the following conclusions.

\begin{theorem}\label{2045}
Let $t\in(0,\infty)$, $r\in[1,\infty)$,
and $\Phi$ be an Orlicz function with
both positive lower type $u^-_{\Phi}$
and positive upper type $u^+_\Phi$.
Let $1\leq u^-_{\Phi}\leq u^+_{\Phi}<\infty$
and $f\in \dot{W}^{1,(E_\Phi^r)_t}$.
\begin{enumerate}
  \item[{\rm(i)}] Let $q\in(0,\infty)$ satisfy
  $n(\frac{1}{\min\{r,u^-_\Phi\}}-\frac1q)<1$ and let
  $\gamma\in\Gamma_{\min\{r,u^-_\Phi\},q}$. Then \eqref{xgamma<0e1} holds
with $X:=(E_\Phi^r)_t$.
\item[{\rm(ii)}] Let all the symbols be the same as in Theorem \ref{c1016}.
Further assume $f\in F^{s_1}_{(E_{\Phi_{p_1}}^r)_t,p_1}\cap W^{1,(E_\Phi^r)_t}$.
If $s_1\in(-\infty,\frac{1}{p_1})\cup(1,\infty)$, then
\eqref{e853} holds with $X:=(E_\Phi^r)_t$.
\item[{\rm(iii)}]Let all the symbols be the same as in Theorem \ref{2256}.
Then \eqref{939} holds with $X:=(E_\Phi^r)_t$.
\item[{\rm(iv)}]Let all the symbols be the same as in Theorem \ref{2255}.
If $p_1\in[1,\infty)$, then
\eqref{2053} holds with $X:=(E_\Phi^r)_t$.
If $p_1=\infty$, then
\eqref{2054} holds with $X:=(E_\Phi^r)_t$.
\end{enumerate}
\end{theorem}

\begin{proof}
By the proof of \cite[Theorem 5.28]{zyy23-1},
we conclude that the Orlicz-slice space
$(E_\Phi^r)_t$ under consideration
satisfies all the assumptions of Theorems \ref{xgamma<0}, \ref{c1016},
\ref{2256}, and \ref{2255}.
This further implies that
Theorems \ref{xgamma<0}, \ref{c1016},
\ref{2256}, and \ref{2255} with
$X:=(E_\Phi^r)_t$ hold,
which then completes the proof of Theorem \ref{2045}.
\end{proof}

\begin{remark}
\begin{enumerate}
  \item[{\rm(i)}] We point out that Theorem \ref{2045}(i) when
  $\min\{r,u^-_\Phi\}=1$ and $\gamma\in(-\infty,-q)$ or when
  $\min\{r,u^-_\Phi\}\in(1,n]$ and $q\in[\min
  \{r,u^-_\Phi\},\infty)$ is new
  and, in the other cases, coincides with \cite[Theorem 5.28]{zyy23-1}.
  \item[{\rm(ii)}] Both (iii) and (iv) of
  Theorem \ref{2045}
  when $\min\{r,u^-_\Phi\}= 1$, $\gamma\in (-\infty,-1)$, and
  $n\in\mathbb{N}\cap[2,\infty)$ are new and,
  in the other cases, coincide, respectively,
  with \cite[Theorems 5.30 and 5.29]{zyy23-1}.
  To the best of our knowledge, Theorem \ref{2045}(ii) is completely new.
\end{enumerate}
\end{remark}

\subsection{Local and Global Generalized Herz Spaces}\label{5.2}

The main target of this subsection is to
show the BSVY formula as well as Gagliardo--Nirenberg type inequalities
in local and global generalized Herz spaces.
Recall that the classical Herz space was originally
introduced by Herz \cite{herz} to
study Bernstein's theorem on absolutely
convergent Fourier transforms.
Recently, Rafeiro and Samko \cite{rs2020}
introduced the local and global generalized Herz spaces
(see Definition \ref{gh})
which generalize the classical
Herz spaces and generalized Morrey type spaces.
For more studies on Herz spaces,
we refer to \cite{gly1998,hy1999,ly1996,
lyh2320,rs2020,zyz2022}.

Let $\mathbb{R}_+:=(0,\infty)$
and $\omega$ be a nonnegative function on $\mathbb{R}_+$.
Then the function $\omega$ is said to be
\emph{almost increasing}
(resp. \emph{almost decreasing})
on $\mathbb{R}_+$ if there exists
a constant $C\in[1,\infty)$ such that,
for any $t,\tau\in\mathbb{R}_+$ satisfying
$t\leq\tau$ (resp. $t\geq\tau$),
$\omega(t)\leq C\omega(\tau).$

\begin{definition}
The \emph{function class} $M(\mathbb{R}_+)$
is defined to be the set of all positive functions
$\omega$ on $\mathbb{R}_+$ such that, for any $0<\delta<N<\infty$,
$0<\inf_{t\in(\delta,N)}\omega(t)\leq\sup_
{t\in(\delta,N)}\omega(t)<\infty$ and
there exist four constants $\alpha_{0},
\beta_{0},\alpha_{\infty},\beta_{\infty}\in\mathbb{R}$ such that
\begin{enumerate}
  \item[(i)] for any $t\in(0,1]$,
  $\omega(t)t^{-\alpha_{0}}$ is almost increasing and
  $\omega(t)t^{-\beta_{0}}$ is almost decreasing;
  \item[(ii)] for any $t\in[1,\infty)$,
  $\omega(t)t^{-\alpha_{\infty}}$ is
  almost increasing and  $\omega(t)t^{-\beta_{\infty}}$ is
  almost decreasing.
\end{enumerate}
\end{definition}

We now present the Matuszewska--Orlicz indices
as follows, which were introduced in
\cite{MO,mo65} and characterize the properties of functions at
origin or infinity (see also \cite{lyh2320}).

\begin{definition}
Let $\omega$ be a positive function on $\mathbb{R}_+$. Then
the \emph{Matuszewska--Orlicz indices}
$m_0(\omega)$, $M_0(\omega)$,
$m_\infty(\omega)$, and $M_\infty(\omega)$ of
$\omega$ are defined, respectively, by setting
$$m_0(\omega):=\sup_{t\in(0,1)}
\frac{\log\left[\limsup\limits_{h\to0^+}\frac{\omega(ht)}
{\omega(h)}\right]}{\log t},\
M_0(\omega):=\inf_{t\in(0,1)}\frac{\log\left[\liminf\limits_{h\to0^+}
\frac{\omega(ht)}{\omega(h)}\right]}{\log t},$$
$$
m_{\infty}(\omega):=\sup_{t\in(1,\infty)}
\frac{\log\left[\liminf\limits_{h\to\infty}
\frac{\omega(ht)}{\omega(h)}\right]}{\log t},\
\mathrm{and}\ M_\infty(\omega)
:=\inf_{t\in(1,\infty)}\frac{\log\left[\limsup\limits_
{h\to\infty}\frac{\omega(ht)}{\omega(h)}\right]}{\log t}.
$$
\end{definition}

The following concept of generalized Herz spaces
were originally introduced by
Rafeiro and Samko in \cite[Definition 2.2]{rs2020}
(see also \cite{lyh2320}).

\begin{definition}\label{gh}
Let $u,r\in(0,\infty]$ and $\omega\in M(\mathbb{R}_+)$.
\begin{enumerate}
\item[{\rm(i)}] Let $\xi\in\mathbb{R}^n$.
The \emph{local generalized Herz space}
$\dot{\mathcal{K}}^{u,r}_{\omega,\xi}$
is defined to be the set of all
locally integrable functions
$f$ on $\mathbb{R}^n\setminus\{\xi\}$ such that
\begin{equation*}
\|f\|_{\dot{\mathcal{K}}^{u,r}_{\omega,\xi}}:=\left\{\sum_{k\in\mathbb{Z}}
\left[\omega\left(2^{k}\right)\right]^{r}
\left\|f{\bf 1}_{B({\bf 0},2^{k})\setminus B({\bf 0},2^{k-1})}
\right\|_{L^{u}}^{r}\right\}^{\frac{1}{r}}<\infty.
\end{equation*}
\item[{\rm(ii)}] The
\emph{global generalized
Herz space} $\dot{\mathcal{K}}^{u,r}_{\omega}$
is defined to be the set of all
$f\in L^u_{\loc}$ such that
$\|f\|_{\dot{\mathcal{K}}^{u,r}_{\omega}}
:=\sup_{\xi\in\mathbb{R}^n}
\|f\|_{\dot{\mathcal{K}}^{u,r}_{\omega,\xi}}<\infty$.
\end{enumerate}
\end{definition}

\begin{remark}
As showed in \cite[Theorems 1.2.46 and 1.2.48]{lyh2320},
if $\xi\in\mathbb{R}^n$,
$u,r\in[1,\infty]$, and $\omega\in M(\mathbb{R}_+)$
satisfies $-\frac{n}{u}<m_0(\omega)\leq M_0(\omega)<\frac{n}{u'}$,
then the local generalized Herz space
$\dot{\mathcal{K}}^{u,r}_{\omega,\xi}$
is a ball Banach function space;
if $u,r\in[1,\infty]$ and
$\omega\in M(\mathbb{R}_+)$ satisfies both
$m_0(\omega)\in(-\frac{n}{u},\infty)$
and $M_\infty(\omega)\in(-\infty,0)$,
then the global generalized Herz space
$\dot{\mathcal{K}}^{u,r}_{\omega}$ is
a ball Banach function space.
Moreover, using \cite[Remark 4.15]{cjy23},
we find that these Herz spaces may not be Banach function spaces.
\end{remark}

The following results give the BSVY formulae
and Gagliardo--Nirenberg type inequalities
in local and global generalized Herz spaces.

\begin{theorem}\label{herz}
Let $u,r\in[1,\infty)$ and $\omega\in M(\mathbb{R}_+)$ satisfy
\begin{align*}
-\frac nu<m_0(\omega)\le M_0(\omega)<\frac n{u'}
\ and \
-\frac nu<m_\infty(\omega)\le M_\infty(\omega)<\frac n{u'}.
\end{align*}
\begin{enumerate}
\item[{\rm(i)}] Let $q\in(0,\infty)$ satisfy
$n(\frac1u-\frac1q)<1$ and $\gamma\in\Gamma_{u,q}$.
Then, for any $\xi\in\mathbb{R}^n$
and $f\in \dot{W}^{1,\dot{\mathcal{K}}^{u,r}_{\omega,\xi}}$,
\eqref{xgamma<0e1} holds with
$X:=\dot{\mathcal{K}}^{u,r}_{\omega,\xi}$; moreover, for any
$f\in \dot{W}^{1,\dot{\mathcal{K}}^{u,r}_{\omega}}$,
\eqref{xgamma<0e1} holds with
$X:=\dot{\mathcal{K}}^{u,r}_{\omega}$.
\item[{\rm(ii)}] Let all the symbols be the same as in Theorem \ref{c1016}.
If $s_1\in(-\infty,\frac{1}{p_1})\cup(1,\infty)$,
then, for any $\xi\in\mathbb{R}^n$
and $f\in F^{s_1}_{\dot{\mathcal{K}}^{up_1,rp_1}_{\omega^{p_1},\xi},p_1}\cap
W^{1,\dot{\mathcal{K}}^{u,r}_{\omega,\xi}}$,
\eqref{e853} holds with $X:=\dot{\mathcal{K}}^{u,r}_{\omega,\xi}$;
moreover, for any $f\in F^{s_1}_{\dot{\mathcal{K}}^{up_1,rp_1}_{\omega^{p_1}},p_1}\cap
W^{1,\dot{\mathcal{K}}^{u,r}_{\omega}}$,
\eqref{e853} holds with $X:=\dot{\mathcal{K}}^{u,r}_{\omega}$.
\item[{\rm (iii)}]Let all the symbols be the same as in Theorem \ref{2256},
$\xi\in\mathbb{R}^n$, $X\in\{\dot{\mathcal{K}}^{u,r}_{\omega,\xi},
\dot{\mathcal{K}}^{u,r}_{\omega}\}$, and $f\in \dot{W}^{1,X}$.
Then \eqref{939} holds.
\item[{\rm (iv)}]Let all the symbols be the same as in Theorem \ref{2255},
$\xi\in\mathbb{R}^n$, $X\in\{\dot{\mathcal{K}}^{u,r}_{\omega,\xi},
\dot{\mathcal{K}}^{u,r}_{\omega}\}$, and $f\in \dot{W}^{1,X}$.
If $p_1\in[1,\infty)$, then
\eqref{2053} holds.
If $p_1=\infty$, then \eqref{2054} holds.
\end{enumerate}
\end{theorem}

\begin{proof}
We first show (i). Fix $\xi\in\mathbb{R}^n$.
Then, by the proof of \cite[Theorem 4.15]{zyy23},
we conclude that the Herz space
$\dot{\mathcal{K}}^{u,r}_{\omega,\xi}$
under consideration satisfies
all the assumptions of Theorem \ref{xgamma<0} with
$p:=u$ therein and $q,\gamma$ as in (i).
This further implies that, for any
$f\in\dot{W}^{1,\dot{\mathcal{K}}^{u,r}_{\omega,\xi}}$,
\eqref{xgamma<0e1} holds with
$X:=\dot{\mathcal{K}}^{u,r}_{\omega,\xi}$.

Next, we choose
$f\in \dot{W}^{1,\dot{\mathcal{K}}^{u,r}_{\omega}}$.
Then, from Definition \ref{gh},
we infer that, for any $\xi\in\mathbb{R}^n$,
$\|\,|\nabla f|\,\|_
{\dot{\mathcal{K}}^{u,r}_{\omega,\xi}}
\le\|\,|\nabla f|\,\|_
{\dot{\mathcal{K}}^{u,r}_{\omega}}<\infty$,
which further implies $f\in\dot{W}^{1,\dot{\mathcal{K}}^{u,r}_{\omega,\xi}}$.
Combining this and (i), we find that, for any $\xi\in\mathbb{R}^n$,
\eqref{xgamma<0e1} holds with $X:=\dot{\mathcal{K}}^{u,r}_{\omega,\xi}$.
On the other hand, using
\cite[Theorems 1.5.1 and 1.7.3 and Lemma 1.7.7]{lyh2320},
we conclude that, for any $\xi\in\mathbb{R}^n$,
the operator norm $\|M\|_{[(\dot{\mathcal{K}}^{u,r}_{\omega,\xi})^{\frac1u}]'
\to [(\dot{\mathcal{K}}^{u,r}_{\omega,\xi})^{\frac1u}]'}$
of the Hardy--Littlewood maximal operator
is independent of $\xi$.
This, together with \eqref{xgamma<0e1},
further implies that, for any $\xi\in\mathbb{R}^n$,
\begin{align*}
\sup_{\lambda\in(0,\infty)}\lambda
\left\|\left[\int_{\mathbb{R}^n}
{\bf 1}_{E_{\lambda,\frac{\gamma}{q}}[f]}(\cdot,y)
|\cdot-y|^{\gamma-n}\,dy\right]^{\frac1q}\right\|_
{\dot{\mathcal{K}}^{u,r}_{\omega,\xi}}
\sim\left\|\,\left|\nabla f\right|\,\right\|_
{\dot{\mathcal{K}}^{u,r}_{\omega,\xi}}
\end{align*}
with the positive equivalence constants independent
of both $f$ and $\xi$.
From this and Definition \ref{gh} again,
we further deduce that \eqref{xgamma<0e1} holds with
$X:=\dot{\mathcal{K}}^{u,r}_{\omega}$, which completes
the proof of (i).

Finally, applying an argument similar
to that used in the proof of (i)
with Theorem \ref{xgamma<0} with
$p:=u$ therein and $q,\gamma$ as in (i)
replaced by Theorem \ref{xgamma<0} with
$p:=q:=1$ and $\gamma\in\Gamma_{1,1}$,
we conclude (ii), (iii), and (iv) and hence
finish the proof of Theorem \ref{herz}.
\end{proof}

\begin{remark}
To the best of our knowledge,
Theorem \ref{herz} is completely new.
\end{remark}

\noindent\textbf{Data Availability}\quad The manuscript has no associated data.

\section*{Declarations}

\noindent\textbf{Conflict of interest}\quad On behalf of all authors, 
the corresponding author states that there is no conflict of
interest.

\bigskip

\noindent Yinqin Li,
Dachun Yang (Corresponding author), Wen Yuan, Yangyang Zhang and Yirui Zhao

\medskip

\noindent  Laboratory of Mathematics and Complex Systems
(Ministry of Education of China),
School of Mathematical Sciences, Beijing Normal University,
Beijing 100875, The People's Republic of China

\smallskip

\noindent{\it E-mails:} \texttt{yinqli@mail.bnu.edu.cn} (Y. Li)

\noindent\phantom{\it E-mails:} \texttt{dcyang@bnu.edu.cn} (D. Yang)

\noindent\phantom{\it E-mails:} \texttt{wenyuan@bnu.edu.cn} (W. Yuan)

\noindent\phantom{\it E-mails:} \texttt{yangyzhang@bnu.edu.cn} (Y. Zhang)

\noindent\phantom{\it E-mails:} \texttt{yiruizhao@mail.bnu.edu.cn} (Y. Zhao)

\end{document}